\documentclass[11pt]{amsart}
\usepackage{latexsym, amsmath, amssymb,amsthm,amsopn,amsfonts}
\usepackage{version}
\usepackage{mathtools, epsfig,graphics,color,graphicx,graphpap}
\usepackage{amssymb}
\usepackage{todonotes}
\usepackage{listings}
\usepackage{mathrsfs}
\usepackage{cancel}
\usepackage[normalem]{ulem}
\usepackage[citecolor=blue]{hyperref}

\usepackage[framemethod=tikz]{mdframed}
\newcommand{\redsout}{\bgroup\markoverwith{\textcolor{red}{\rule[0.5ex]{2pt}{.4pt}}}\ULon}

\newcommand{\LV}{\left|}
\newcommand{\RV}{\right|}
\newcommand{\LC}{\left(}
\newcommand{\RC}{\right)}

\newcommand{\p}{\partial}

\numberwithin{equation}{section}
\newtheorem{theorem}{Theorem}[section]
\newtheorem{corollary}[theorem]{Corollary}
\newtheorem{proposition}[theorem]{Proposition}
\newtheorem{lemma}[theorem]{Lemma}
\newtheorem{definition}{Definition}[section]

\newtheorem{remark}{Remark}[section]
 
\newcommand{\R}{\mathbb R}

\usepackage{url}
\definecolor{mycolor}{rgb}{0.122, 0.435, 0.698}
\definecolor{aliceblue}{rgb}{0.94, 0.97, 1.0}

\newmdenv[innerlinewidth=0.5pt, roundcorner=4pt,linecolor=mycolor,innerleftmargin=6pt,
innerrightmargin=6pt,innertopmargin=6pt,innerbottommargin=6pt]{mybox}
\newmdenv[backgroundcolor=aliceblue,innerlinewidth=0.5pt, roundcorner=4pt,linecolor=mycolor,innerleftmargin=6pt,
innerrightmargin=6pt,innertopmargin=6pt,innerbottommargin=6pt]{mybox1}
\author[Lai]{Ru-Yu Lai}
\address{School of Mathematics, University of Minnesota, Minneapolis, MN 55455, USA}
\curraddr{}
\email{rylai@umn.edu}

\author[Zhou]{Hanming Zhou}
\address{Department of Mathematics, University of California Santa Barbara, Santa Barbara, CA 93106, USA}
\email{hzhou@math.ucsb.edu}

\thanks{\textbf{Key words}: Nonlinearity,  Inverse problems, BGK model, in-flow boundary condition}
\title[The BGK model with in-flow boundary condition: Forward and inverse problems]{The BGK model with in-flow boundary condition: Forward and inverse problems}  
 
\begin{document} 

\begin{abstract}
We study the Bhatnagar-Gross-Krook (BGK) equation in a bounded domain with in-flow boundary condition. The BGK model is a simple relaxation of the Boltzmann equation through replacing the quadratic nonlinearity by the so-called local Maxwellian, which consists of the density, the velocity, and the temperature.

The aim of this paper is twofold. First, we establish the Dirichlet boundary value problem for the BGK equation by showing a local existence result in the weighted $L^\infty$ norm if the initial data and boundary condition are close to the global Maxwellian. For the purpose of studying the inverse problem, we further derive an
expansion of the BGK solution with respect to a small parameter.
This allows us to decompose the nonlinear BGK into purely linear transport equations based on various order of the expansion.

Second, we investigate an inverse problem of determining a general collision frequency, depending on both density and temperature, in the BGK equation from the boundary measurement operator, which maps from the in-flow boundary data to the out-flow data. We apply highly concentrated in-flow test data to the solution of the resulting linearized BGK equations to extract hidden information of the collision frequency.   
\end{abstract}

\maketitle

\tableofcontents

\section{Introduction}
The dynamic of a gas system is described by the Boltzmann equation
\begin{align}\label{EQN:Boltzmann intro}
	\p_t F+v\cdot\nabla_x F = Q(F,F),
\end{align} 
where $F(t,x,v)$ denotes the distribution function representing density on the phase space point $(x,v)$ in $\Omega\times\R^3$ at time $t\geq 0$. Here $\Omega$ is a bounded domain in $\R^3$ with smooth boundary.
The quadratic collision operator, which models the rate of change of $F$ due to binary interactions,  
takes the form
$$
    Q(F_1,F_2) = \int_{\R^3}\int_{\mathbb{S}^{2}}K[F_1 (t,x,u')F _2(t,x,v')-F_1 (t,x,u)F_2 (t,x,v)]\,d\omega du,
$$
where $u$ and $v$ are velocities before a collision of particles with post-collision velocities  
\begin{equation}\label{incoming and outgoing velocities}
u' = u-[(u-v)\cdot \omega]\omega\quad\hbox{ and }\quad v' = v+[(u-v)\cdot \omega]\omega
\end{equation}
at an angle $\omega\in \mathbb{S}^{2}$. The function $K\equiv K(v,u,\omega)$ is called the collision cross section.  
The subtle interaction between the linear transport operator and the nonlinear collision operator makes important physical and mathematical features of the Boltzmann equation. 
From the computational point of view, however, the Boltzmann collision operator is very expensive to compute because of its nonlinear integral.  
A simplified model due to Bhatnagar-Gross-Krook \cite{BGK1954} is introduced:
\begin{align}\label{EQN:BGK intro}
 \p_t F+v\cdot\nabla_x F = q(M(F) -F),
\end{align}
which replaces Boltzmann binary collision term with the BGK collision operator on the right-hand side of \eqref{EQN:BGK intro} containing collision frequency $q$ and a local Maxwellian distribution $M(F)$, defined by
\begin{align}\label{DEF:maxwellian}
M(F) =  { \rho \over ( 2\pi T)^{3/2}}  \exp\LC{-|v-U|^2\over 2 T}\RC.
\end{align}
Here $\rho(t,x)$, $U(t,x)$, and $T(t,x)$ respectively are the mass density, velocity, and temperature at time $t$ and at a position $x$, and they are implicitly defined through the moments of $F$:  
\begin{align}\label{DEF:density_velocity_temp}
    \left(\begin{array}{c}
    \rho \\
    \rho U\\
    \rho|U|^2+3\rho T 
\end{array}\right)(t,x) = \int_{\R^3}\left(\begin{array}{c}
    1\\
    v\\
    |v|^2
\end{array}\right)F(t,x,v)\, dv,\quad t\geq 0,\quad x\in  \Omega ,
\end{align}
which is equivalent to 
\begin{align*} 
    \rho=\int_{\R^3} F\,dv,\quad U= \rho^{-1}\int_{\R^3} v F\,dv,\quad \hbox{and}\quad T= (3\rho)^{-1} \int_{\R^3} |v-U|^2 F\,dv.
\end{align*}  

This BGK model reduces significantly the computational cost  while preserving many basic properties of the Boltzmann equation. In fact, applying the definitions \eqref{DEF:maxwellian} and \eqref{DEF:density_velocity_temp} yields    
\begin{align}\label{ID:conservation}
		  \int_{\R^3}  q (M(F)-F)\left(\begin{array}{c}
			1\\
			v\\
			|v|^2 \\
		\end{array}\right) \,dv=0,
\end{align}
which ensures the conservation of mass, momentum, and energy.

The main difficulty of the BGK model comes from the nonlinear term, $q (M(F)-F)$, which is worse than the quadratic one of the Boltzmann collision $Q(F)$. 
For the constant setting (i.e., $q\equiv 1$), a global existence result of a weak solution in $\R^n$ was obtained by \cite{Perthame89} under the assumption that global mass, energy and entropy of the initial data are bounded. The unique existence of solutions in a weighted $L^\infty$ space in the three-dimensional torus was investigated by \cite{PerthameARMA}, whose uniqueness result was later extended to the whole space in \cite{Mischler96}.
Also, a global existence of the solution to the BGK in a bounded domain with Dirichlet boundary condition was proved in \cite{PerthameDinh}. 
As for the non-constant collision frequency, a global existence of the solution near a global Maxwellian regime was studied for the case $q =q_1(\rho) q_2(T)$ in $\R^3$ \cite{Bellouquid}. 
When $q =\rho^\alpha T^\beta$, the Cauchy problem for the BGK on three-dimensional torus $\R^3/\mathbb{Z}^3$ was shown in \cite{Yun10}   under the assumption that the energy is sufficiently small. With diffusive reflection boundary condition, a weighted $L^\infty$ solution to the BGK was shown in \cite{CKP2026}.

\subsection{Problem setup of the paper: BGK model with non-constant collision frequency}
 
In this paper, we consider the BGK model in a bounded domain with a collision frequency taking the following general form:  
$$
q(t,x,F)=\gamma(x)(\rho^{\alpha}T^{\beta})(t,x),
$$
where $\gamma$ is a function of $x$, and $\alpha,\,\beta$ are constants in $\R$.  It is clear that this $q$ depends nonlinearly on the solution $F$ when $\beta\neq 0$. 

Suppose that the domain $\Omega$ is convex and bounded with smooth boundary. For $t^*>0$, we consider the initial-boundary value problem for the following BGK model:
\begin{align}\label{EQN:BGK}
    \left\{
    \begin{array}{ll}
    \p_t F+v\cdot\nabla_x F = q(M(F) -F), & \quad (t,x,v)\in (0,t^*)\times\Omega\times\R^3,\\
    F(0,x,v)= F_{in}(x,v), & \quad ( x,v)\in  \Omega\times\R^3,\\
    F(t,x,v)= F_-(t,x,v),&\quad (t,x,v)\in (0,t^*)\times S_-,
    \end{array}
    \right. 
\end{align}
where the outgoing and incoming boundaries are denoted by
\begin{align*}
	S_\pm := \{(x,v)\in \p\Omega\times\R^3:\,  \pm  \nu(x)\cdot v > 0\},
\end{align*}
and $\nu(x)$ is the unit outward normal vector at $x\in\p \Omega$.

The purpose of this paper is twofold: to study well-posedness of the BGK and to recover the collision frequency. In Section~\ref{sec:intro Forward}, we will state our first main result on unique existence of the nonlinear BGK solution, which is considered as a perturbation of the equilibrium.  
For the study of inverse problems, the measurement we utilize is the albedo operator  
$$
\mathcal{A}_q: F_-  \mapsto F_+,  
$$
which is well-defined based on well-posed result in Theorem~\ref{theorem:main forward}, where $F_+:=F|_{(0,t^*)\times S_+}$ represents the outgoing density.
We are interested in the following question:
\begin{center}
    \noindent{\bf (IP)\quad}   \textit{From given albedo operator $\mathcal{A}_q$, can one uniquely determine both $\gamma(x)$ and $(\alpha,\,\beta)$  
    in the collision frequency $q$?}
\end{center}
The answer of this question is positive. We will state the unique determination result in Section~\ref{sec:intro IP}.

\subsection{Forward problem: bounded solutions with decay in the weighted $L^\infty$ norm}\label{sec:intro Forward}
We will first present a local existence result of the solution of the BGK equation \eqref{EQN:BGK}. This shows that the BGK solution exists in a finite time interval, and it serves the purpose of the study of inverse problem later. The investigation of a global existence result falls outside the current project's scope, and it is left for future study.

We consider a linearization of the BGK equation around the normalized global Maxwellian (uniform steady equilibrium):
$$
\mu(v) = (2\pi)^{-3/2} e^{-{|v|^2\over 2}}, \quad v\in\R^3,
$$
which satisfies $M(\mu)=\mu$ due to the properties of the normal distribution:
\begin{align}\label{normal distribution}
    \int_{\R^3} \mu(v) \,dv=1,\quad \int_{\R^3} v\mu(v) \,dv=0,\quad\hbox{ and }\quad \int_{\R^3} |v|^2\mu(v)  \,dv=3.
\end{align}

\begin{definition}
      For any fixed $(t,x)$, we define the macroscopic projection $\mathbf{P}$ as follows:
\begin{align*}
    \mathbf{P} f= \sum^{5}_{i=1} \langle f,e_i\rangle e_i,
\end{align*}
where $\{e_i\}$ forms an orthonormal basis for $5$ dimensional linear space spanned by 
$$
    e_1=\sqrt{\mu},\quad e_{j+1}= v_j \sqrt{\mu} \quad\hbox{for }j=1,\,2,\,3,\quad e_{5}={|v|^2-3 \over \sqrt{6}}\sqrt{\mu}.
$$
\end{definition}

Following a standard perturbation argument, we sought the solution $F$ in the form:
$$
    F(t,x,v) =\mu + \sqrt{\mu}f(t,x,v).
$$ 
Let $F(0,x,v) = \mu+\sqrt{\mu}f_{in}$ and $F(t,x,v)|_{S_-}=\mu + \sqrt{\mu}f_-$. Substituting such $F$ into the BGK equation \eqref{EQN:BGK} and using $M(\mu)=\mu$ yield the linearized BGK equation for $f$:
\begin{align}\label{EQN:Linearized BGK}
    \left\{
    \begin{array}{ll}
     \p_tf + v\cdot\nabla _x f= L_\gamma f+\Gamma f, & \quad (t,x,v)\in (0,t^*)\times\Omega\times\R^3,\\
    f(0,x,v)= f_{in}(x,v), & \quad ( x,v)\in  \Omega\times\R^3,\\
    f(t,x,v)= f_-(t,x,v),&\quad (t,x,v)\in (0,t^*)\times S_-,
    \end{array}
    \right. 
\end{align}
where we split $( \sqrt{\mu})^{-1}q(M(F)-F)=L_\gamma f+\Gamma f$ with the linearized collision operator  
$
    L_\gamma = \gamma(\mathbf{P} -\mathbf{I})
$ and nonlinear collision operator $\Gamma$, where $\mathbf{I}$ is the identity operator. The explicit forms and detailed derivations of $L_\gamma$ and $\Gamma$ will be discussed in Proposition~\ref{Prop:linearization of M(F)} in Section~\ref{sec:linearization}.

We define a weight function:
$$
    w(v) = (1+ c_1|v|^{2})^{c_2}, \quad\hbox{for constants $c_1,\,c_2>0$},
$$
which satisfies $w^{-2} \in L^1(\R^3)$. Similar weight functions are also used to prove uniqueness for the solution of the BGK equation with periodic boundary conditions in \cite{PerthameARMA}, and to study the Boltzmann equation in a bounded domain in \cite{Guo2010}. In \cite{CKP2026}, the paper considers the exponential weight for the BGK with the diffusive reflection boundary condition.

The first main result is stated below:
\begin{theorem}\label{theorem:main forward} Let $\Omega$ be an open bounded and convex domain in $\R^3$ with smooth boundary $\p\Omega$. Suppose that $\gamma \in L^\infty(\Omega)$ satisfies $\gamma \geq c_0$ almost everywhere for some positive constant $c_0$. There exists $\varepsilon_0>0$ such that if $F(0,x,v) = \mu+\sqrt{\mu}f_{in}$ and $F(t,x,v)|_{S_-}=\mu + \sqrt{\mu}f_-$ satisfy
    \begin{align}\label{delta small condition}
            \|wf_{in}\|_{L^\infty_{x,v}}  + \sup_{0\leq t\leq t^*} \|wf_-(t)\|_{L^\infty_{x,v}}<\varepsilon, \quad \hbox{ with }0<\varepsilon<\varepsilon_0,
    \end{align}
then there exists a unique solution $F(t,x,v)=\mu+ \sqrt{\mu}f$ to the in-flow boundary value problem \eqref{EQN:BGK}. Moreover, the following estimate holds:
    \begin{align}\label{Theorem:weight estimate}
     \sup_{0\leq t\leq t^*} \|wf(t)\|_{L^\infty_{x,v}}\leq C\LC\|wf_{in}\|_{L^\infty_{x,v}}+\sup_{0\leq t\leq t^*} \|wf_-(t)\|_{L^\infty_{x,v}}\RC,  
    \end{align}
where the constant $C>0$ depends on $t^*$ and $\varepsilon_0$.
\end{theorem}
Note that the positivity condition of $\gamma$ is necessary to ensure the coercivity of $L_\gamma$, see Lemma~\ref{lemma:coercivity}.
 
It is known that the highly nonlinear BGK collision operator introduces a challenge of studying the BGK equation. Since the nonlinear term $M(F)$ does not depend smoothly on the solution $F$, classical perturbation theory is not applicable and a treatment that takes into account the special structure of the $M(F)$ is necessary. For example, in \cite{PerthameARMA}, by estimating the moments of the solution $F$ in a weighted $L^\infty$ norm, the contraction mapping principle was applied to prove existence and uniqueness of the BGK solutions. 

In this paper, we seek the BGK solution taking the form $F(t,x,v) =\mu + \sqrt{\mu}f(t,x,v)$, and the main focus turns to solving the problem \eqref{EQN:Linearized BGK}.  
Notice that given the hypothesis \eqref{delta small condition}, the linearized BGK equation in \eqref{EQN:Linearized BGK} contains only a weakly nonlinear term $\Gamma f$, which can be viewed as a small perturbation of the linear part $L_\gamma f$. Hence, the well-posedness problem of the nonlinear BGK boils down to whether one can solve the following inhomogeneous linear transport equation:
\begin{align}\label{Intro:linear BGK}
    \p_t f + v\cdot\nabla _x f=  \gamma  ( \mathbf{P}  -\mathbf{I}) f +g,\quad \hbox{ with }f|_{t=0}= f_{in},\quad f|_{S_-}= f_-,
\end{align}
in a suitable functional space.
To this end, we will prove the $L^2$-$L^\infty$ existence of solution of \eqref{Intro:linear BGK}. 
After establishing the $L^2$ solution of \eqref{Intro:linear BGK}, the associated $L^2$ stability estimate is then utilized to control the term $\mathbf{P} f$ so that it leads to existence of the weighted $L^\infty$ solution of \eqref{Intro:linear BGK}, see Section~\ref{sec:L2 solution for the linear BGK} for $L^2$ solutions and Section~\ref{sec:bounded solution for the linear BGK} for bounded solutions to \eqref{Intro:linear BGK}. This technique resolves a difficulty of the study of the nonlinear BGK equation by controlling $\mathbf{P}f$ in a weighted $L^\infty$ space. Such a framework was discussed in \cite{Guo2010} for showing global existence results for the nonlinear Boltzmann equation with various boundary conditions, and the techniques developed there have significantly advanced the study of the Boltzmann equation on a bounded domain.  
Finally, Theorem~\ref{theorem:main forward} is established by the use of an iteration argument. 
Similar ideas were used in the study of the BGK solution under the diffusive reflection boundary condition in \cite{CKP2026}.

\subsection{Inverse problems - unique determination of collision frequency}\label{sec:intro IP}
The objective of the second part is to recover $\gamma$ and $(\alpha,\beta)$ in the collision frequency $q=\gamma \rho^\alpha T^\beta$.

For the purpose of the inverse problem, it is sufficient to consider the solution $F$ with initial data $F|_{t=0}=\mu$ (thus, $f_{in}\equiv 0$) since the unknown coefficient $\gamma$ only depends on the spatial variable $x$.

We denote the set 
$$
\mathcal{X}_{\varepsilon_0}:=\{\varphi:\,\, 
\varphi(t,x,v)=\mu+ \sqrt{\mu}f_-,\,\,\hbox{ where $f_-$ satisfies } \eqref{delta small condition}\},
$$
with the same $\varepsilon_0$ in Theorem~\ref{theorem:main forward}.

Since Theorem~\ref{theorem:main forward} ensures that the initial boundary value problem \eqref{EQN:BGK} is uniquely solvable,  
we can define the albedo operator  
$$
\mathcal{A}_q: F_-\in \mathcal{X}_{\varepsilon_0} \mapsto F_+,  
$$ 
where $F_+:=F|_{(0,t^*)\times S_+}=\mu+\sqrt{\mu}f_+$ with the outgoing density $f_+:=f|_{(0,t^*)\times S_+}$ satisfying $\sup_{0\leq t\leq t^*} \|w f_+(t)\|_{L^\infty_{x,v}}<\infty$.  
We are interested in whether the albedo operator $\mathcal{A}_q$ determines uniquely both $\gamma$ and $(\alpha,\beta)$ in the collision frequency $q$.

The first result is the unique determination of $\gamma$ without any smallness assumption. 
\begin{theorem}[Recover $\gamma$] \label{theorem:main inverse}  Let $\Omega$ be an open bounded and convex domain in $\R^3$ with smooth boundary $\p\Omega$. Suppose that $\gamma_j\in L^\infty(\Omega)$ $j=1,\,2$ satisfy $\gamma_j \geq c_0$ almost everywhere for some positive constant $c_0$.  
Let 
$$
    q_j(t,x,F)=\gamma_{j}(x)(\rho^{\alpha_j}T^{\beta_j})(t,x),\quad j=1,\,2.
$$ 
Let $\varepsilon_0>0$ be sufficiently small so that $\mathcal A_{q_j}$, $j=1,2$, are well-defined on $\mathcal X_{\varepsilon_0}$, if the corresponding albedo operators $\mathcal A_{q_1} =\mathcal A_{q_2}$ on $\mathcal{X}_{\varepsilon_0}$, then 
$$
     \gamma_{1} =\gamma_{2} \quad \hbox{ in } \Omega.
$$ 
\end{theorem}

The second result states that $(\alpha,\,\beta)$
can be recovered provided that the following assumption holds. 
We rewrite the operator $\mathbf{P}$ with kernel $k(v,v')$ as follows:
$$
    \mathbf{P}f = \int_{\R^3} k(v,v')f(v')\, dv', 
\quad\hbox{ with }
    k(v,v') :=  \sqrt{\mu (v)}\sqrt{\mu (v')} \bigg(1+v\cdot v'+\frac{(|v|^2-3)(|v'|^2-3)}{6}\bigg).
$$ 

\noindent{\bf Assumption 1: }
Assume that for some $b\in (0,{1\over 2})$, there holds
\begin{align}\label{intro:Hypothesis small}
 t^* c_b  \|\gamma\|_{L^\infty(\Omega)}  \LC\int_{\R^3} \mu^{b}(v) \,dv\RC < \tilde\varepsilon \in (0,1), \quad\hbox{with}\quad c_b:=\max_{v,v'\in\R^3} \mu^{-b}(v) |k(v,v')|<\infty.
\end{align}

\begin{theorem}[Recover $\alpha,\,\beta$] \label{theorem:main inverse power}  
Let $\Omega$ be an open bounded and convex domain in $\R^3$ with smooth boundary $\p\Omega$. Let 
$$
    q_j(t,x,F)=\gamma(x)(\rho^{\alpha_j}T^{\beta_j})(t,x),\quad j=1,\,2.
$$ 
There exists $\tilde\varepsilon>0$ sufficiently small, so that if $\gamma\in L^\infty(\Omega)$ satisfies the {\bf Assumption 1} for $\tilde\varepsilon$ and $\gamma \geq c_0$ almost everywhere for some positive constant $c_0$, and $\mathcal A_{q_1}=\mathcal A_{q_2}$ on $\mathcal X_{\varepsilon_0}$ for sufficiently small $\varepsilon_0>0$, then
$$\alpha_1=\alpha_2,\quad \beta_1=\beta_2.$$ 
\end{theorem}

\begin{remark}
    The choice of $\tilde\varepsilon$ depends on $t^*$, $\Omega$, and $|\alpha_j|+|\beta_j|$, $j=1,2$. 
\end{remark}

\begin{remark}
   Our proof of the recovery of $\gamma$ and $\alpha, \beta$ not only derives the uniqueness of the inverse problem, but also provides a reconstruction procedure.  
\end{remark}

The study on recovering coefficients in the Boltzmann equation has attracted great attention due to its broad applications. In particular, for the radiative transfer equation (known as a linear Boltzmann equation), unique reconstruction of absorption and scattering coefficients from given albedo operators was investigated in the works \cite{CS1, CS2, CS3, CS98, SU2d}, and the corresponding stability estimates were deduced in \cite{Bal14, Bal10, Bal18, DFRVinstability26, lai_inverse_2019, Wang1999, ZhaoZ18}, see also the relevant results for the nonlinear transport equation \cite{StefanovZhong2022} and for the transport equation in the Riemannian setting \cite{AY15, MST10, MST10stability, MST11, McDowall04} and review papers \cite{Balreview, Stefanov_2003}. In addition, the application of Carleman estimates was applied to study the dynamic transport equation with time-independent coefficients, see for example, \cite{Gaitan14,  Yamamoto2016, KlibanovP2006, Klibanov08, Lai-L, LaiUhlmannZhou22, Machida14}. As for the determination of the time-dependent coefficient, it was treated by constructing geometric optics solutions for the linear transport equation, see \cite{bellassouedInverseProblemLinear2019,BELLASSOUED19, LaiZhou2024} for more detailed discussions. Finally, based on an inversion of the Born series, a direct reconstruction method was proposed in \cite{MachidaSchotland15}.

There is less work on the inverse problem for the nonlinear Boltzmann equation. When considering the Boltzmann equation near the vacuum, the unique reconstruction of the collision kernel was showed in \cite{LaiUhlmannYang}. For the dynamic Boltzmann equation, the time-independent collision kernel was uniquely recovered through linearization around the global Maxwellian in \cite{LiOuyang2022}. Also, the stability estimate for the time-dependent kernel was derived in \cite{LY2023}. A related work by using the source-to-solution map to recover the metric is addressed in \cite{BKLL21}.

In this paper, we consider an inverse problem for the BGK with the goal of recovering the collision frequency $q$ from boundary measurements.
To achieve this, we introduce a small parameter $\varepsilon$ into the data, that is, 
\begin{align}\label{DEF:intro small data}
	F_\varepsilon|_{t=0}=\mu + \varepsilon \sqrt{\mu}f_{in}(x,v),\quad F_\varepsilon |_{S_-}=\mu + \varepsilon \sqrt{\mu}f_-(t,x,v),
\end{align}
and deduce an $\varepsilon$-expansion of the BGK solution: 
$$
    F_\varepsilon = \mu+\sqrt{\mu}f_\varepsilon \sim \mu+\sqrt{\mu}\LC \varepsilon f^{(1)}+{\varepsilon^2\over 2} f^{(2)} \RC +\hbox{higher-order terms}.
$$  
With the help of the nonlinearity of $\Gamma f$, one can decouple the effects associated with the linear term $L_\gamma f_\varepsilon$ and the nonlinear one $\Gamma f_\varepsilon$. Specifically, we have
\begin{align}\label{Intro:first linearization}
    \{\p_t + v\cdot\nabla _x -L_\gamma\} f^{(1)} =0,\quad\hbox{with } f^{(1)}|_{t=0}= f_{in},\quad f^{(1)}|_{S_-} = f_-,
\end{align}
and  
\begin{align}\label{Intro:second linearization}
     \{\p_t + v\cdot\nabla _x - L_\gamma\} f^{(2)} = \mathcal{G}(\alpha,\beta,\gamma, f^{(1)}), \quad\hbox{with } 
      f^{(2)}|_{t=0}=0,\quad f^{(2)}|_{S_-} = 0,
\end{align}
where $\mathcal{G}\equiv \mathcal{G}(\alpha,\beta,\gamma, f^{(1)})$, defined in \eqref{DEF:inhomogeneous G}, depends on $\alpha,\,\beta,\,\gamma,\,f^{(1)}$.
As a result, with appropriately chosen in-flow boundary condition and trivial initial data, we can recover $\gamma$ from \eqref{Intro:first linearization}, and then determine the powers $\alpha$ and $\beta$ from the succeeding linearized BGK model \eqref{Intro:second linearization}.

The main contributions of this work is as follows. First, we establish $L^\infty_{x,v}$ estimate for the nonlinear operator $\Gamma$ and prove the well-posedness result for the nonlinear BGK equation in a bounded domain with in-flow boundary condition. In addition, we derive a higher-order expansion of the BGK solution and utilize it to study an inverse problem for the non-constant collision frequency. It is worth pointing out that due to the established $L^2$-$L^\infty$ existence of the BGK solution, the recovery of $\gamma$ does not require a-priori smallness assumption. Finally, 
since the structure of $\mathcal{G}$ enhances the overall regularity, and makes it hard to recover the powers $(\alpha,\beta)$ from $\mathcal{G}$, we impose an additional assumption to distinguish the contribution from the singular part of $f^{(1)}$. We refer to Section~\ref{Sec:inverse problem} for more detailed discussions on how the reconstruction procedure works.

\subsection{Comparison between the Boltzmann and the BGK models}
While BGK model is considered as a relaxation of the Boltzmann, there are several distinct features between these two models in the study of both forward and inverse problems. In the following, we will make several comparisons between Boltzmann and BGK by considering the solution taking the form $F=\mu+\sqrt{\mu}f$. The Boltzmann equation \eqref{EQN:Boltzmann intro} with $K\equiv K(|v-u|,\omega)$ can be rewritten as
$$
     \{\p_t + v\cdot\nabla_x\}f  = {1\over \sqrt{\mu}} (Q(\mu,\sqrt{\mu} f) +Q(\sqrt{\mu}f,\mu))  +{1\over \sqrt{\mu}} Q(\sqrt{\mu} f,\sqrt{\mu} f)=: L^B(f)+\Gamma^B(f),
$$  
with the linear operator given by
$$
    L^B(f)= -q_B(v) f+ K^B f,\quad \hbox{ where }K^B f:=\int_{\R^3}k_B(v,v')f(v')\,dv',
$$
collision frequency $q_B(v)\sim (1+|v|)^b$ and $k_B(v,v')\lesssim {e^{-C|v-v'|^2}\over |v-v'|}$, see for instance \cite{Guo2010} for details.

Unlike the linear Boltzmann operator $K^B f$ having singularity in the kernel $k_B$, it is rather straightforward to prove that the linear BGK operator $\mathbf{P}$ satisfies
\begin{align*}
    w (v)\mathbf{P} f(t,x,v)&= w(v)\sum^{5}_{i=1}\LC\int_{\R^3} f(t,x,v')e_i(v')\,dv'  \RC e_i(v)\\
    &\leq \|wf \|_{L^\infty_{t,x}} \sum^{5}_{i=1}\LC\int_{\R^3} w^{-1}(v')e_i(v')\,dv'  \RC w(v)e_i(v)\lesssim \|wf\|_{L^\infty_{x,v}}.
\end{align*}
However, in contrast to the quadratic nonlinearity of $\Gamma^B f$, the term $\Gamma f$ of the BGK consists of relatively strong nonlinearity. In particular, as proved in Proposition~\ref{Prop:linearization of M(F)}, we obtain the Taylor expansion of $\Gamma f$, and the functions $\mathcal{P}_\ell$ and $\mathcal{Q}_{ij}$ depend on the macroscopic quantities $(\rho,U,T)$ of the solution $F$, which suggests nonlinear dependence on $f$. Such a complicated structure creates difficulty for controlling $\Gamma f$ in the weighted $L^\infty$ norm, that is, $\|w\Gamma(f)\|_{L^\infty_{x,v}}$, and therefore a careful treatment is needed to overcome this issue, see Appendix~\ref{sec:appendix} for relevant estimates of $\Gamma f$.

For inverse problems, the reconstruction of unknown coefficients also differs between the Boltzmann and BGK. Since $\Gamma^B f$ has quadratic nonlinearity, when taking small data \eqref{DEF:intro small data}, the first-order linearization of the Boltzmann is 
$
    \{\p_t + v\cdot\nabla _x -L^B\} f^{(1)}=0,
$
which looks similar to \eqref{Intro:first linearization}, yet the inhomogeneous part of its second-order linearization 
$
    \{\p_t + v\cdot\nabla _x -L^B\} f^{(2)}=\Gamma^B(f^{(1)})
$ 
is simpler than $\mathcal{G}$ in \eqref{Intro:second linearization}. Based on this linearization scheme (linearize around the vacuum), in \cite{LaiUhlmannYang}, for time-independent Boltzmann, the collision frequency was recovered by suitably choosing free transport solutions $f^{(1)}$ as test functions to infer the unknown kernel, see also \cite{LiOuyang2022} for inverse problem for the time-dependent one \eqref{EQN:Boltzmann intro}.

The remainder of this paper is organized as follows. In Section~\ref{sec:linearization}, we introduce notations, and then derive fundamental properties of the operators $L_\gamma$ and $\Gamma$ in order to expand $M(F)$. Section~\ref{sec:L2 solution for the linear BGK} is devoted to establish the $L^2$ existence of the solution for the linearized BGK equation. In Section~\ref{sec:bounded solution for the linear BGK}, an iteration scheme is applied to prove bounded solutions of the nonlinear BGK equation after showing the existence of the solution to the linearized problem in the weighted $L^\infty$ norm and deriving the related stability estimate. Finally, we study an inverse problem for the BGK equation in Section~\ref{Sec:inverse problem}. We show that both coefficient $\gamma$ and the powers $(\alpha,\beta)$ can be uniquely determined from the albedo operator. In Appendix~\ref{sec:appendix}, several basic estimates for the macroscopic quantities $(\rho,\,U,\,T)$ are discussed.
Moreover, in Appendix~\ref{sec:appendix B}, we provide detailed computations for the linearization of the BGK equation through directly differentiating the equation with respect to a small parameter.

\section{Preliminaries and analysis of BGK collision operator}\label{sec:linearization}
The objective of this section is to analyze the Maxwellian $M(F)$, and expand it near equilibrium $\mu$ under the assumption that $w(v)f(t,\cdot,\cdot) \in L^\infty_{x,v}$ so that $f(t,\cdot,\cdot)\in L^2(\Omega\times\R^3)$. Recall the weight function $w(v) = (1+ c_1|v|^{2})^{c_2}$ for some constants $c_1,$ $c_2$. The reason of considering these spaces will become clearer in the proof of Theorem~\ref{theorem:main forward}. 

The analysis developed here throughout the way will be essential to handle the nonlinearity in the BGK collision operator. Equipped with this, we will show 
the existence of the $L^2$ solution of the linearized BGK in Section~\ref{sec:L2 solution for the linear BGK}, and then establish the well-posedness for the BGK in Section~\ref{sec:bounded solution for the linear BGK}.

\subsection{Preliminaries}\label{sec:preliminaries}
To prepare the upcoming discussions, we first introduce notations of function spaces and norms, and also state basic preliminary results. 

 For $1\leq p<\infty$, we use $\|\cdot\|_{L^p_{x}}$, $\|\cdot\|_{L^p_{v}}$, $\|\cdot\|_{L^p_{x,v}}$, and $\|\cdot\|_{L^p}$ to denote the norm for the space $L^p(\Omega)$, $L^p(\R^3)$, $L^p(\Omega\times\R^3)$, and $L^p((0,t^*)\times\Omega\times\R^3)$, respectively. 
We denote phase boundary $\p\Omega\times\R^3$ by $S$ and define the space $L^p(\p\Omega\times\R^3)$ with the norm
$$
    \|g\|_{L^p(S) }:=\LC\int_{S } |g(t,x,v)|^p  d\Sigma\RC^{1/p},  
$$ 
where $d\Sigma=(\nu(x)\cdot v)d\mathcal{S}_xdv$ with the standard surface measure $d\mathcal{S}_x$ on $\p\Omega$. Moreover, for the spaces $L^p(S_\pm)$, the norms are
$$
\|g\|_{L^p(S_\pm)}=\LC \int_{S_\pm}|g(t,x,v)|^p\,d\Sigma\RC^{1/p}=\LC\int_{\p\Omega}\int_{\pm\nu(x)\cdot v>0}   |g(t,x,v)|^p |\nu(x)\cdot v|\, dvd\mathcal{S}_x\RC^{1/p}.
$$ 
When $p=\infty$, the above spaces are the standard vector spaces consisting of all functions that are essentially bounded. 

In addition, we use the notations $\langle\cdot,\cdot\rangle_v$ and $\langle\cdot,\cdot\rangle_{x,v}$ for the $L^2(\R^3)$ inner product in the $v$-variable and for the $L^2(\Omega\times\R^3)$ inner product in the $x,v$-variable, respectively.

\subsection{The expansion of $M(F)$ around $\mu$}
Let $F=\mu+\sqrt{\mu}f$. To expand $M(F)$, a new coordinate emerges naturally from the following lemma. 
\begin{lemma}\label{lemma:decomposition}
We obtain
    \begin{align}\label{change of variables 1}
        \rho = 1 + \langle f,e_1\rangle ,\quad \rho  U = (\langle f,e_{2}\rangle,\langle f,e_{3}\rangle,\langle f,e_{4}\rangle),  
    \end{align}
    \begin{align}\label{change of variables 2}
    G :={\rho  |U |^2+3\rho  T \over \sqrt{6}} - {3\over \sqrt{6}}\rho  
    =   \langle f,e_5\rangle.
    \end{align}
\end{lemma}
\begin{proof}
     By \eqref{normal distribution} and the definitions of $(\rho,\,U,\,T)$, 
     direct computations give
	$$
	\rho = \int_{\R^3} F\,dv=1+\int_{\R^3} \sqrt{\mu}f\,dv =1+\langle f,e_1\rangle ,\quad \rho U =  \int_{\R^3} vF\,dv=\int_{\R^3} v\sqrt{\mu}f\,dv = (\langle f,e_{2}\rangle,\langle f,e_{3}\rangle,\langle f,e_{4}\rangle),  
	$$
	and 
	\begin{align*}
		3\rho  T 
        &= 3+  \LC\int_{\R^3} |v|^2 \sqrt{\mu}f\,dv\RC -\rho  |U |^2\\
        &= 3+ \sqrt{6} \LC\int_{\R^3} {|v|^2-3\over \sqrt{6}} \sqrt{\mu}f\,dv\RC + 3 \LC\int_{\R^3}  \sqrt{\mu}f\,dv\RC-\rho  |U |^2.
	\end{align*}
    Also, using the definition of $e_5$ and $3\rho =3+3 \int_{\R^3} \sqrt{\mu}fdv$ leads to 
    $$
    \rho  |U |^2+3\rho  T -3\rho  = \sqrt{6}\langle f,e_5\rangle.
    $$
\end{proof}

Lemma~\ref{lemma:decomposition} indicates that the function $G$ arises from the projection of $f$ onto $e_5$ in \eqref{change of variables 2}.  
In particular, we have 
$
(\rho-1,  \rho U,G)=(\langle f, e_1\rangle, \ldots,\langle f, e_5\rangle).
$
Hence, this observation inspires application of the following change of variables.

\begin{lemma}\label{lemma:jacobian}
    The Jacobian matrix of the change of variables $(\rho,U,T)\rightarrow (\rho,\rho U,G)$ is
    \begin{align*}
        {\p(\rho,\rho U,G)\over \p(\rho, U, T)}
        =\left(\begin{array}{ccccc}
            1 &  U_1 & U_2 &U_3 & {|U |^2+3 T  -3\over\sqrt{6}} \\
             0& \rho&0 & 0 & {2\rho U_1 \over \sqrt{6}}\\
           0 & 0& \rho&0  & {2\rho U_2 \over \sqrt{6}}\\
            0& 0&0  &\rho& {2\rho U_3 \over \sqrt{6}}\\
         0 &0 &0&0& {3\rho\over \sqrt{6}} \\
        \end{array}\right)
    \end{align*}
    where we denote $U:=(U_1,U_2,U_3)$, and
        \begin{align}\label{Jacobia}
        {\p(\rho,\rho U,G)\over \p(\rho, U, T)}^{-1}
        =\left(\begin{array}{ccccc}
            1 &  -{U_1\over \rho} & -{U_2\over \rho} &-{U_3\over \rho} &{-3T+|U|^2+3 \over 3\rho}\\
            0 & {1\over \rho}&0 & 0 &-{2U_1\over 3\rho}\\
            0 & 0& {1\over \rho}&0  &-{2U_2\over 3\rho}\\
            0& 0&0 & {1\over \rho} &-{2U_3\over 3\rho}\\
            0 &0&0&0&{\sqrt{2\over 3}\over \rho} \\
        \end{array}\right).
    \end{align}
\end{lemma}
\begin{proof}
    The proof follows directly by first computing the Jacobian $ {\p(\rho,\rho U,G)\over \p(\rho, U, T)}$ and then inverting this matrix function to get ${\p(\rho,\rho U,G)\over \p(\rho, U, T)}^{-1}$. 
\end{proof}

In the following, we will use the subscript $\theta$ in $F_\theta=\mu+\theta\sqrt{\mu} f$ to describe the transition from $\mu$ to $\mu+\sqrt{\mu}f$.    
We also denote the corresponding density, velocity, and temperature as follows:
\begin{align}\label{DEF:macro epsilon}
    \rho_\theta=\int_{\R^3} F_\theta \,dv,\quad U_\theta= \rho_\theta^{-1}\int_{\R^3} v F_\theta\,dv,\quad \hbox{and}\quad T_\theta= (3\rho_\theta)^{-1} \int_{\R^3} |v-U_\theta|^2 F_\theta \,dv.
\end{align}
In particular, when $\theta=0$, we have $F_0=\mu(v)$ and thus $(\rho_0,\rho_0 U_0,G_0)=(1,0,0)$ due to $\rho_0=1$, $U_0=0$, $T_0=1$, which follows from \eqref{normal distribution}. 
Also, when $\theta=1$, we simply denote
$$
      (\rho,U,T)\equiv (\rho_1,U_1,T_1), \quad\hbox{ and }\quad  F\equiv F_1.
$$
Following the same computations in the proof of Lemma~\ref{lemma:decomposition} (replacing $f$ by $\theta f$), we have  
\begin{align}\label{DEF:macro theta expansion}
    (\rho_\theta,\rho_\theta U_\theta,G_\theta) = (1-\theta) (1,0,0) + \theta (\rho,\rho U,G),\quad \theta\in [0,1],
\end{align}
which describes the transition from $(1,0,0)$ to $(\rho,\rho U,G)$.

\begin{lemma}\label{lemma:first order 1}
For $\theta\in [0,1]$, let $F_\theta=\mu+\theta\sqrt{\mu} f$, we have
$$
    M(F_{\theta}) = { \rho_{\theta} \over ( 2\pi T_{\theta})^{3/2}}  \exp\LC{-|v-U_{\theta}|^2\over 2 T_{\theta}}\RC,\quad v=(v_1,v_2,v_3)\in\R^3.  
$$
Then its derivatives $D_{(\rho_{\theta },  U_{ \theta }, T_{ \theta })}  M(F_{\theta })
        :=\LC{\p M(F_\theta)\over \p \rho_\theta},{\p M(F_\theta)\over \p U_\theta},{\p M(F_\theta)\over \p T_\theta}\RC^{tr}$ is as follows:
\begin{align*}
        D_{(\rho_{\theta },  U_{ \theta }, T_{ \theta })}  M(F_{\theta })
        =   V_\theta M(F_\theta),\quad \hbox{ where we denote }V_\theta := \left(\begin{array}{c}
            {1\over \rho_{\theta }}\\
              {v_1-U_{\theta ,1} \over T_{\theta }}\\
               {v_2-U_{\theta ,2} \over T_{\theta }}\\
                {v_3-U_{\theta ,3} \over T_{\theta }}\\
               {|v-U_{\theta }|^2-3T_{\theta }\over 2T_{\theta }^2}\\
    \end{array}\right),
\end{align*}
where the transpose of a matrix $A$ is denoted by $A^{tr}$.
In particular, when $\theta=0$, we have
\begin{align*}
        \left. D_{(\rho_{\theta },  U_{ \theta }, T_{ \theta })}  M(F_{\theta }) \right|_{\theta=0}=   \left(\begin{array}{c}
           1\\
           v_1 \\
           v_2 \\
           v_3 \\
          {|v|^2-3\over 2} \\
        \end{array}\right) \mu(v).
\end{align*}
\end{lemma}
\begin{proof}
    The first identity can be verified by straightforward computations.
    When $\theta=0$, we have $F_0(t,x,v)=\mu(v)$, and combine it with $\rho_0=T_0=1$, $U_0=0$ leading to the desired result.
\end{proof}

Now we derive an expansion of the nonlinear term $M(F)$, see \cite{CKP2026, Yun10} for relevant discussions.
\begin{proposition}\label{Prop:linearization of M(F) new} Let $F =\mu+ \sqrt{\mu}f$. 
    Then the BGK operator $M(F)$ can be expanded around $\mu$ as follows:
		\begin{align*}
			M(F) = \mu +  \sqrt{\mu}\, \mathbf{P} f+ \sum_{i,j=1}^5 \LC \int^1_0  D^2_{(\rho_{\theta }, \rho_{ \theta } U_{ \theta }, G_{ \theta })}   M(F_{\theta})(1-\theta)\,d\theta\RC_{ij} \langle f,e_i\rangle\langle f,e_j\rangle . 
		\end{align*}
\end{proposition}
\begin{proof}
We define a function  
    $$
    h(\theta):= M(F_{\theta }) ,\quad \theta\in [0,1],
    $$
    which represents the transition from the global Maxwellian $h(0)=M(F_0
    )=\mu$ to the local Maxwellian $h(1) =M(F )$.

    Applying the Taylor's theorem around $\theta=0$ yields 
    \begin{align*}
        h(1) = h(0)+h'(0) + \int^1_0h''(\theta)(1-\theta)d\theta.
    \end{align*}
    To compute $h'(0)$, recall that we can view $M(F_\theta)$ as a function of the variables $(\rho_\theta,\rho_\theta U_\theta,G_\theta)$. We apply \eqref{DEF:macro theta expansion}, Lemma~\ref{lemma:first order 1}, and the chain rule to get
    \begin{align*}
        h'(0)&=   \left. {d\over d\theta}M(F_{\theta }) \right|_{\theta=0}\\
        &=  {d\over d\theta} (\rho_{ \theta }, \rho_{ \theta } U_{\theta }, G_{\theta })
        \left. D_{(\rho_{\theta }, \rho_{ \theta } U_{ \theta }, G_{ \theta })}  M(F_{\theta })\right|_{\theta=0}\\
         &=(\rho-1, \rho  U , G ) {\p(\rho_{\theta} , \rho_{\theta} U_{\theta} , G_{\theta})\over \p(\rho_{ \theta}, U_{ \theta}, T_{ \theta})}^{-1} \left. D_{(\rho_{\theta},  U_{\theta }, T_{\theta})} M(F_{\theta })\right|_{\theta=0}\\
         &=(\rho-1, \rho U, G)
         \left(\begin{array}{ccccc}
            1 &  0 & 0 &0 &0\\
            0& 1&0 & 0 &0\\
            0& 0&1 &0  &0\\
            0& 0&0 &1&0\\
           0 & 0 &0 &0& \sqrt{2\over 3} \\
        \end{array}\right)
        \left(\begin{array}{c}
            1 \\
            v_1 \\
             v_2\\
             v_3\\
             |v|^2-3\over 2 \\
        \end{array}\right) \mu (v) 
        =\sqrt{\mu}\,\mathbf{P}f,
    \end{align*}
   where in the last identity above, the $5\times 5$ matrix follows by using \eqref{Jacobia} in Lemma~\ref{lemma:jacobian} and $\rho_0=1$, $U_0=0$, and $T_0=1$.
   Also, we applied the definition of the operator $\mathbf{P}$ and the fact that 
   $$
       (\rho-1, \rho U, G) =(\langle f,e_1\rangle, \ldots, \langle f,e_5\rangle).
   $$
   Finally, we conclude the proof by using the chain rule again to get  
   \begin{align*}
       h''(\theta)&={d^2\over d\theta^2}M(F_{\theta })\\
       &= (\rho-1, \rho U, G) D^2_{(\rho_{\theta}, \rho_{ \theta} U_{ \theta}, G_{ \theta})} M(F_{\theta })(\rho-1, \rho U, G)^{tr}\\
       &=(\langle f,e_1\rangle ,  \ldots,\langle f,e_5\rangle)
       D^2_{(\rho_{\theta}, \rho_{ \theta} U_{ \theta}, G_{ \theta})}  M(F_{\theta })
       (\langle f,e_1\rangle , \ldots, \langle f,e_5\rangle)^{tr}.
   \end{align*}
\end{proof}

\begin{proposition}\label{Prop:linearization of q} Let $F =\mu+ \sqrt{\mu}f$  
    and
    $$      
    q(t,x,F)=\gamma(x)\rho^{\alpha}T^{\beta}.
    $$
Then $q$ can be expanded as follows:
		\begin{align*}
			q(t,x,F)= \gamma +\sum^{5}_{i=1} \LC\int^1_0 D_{(\rho_{\theta},\rho_{\theta} U_{\theta}, G_{\theta})} q(t,x, F_\theta) d\theta\RC_i\langle f,e_i\rangle.
		\end{align*}
\end{proposition}
 
\begin{proof}
We denote $\tilde{h}(\theta):=q(t,x,F_{\theta})$.
From Taylor's theorem, 
\begin{align*}
        \tilde{h}(1) = \tilde{h}(0)+  \int^1_0 \tilde{h}' (\theta) d\theta,
\end{align*}
where $\tilde{h}(0)=\gamma$ as $\rho_0=T_0=1$. We compute
    \begin{align*}
        \tilde{h}'(\theta)&=(\rho-1, \rho U, G) D_{(\rho_{\theta},\rho_{\theta} U_{\theta}, G_{\theta})}q(t,x,F_{\theta})  \\
        &= (\langle f,e_1\rangle , \ldots, \langle f,e_5\rangle) {\p(\rho_{\theta},\rho_{\theta} U_{\theta},G_{\theta})\over \p(\rho_{\theta}, U_{\theta}, T_{\theta})}^{-1}  D_{(\rho_{\theta}, U_{\theta},T_{\theta})} q(t,x,F_{\theta}),
    \end{align*}
    where more precisely,
    \begin{align}\label{first derivative of q}
        D_{(\rho_{\theta},\rho_{\theta} U_{\theta}, G_{\theta})}q(t,x,F_{\theta})
        &={\p(\rho_{\theta},\rho_{\theta} U_{\theta},G_{\theta})\over \p(\rho_{\theta}, U_{\theta}, T_{\theta})}^{-1}  D_{(\rho_{\theta}, U_{\theta},T_{\theta})} q(t,x,F_{\theta}) \notag\\
        &=\underbrace{\left(\begin{array}{ccccc}
            1 &  -{U_{\theta,1}\over \rho_\theta} & -{U_{\theta,2}\over \rho_\theta} &-{U_{\theta,3}\over \rho_\theta} &{-3T_{\theta}+|U_{\theta}|^2+3 \over 3\rho_{\theta}}\\
            0 & {1\over \rho_{\theta}}&0 & 0 &-{2U_{\theta,1}\over 3\rho_{\theta}}\\
            0 & 0& {1\over \rho_{\theta}}&0  &-{2U_{\theta,2}\over 3\rho_{\theta}}\\
            0& 0&0 & {1\over \rho_{\theta}} &-{2U_{\theta,3}\over 3\rho_{\theta}}\\
            0 &0&0&0&{\sqrt{2\over 3}\over \rho_{\theta}} \\
        \end{array}\right)}_{\hbox{ denoted by $J_{\theta}$}}
        \underbrace{\left(\begin{array}{c}
            \gamma(x)\alpha\rho_{\theta}^{\alpha-1} T_{\theta}^{\beta}\\
            0\\
            0\\
            0\\
            \gamma(x)\beta \rho_{\theta}^{\alpha} T_{\theta}^{\beta-1}\\
        \end{array}\right)}_{D_{(\rho_{\theta}, U_{\theta},T_{\theta})} q(t,x,F_{\theta})}.
    \end{align}
\end{proof}

Fix $\varepsilon>0$, following arguments in \cite{CKP2026, Yun10}, we denote 
$$
 \mathcal{P}_i (f):= \LC D_{(\rho_{\theta},\rho_{\theta} U_{\theta}, G_{\theta})} \tilde{h}(\theta)\RC_i,\quad  1\leq i\leq 5, 
$$
and 
$$
    \mathcal{Q}_{ij}(f) : = \{D^2_{(\rho_{\theta}, \rho_{ \theta} U_{ \theta}, G_{\theta})}  h(\theta)\}_{ij},\quad 1\leq i,\,j\leq 5,
$$
where we recall $ h(\theta)=M(F_{\theta}) =  { \rho_{\theta} \over ( 2\pi T_{\theta})^{3/2}}  \exp\LC{-|v-U_{\theta}|^2\over 2 T_{\theta}}\RC $.
Without fully expanding the expressions of $\mathcal{P}_i$ and $\mathcal{Q}_{ij}$, the following forms of $\mathcal{P}_i$ and $\mathcal{Q}_{ij}$ will be sufficient for our analysis later in proving the well-posedness results. 

For a positive integer $m$, we denote $m$-dimensional multi-index  
$\beta=(\beta_1,\ldots,\beta_{m})\in \R^m$ of non-negative integers $\beta_j$, and $|\beta|=\beta_1+\ldots+\beta_{m}$. 
We write $x^\beta = x_1^{\beta_1}\ldots x_{m}^{\beta_{m}}$ for a vector $x\in \R^{m}$.

From the above computations of $\tilde{h}'(\theta)$ in Proposition~\ref{Prop:linearization of q}, it is clear that the denominator of each term of $D_{(\rho_{\theta},\rho_{\theta} U_{\theta}, G_{\theta})}q$ only consists of functions $\rho_{\theta}$ and $T_{\theta}$. Thus, we have the following result.

\begin{lemma}\label{lemma:first derivative of P}
For each $1\leq i\leq 5$, we have
\begin{align*}
       \mathcal{P}_i(f) = {P_{i} ( \rho_{ \theta},U_{ \theta}, T_{ \theta})\over R_{i}(\rho_{ \theta},T_{ \theta})}, 
    \end{align*}
    where $R_{i}(\rho_{ \theta},T_{ \theta}) = a_\sigma (\rho_{ \theta})^{\sigma_1} (T_{ \theta})^{\sigma_2}$ is a monomial (a polynomial which has only one term) with constants $a_{\sigma}>0$ and $\sigma_1,\,\sigma_2\geq 0$ are constants. And
    \begin{align*}
        P_{i}(\rho_{ \theta},U_{ \theta}, T_{ \theta}) = \sum_{\kappa\in S_{i} }b_\kappa (\rho_{ \theta})^{\kappa_1}  (U_{ \theta,1})^{\kappa_2} (U_{ \theta,2})^{\kappa_3}(U_{ \theta,3})^{\kappa_4}   (T_{ \theta})^{\kappa_5},
    \end{align*}
    where $b_\kappa $ is a constant, $\kappa=(\kappa_1,\ldots,\kappa_5)$ with $\kappa_1,\ldots,\kappa_5\geq 0$ are constants, and $S_i$ is a collection of finitely many $\kappa$.
    Note that $(\rho_{\theta}, U_{\theta}, T_{\theta})$ is a $5$-dim vector function.
\end{lemma}
 
\begin{remark}
     Notice that in the definition of $ q(t,x,F)=\gamma(x)\rho^{\alpha}T^{\beta}$, we only suppose that $\alpha,\,\beta$ are some constants, and thus the powers (i.e., $\sigma_j$ and $\kappa_k$) in $\mathcal{P}_i(f)$ above are not necessarily integers.  
\end{remark}

The second derivatives of $M(F_{\theta})$ appearing in $h''(\theta)$ in the proof of Proposition~\ref{Prop:linearization of M(F) new} can be expressed implicitly as follows:
\begin{lemma}\label{lemma:first derivative of Qij}
For each $1\leq i,\,j\leq 5$, we have
    \begin{align*}
        \mathcal{Q}_{ij}(f)={P_{ij}(\rho_{ \theta},v-U_{ \theta},U_{ \theta}, T_{ \theta})\over R_{ij}(\rho_{ \theta},T_{ \theta})}M(F_{ \theta}),\quad  
    \end{align*}
    where $R_{ij}(\rho_{ \theta},T_{ \theta}) = a_\sigma(\rho_{ \theta} T_{ \theta})^{\sigma}$ is a monomial with constants $a_{\sigma}>0$ and $\sigma\in (\mathbb{N}\cup \{0\})^2$ is a multi-index depending on $R_{ij}$. Moreover,
    \begin{align*}
        P_{ij}(\rho_{ \theta},v-U_{ \theta},U_{ \theta}, T_{ \theta}) = \sum_{ \kappa \in S_{ij} }b_\kappa [ \rho_{ \theta} (v-U_{ \theta}) U_{ \theta}  T_{ \theta})]^\kappa,
    \end{align*}
    where $b_\kappa $ is a constant, $\kappa\in (\mathbb{N}\cup \{0\})^{8}$ is a multi-index, and $S_{ij}$ is a collection of finitely many $\kappa$. 
    Note that $(\rho_{ \theta} , (v-U_{ \theta}) ,U_{ \theta},  T_{ \theta})$ is an $8$-dim vector function.  
\end{lemma}
\begin{proof}
From Lemma~\ref{lemma:jacobian} and Lemma~\ref{lemma:first order 1}, by recalling the definition of $V_\theta$ and $J_\theta={\p(\rho_{\theta},\rho_{\theta} U_{\theta},G_{\theta})\over \p(\rho_{\theta}, U_{\theta}, T_{\theta})}^{-1}$, we have
\begin{align*}
        D_{(\rho_{\theta }, \rho_\theta U_{ \theta }, G_{ \theta })}  M(F_{\theta })
        = {\p(\rho_{ \theta }, \rho_{ \theta } U_{ \theta }, G_{\theta})\over \p(\rho_{ \theta }, U_{ \theta }, T_{ \theta })}^{-1}
       D_{(\rho_{\theta },  U_{\theta }, T_{\theta })} M(F_\theta)
       =   J_\theta V_\theta M(F_\theta).
\end{align*}
This then leads to
\begin{align*}
     \{D^2_{(\rho_{\theta }, \rho_\theta U_{ \theta }, G_{ \theta })}  M(F_{\theta })\}_{ij}
     &=\{D_{(\rho_{\theta }, \rho_\theta U_{ \theta }, G_{ \theta })} D_{(\rho_{\theta }, \rho_\theta U_{ \theta }, G_{ \theta })}  M(F_{\theta }) \}_{ij}\\
     &= \{ D_{(\rho_{\theta }, \rho_\theta U_{ \theta }, G_{ \theta })} [J_\theta V_\theta M(F_\theta)]\}_{ij}\\
     &=\{ J_{\theta}D_{(\rho_{\theta },  U_{\theta }, T_{\theta  })}[J_\theta V_\theta M(F_\theta)]\}_{ij}\\
     &={P_{ij}(\rho_{ \theta},v-U_{ \theta},U_{ \theta}, T_{ \theta})\over R_{ij}(\rho_{ \theta},T_{ \theta})}M(F_{ \theta}).
\end{align*}
\end{proof}

Finally, we end this section by stating the following linearization of the nonlinear BGK collision.
\begin{proposition}\label{Prop:linearization of M(F)} Let $F=\mu+ \sqrt{\mu}f$. Then we have
		\begin{align*}
			q(t,x,F ) \LC M(F )-F \RC = \sqrt{\mu}L_\gamma f+ \sqrt{\mu}( \Gamma f),
		\end{align*}
    where the linear term is denoted as
    \begin{align}\label{DEF:L f}
    L_\gamma (f) := \gamma (\mathbf{P} -\mathbf{I})f,
\end{align}
    and the nonlinear term is denoted as 
    \begin{align}\label{DEF:Gamma f}
    \Gamma (f)
    &=:   \Gamma_1(f)+  \Gamma_2(f)+ \Gamma_3(f),
\end{align}
with
$$
     \Gamma_1(f):= (\mathbf{P}f-f) \underbrace{\LC  \sum^{5}_{\ell=1} \LC\int^1_0 \mathcal{P}_\ell (f)d\theta\RC\langle f,e_\ell\rangle \RC}_{\hbox{denoted by }\Upsilon_1(f)},  
$$
$$
    \Gamma_2(f):= \gamma \sqrt{\mu}^{-1} \underbrace{\sum_{i,j=1}^5\LC \int^1_0  \mathcal{Q}_{ij}(f)(1-\theta)\,d\theta\RC \langle f,e_i\rangle\langle f,e_j\rangle }_{\hbox{denoted by }\Upsilon_2(f)}   ,
$$
and
$$
    \Gamma_3(f):= \sqrt{\mu}^{-1} \underbrace{\LC \sum_{i,j,\ell=1}^5 \LC\int^1_0 \mathcal{P}_\ell (f)d\theta\RC\LC \int^1_0  \mathcal{Q}_{ij}(f)(1-\theta)\,d\theta\RC\langle f,e_i\rangle\langle f,e_j\rangle\langle f,e_\ell\rangle\RC}_{\hbox{denoted by }\Upsilon_3(f)=\Upsilon_1(f)\Upsilon_2(f)}.
$$
\end{proposition}
  
\begin{proof}
From Proposition~\ref{Prop:linearization of M(F) new} and Proposition~\ref{Prop:linearization of q}, Lemma~\ref{lemma:first derivative of P} and Lemma~\ref{lemma:first derivative of Qij}, the proof is complete by expanding the following:
\begin{align*}
    & \gamma \rho^{\alpha} T^{\beta}  \LC M(F )-F \RC\\
    &=\left[ \gamma +  
      \LC  \sum^{5}_{\ell=1} \LC\int^1_0 \mathcal{P}_\ell d\theta\RC\langle f,e_\ell\rangle \RC   \right] 
    \cdot\left[  \sqrt{\mu} (\mathbf{P} f-f)+  \sum_{i,j=1}^5\LC \int^1_0  \mathcal{Q}_{ij}(1-\theta)\,d\theta\RC \langle f,e_i\rangle\langle f,e_j\rangle \right]\\
    &= \gamma\sqrt{\mu}(\mathbf{P} f-f) +   \sqrt{\mu} (\mathbf{P} f-f)  \LC  \sum^{5}_{\ell=1} \LC\int^1_0 \mathcal{P}_\ell d\theta\RC\langle f,e_\ell\rangle \RC   + \gamma \sum_{i,j=1}^5\LC \int^1_0  \mathcal{Q}_{ij}(1-\theta)\,d\theta\RC \langle f,e_i\rangle\langle f,e_j\rangle  \\
    &\quad + \LC  \sum^{5}_{\ell=1} \LC\int^1_0 \mathcal{P}_\ell d\theta\RC\langle f,e_\ell\rangle \RC \LC \sum_{i,j=1}^5\LC \int^1_0  \mathcal{Q}_{ij}(1-\theta)\,d\theta\RC \langle f,e_i\rangle\langle f,e_j\rangle\RC.
\end{align*}
\end{proof}

\section{Theory for the $L^2$ estimate in the linear BGK}\label{sec:L2 solution for the linear BGK}
The goal of this section is to establish the existence of $L^2$ solution to the linear BGK equation:
\begin{align}\label{EQN: sec 3 linear BGK}
    \left\{\begin{array}{ll}
          \{\p_t + v\cdot\nabla _x -L_\gamma \} f=  g ,&   \\
         f|_{t=0}= f_{in}(x,v),& \\
          f|_{S_-}= f_-(t,x,v),& \\
    \end{array}\right.
\end{align}
where by the definition of $\mathbf{P}=\sum^{5}_{i=1} \langle f,e_i\rangle e_i $, we write
$$
    \mathbf{P}f = \int_{\R^3} k(v,v')f(v')\, dv', 
\quad\hbox{ with }
    k(v,v') :=  \sqrt{\mu (v)}\sqrt{\mu (v')} \bigg(1+v\cdot v'+\frac{(|v|^2-3)(|v'|^2-3)}{6}\bigg).
$$

We first show some important properties of the operators $L_\gamma=\gamma(x)(\mathbf{P}-\mathbf{I})$ and $\Gamma f$.
\begin{lemma} \label{lemma:coercivity}
Let $f(t)\in L^2_{x,v}$ and $\gamma \in L^\infty(\Omega)$. The operator $L_\gamma$ satisfies the coercivity property:
\begin{align*}
    -\langle L_\gamma f, f\rangle_{x,v}  =  \int_\Omega \gamma(x) \|(\mathbf{P}-\mathbf{I})f\|^2_{L^2_{v}}\,dx\geq c_0 \|(\mathbf{P}-\mathbf{I})f\|^2_{L^2_{x,v}},
\end{align*}
provided that $\gamma(x)\geq c_0$ for some positive constant $c_0$.
Moreover, $\Gamma (f)$ is in the kernel of $\mathbf{P}$, that is,
    $$
    \mathbf{P}(\Gamma(f))=0.
    $$
\end{lemma}
 
\begin{proof} 
Since $\{e_i\}$ forms an orthonormal basis, the orthoganality of $\langle \mathbf{P} f, (\mathbf{I}-\mathbf{P})f\rangle_{v}  =0$ is satisfied. Using this, we get
$$
 \langle (\mathbf{P}-\mathbf{I})f , f\rangle_{v} 
=\langle (\mathbf{P}-\mathbf{I})f , \mathbf{P} f+  (\mathbf{I}-\mathbf{P})f\rangle_{v} =-\langle (\mathbf{P}-\mathbf{I})f , (\mathbf{P}-\mathbf{I})f \rangle_{v} .
$$
From the definition of $L_\gamma$, we can then derive
\begin{align*}
    -\langle L_\gamma f, f\rangle_{v} 
     =  -\gamma (x)\langle (\mathbf{P}-\mathbf{I})f , f\rangle_{v}   
     =\gamma(x)\|(\mathbf{P}-\mathbf{I})f\|^2_{L^2_{v}}\geq c_0 \|(\mathbf{P}-\mathbf{I})f\|^2_{L^2_{v}}.
\end{align*}

Note that $\mathbf{P}^2=\mathbf{P}$, which implies $\mathbf{P}(L_\gamma (f))= \gamma\mathbf{P}(\mathbf{P}-\mathbf{I})f=\gamma(\mathbf{P}^2f-\mathbf{P} f)=0$. Recall that $( \sqrt{\mu})^{-1}q(M(F)-F)=L_\gamma f+\Gamma f$. Due to the conservation of mass, momentum, and energy in \eqref{ID:conservation}, we can derive 
    \begin{align*}
         \mathbf{P}(\Gamma(f))&=\mathbf{P}\left( {q(t,x,F) (M(F) -F) \over  \sqrt{\mu}} \right) -  \mathbf{P}(L_\gamma (f)) \\
         &=   q(t,x,F)  \mathbf{P}\left({(M(F) -F) \over \sqrt{\mu}} \right)\\
         &= q\sum_{k=1}^{3} \LC\int_{\R^3}  (M(F )-F ){e_k\over \sqrt{\mu}}\, dv \RC e_k=0.
    \end{align*}
Here we used the fact that $q=\gamma \rho^\alpha T^\beta$ is independent of $v$.    
\end{proof}

\subsection{The $L^2$ solutions}
We introduce the following crucial Green's identity, whose proof can be found in \cite[Chapter 9, Eq. (2.18)]{Cercignani1994book}.
\begin{lemma}
    Let $t^*>0$. Assume that $f(t,x,v)$, $\psi(t,x,v)\in L^\infty([0,t^*];L^2(\Omega\times\R^3))$, and $\p_t f+v\cdot\nabla_x f$, $\p_t \psi+v\cdot\nabla_x \psi\in L^2([0,t^*]\times\Omega\times\R^3)$. Suppose that $f|_{\p \Omega\times \R^3},\, \psi|_{\p \Omega\times \R^3}\in L^2([0,t^*]\times\p\Omega\times\R^3)$. Then for almost all $t \in [0,t^*]$, 
    \begin{align}\label{Green's ID}
    &\int_0^t \int\int_{\Omega\times\R^3} ( \p_t f+v\cdot\nabla_x f )\psi\, dvdxds + \int_0^t \int\int_{\Omega\times\R^3} ( \p_t\psi+v\cdot\nabla_x \psi )f\, dvdxds  \notag\\
    &=\int\int_{\Omega\times\R^3} \LC f(t)\psi(t)-f(0)\psi(0)\RC \,dvdx+\int_0^t \int_S f\psi\,d\Sigma ds.
    \end{align}
\end{lemma}

Now we are ready to state and show the well-posedness result for the problem~\eqref{EQN: sec 3 linear BGK}.
\begin{proposition}\label{prop: L2 estimate of boltzmann}
      Let $\Omega$ be a bounded domain in $\R^3$ with $C^1$ boundary. Suppose that $\gamma \in L^\infty(\Omega)$ satisfies $\gamma \geq c_0$ almost everywhere for some positive constant $c_0$.
      Suppose that 
      $$\mathbf{P}g(t,x,v)=0\quad (t,x,v)\in (0,t^*)\times \Omega\times\R^3.$$ 
      If the initial data $f_{in}$, in-flow data $f_-$, and $g$ satisfy
      \begin{align}\label{prop:L2 data}
                \|f_{in}\|^2_{L^2_{x,v}} +\int_0^{t^*}  \|f_-(s)\|^2_{L^2(S_{-})}\, ds+\int_0^{t^*}  \|g(s)\|^2_{L^2_{x,v}}\,ds<\infty,
      \end{align}
      then there exists a unique solution $f\in L^2$ to 
      \begin{align}\label{EST:L2 existence for f}
        \{\p_t + v\cdot\nabla _x-L_\gamma\} f=g,\quad  f|_{t=0}=f_{in},\quad f|_{S_-}=f_-.
      \end{align}
      
      Moreover, there exists a constant $C>0$ depending on $\gamma$ such that for almost all $t\in [0,t^*]$, 
      \begin{align}\label{EST:L2 estimate for f}
          \|f(t)\|^2_{L^2_{x,v}} \leq C\LC\|f_{in}\|_{L^2_{x,v}}^2 + \int^t_0  \|f_-(s)\|^2_{L^2(S_-)}\, ds + \int^t_0 \| g(s)\|^2_{L^2_{x,v}}\,ds\RC .
      \end{align}
\end{proposition}

\begin{proof} 
{\bf Step 1. Existence of $L^2$ solution.}
The existence of such $L^2$ solution will be established via the approximating sequence $\{f^m\}_{m\geq 0}$ with $f^{-1}=0$:
\begin{align*}
    \left\{\begin{array}{ll}
          \{\p_t + v\cdot\nabla_x +\gamma(x) \} f^{m} =  \gamma(x)\mathbf{P} f^{m-1} +g,&   \\
          f^{m}|_{t=0}=f_{in} ,& \\
          f^{m}|_{S_-}= f_- .& \\
    \end{array}\right.
\end{align*}

Write $\tilde f^m=f^m-f^{m-1}$ for $m\geq 1$, then $\tilde f^m$ satisfies
$$\{\p_t+v\cdot\nabla_x+\gamma(x)\}\tilde f^m=\gamma(x)\mathbf P \tilde f^{m-1},\quad \mbox{with}\quad \tilde f^{m}|_{t=0}=0, \quad \tilde f^m|_{S_-}=0.$$
Note that by Young's inequality, for $0<\delta<1$,
$$
    \langle \mathbf{P} \tilde f^{m-1},\tilde f^{m}\rangle_{x,v} (t)\leq \delta^{-1}\| \mathbf{P} \tilde f^{m-1}(t)\|^2_{L^2_{x,v}}+\delta \|  \tilde f^{m}(t)\|^2_{L^2_{x,v}}
    \lesssim \delta^{-1}\|\tilde f^{m-1}(t)\|^2_{L^2_{x,v}}+\delta \|\tilde f^{m}(t)\|^2_{L^2_{x,v}}.
$$
The Green's identity implies that for $0<\delta<1$,
\begin{align*} 
        &{1\over 2}\|\tilde f^{m}(t)\|^2_{L^2_{x,v}} + c_0 \int^t_0\| \tilde f^{m}(s)\|^2_{L^2_{x,v}}\,ds+ {1\over 2}\int^t_0 \| \tilde f
        ^{m}(s)\|^2_{L^2(S_+)}\, ds  \notag\\
        &\lesssim \int_0^t \delta^{-1}\|\tilde f^{m-1}(s)\|^2_{L^2_{x,v}}\,ds+\int_0^t \delta \|\tilde f^{m}(s)\|^2_{L^2_{x,v}}\, ds,
\end{align*}
where $A \lesssim B$ means that $A\leq CB$ for some constant $C>0$ that only depends on $\gamma$. 
For $\delta$  sufficiently small, the second term on the right-hand side (RHS) can be absorbed by the left-hand side (LHS), and thus we obtain
$$\|\tilde f^m(t)\|^2_{L^2_{x,v}}\lesssim \delta^{-1}\int_0^t \|\tilde f^{m-1}(s)\|^2_{L^2_{x,v}}\,ds.$$
A recurrence argument gives for $m\geq 1$,
$$\|\tilde f^m(t)\|^2_{L^2_{x,v}}\lesssim {1 \over m!} \LC{t\over \delta}\RC^m \sup_{0\leq t\leq t^*}\|f^{0}(t)\|^2_{L^2_{x,v}}.$$
Let $K=t/\delta$. Since $e^K=\sum (K^m/m!)$ converges for all $K$, we deduce that $\{f^m\}_{m=0}^\infty$ is a Cauchy sequence in the space $L^\infty([0,t^*];L^2(\Omega\times\R^n))$ and it must converges to a function $f$. This shows the existence of solution $f$ to the problem~\eqref{EST:L2 existence for f}.

{\bf Step 2. $L^2$ Estimate.}
    By Green's identity \eqref{Green's ID} and Lemma~\ref{lemma:coercivity}, we get the $L^2$ energy estimate: 
    \begin{align}\label{EST:L^2 scale 1}
        &\|f(t)\|^2_{L^2_{x,v}} + \int^t_0\|(\mathbf{I}-\mathbf{P})f(s)\|^2_{L^2_{x,v}}\,ds + \int^t_0 \|f(s)\|^2_{L^2(S_+)}\, ds   \notag\\
        &\lesssim  \|f(0)\|^2_{L^2_{x,v}}+\int_0^t \int\int_{\Omega\times\R^3}fg\, dvdxds +\int^t_0 \|f(s)\|^2_{L^2(S_-)}\, ds. 
    \end{align}  
    We apply $\mathbf{P} g=0$ and $\langle f,\mathbf{P}h\rangle_v=\langle \mathbf{P}f,h\rangle_v$ to derive   
    \begin{align*}   
        \LV \int_0^t \int\int_{\Omega\times\R^3} fg\, dvdxds \RV
        &= \LV\int_0^t \int\int_{\Omega\times\R^3} f(\mathbf{I}-\mathbf{P})g\, dvdxds \RV \\
        &=\LV\int_0^t \int\int_{\Omega\times\R^3} g(\mathbf{I}-\mathbf{P})f  \, dvdxds \RV \\
        &\lesssim \delta^{-1} \int_0^t \|g\|_{L^2_{x,v}}^2 \,ds + \delta \int^t_0 \| (\mathbf{I}-\mathbf{P})f(s)\|_{L^2_{x,v}}^2 \,ds,\quad \delta>0. 
    \end{align*}
    For small $\delta$, as the term $\delta \int^t_0 \|(\mathbf{I}-\mathbf{P})f\|_{L^2_{x,v}}^2 \,ds$ can be absorbed by the left-hand side of \eqref{EST:L^2 scale 1}, we can further deduce from \eqref{EST:L^2 scale 1} that
        \begin{align*} 
        & \|f(t)\|^2_{L^2_{x,v}} + \int^t_0\|(\mathbf{I}-\mathbf{P})f(s)\|^2_{L^2_{x,v}}\,ds +\int^t_0 \|f(s)\|^2_{L^2(S_+)}\, ds\\
         &\lesssim \|f(0)\|^2_{L^2_{x,v}} + \int^t_0 \|g(s)\|^2_{L^2_{x,v}}\,ds+\int^t_0 \|f(s)\|^2_{L^2(S_-)}\, ds,
    \end{align*} 
    which completes the proof of stability.

   {\bf Step 3. Uniqueness.}
    Suppose there are two $L^2$ solutions $f$ and $\tilde{f}$ to the problem~\eqref{EST:L2 existence for f}. Then we have
    $$
     \{\p_t + v\cdot\nabla _x-L_\gamma\} (f-\tilde{f})=0,\quad  (f-\tilde{f})|_{t=0}=0,\quad (f-\tilde{f})|_{S_-}=0.
    $$
    Since the source term, initial and boundary conditions are trivial, applying \eqref{EST:L2 estimate for f} yields that $\|(f- \tilde{f})(t)\|^2_{L^2_{x,v}} =0$, which implies $f=\tilde{f}$ almost everywhere.    
\end{proof}

\section{Existence of bounded solutions of the BGK}\label{sec:bounded solution for the linear BGK}
In this section, we will prove a local existence result for the solution of the nonlinear BGK in Theorem~\ref{theorem:main forward} in a weighted $L^\infty$ norm.
 
\subsection{Boundedness of the solution to the linearized BGK}
Now we establish the uniqueness of bounded solutions in a weighted space by considering the following problem for the linear BGK:
\begin{align}\label{EQN:h}
    \left\{\begin{array}{ll}
          \{\p_t + v\cdot\nabla _x +\gamma(x)(\mathbf{I}-\mathbf{P}_w)\} h = wg ,&   \\
         h|_{t=0}=w f_{in} ,& \\
         h|_{S_-}=w f_- ,& \\
    \end{array}\right.
\end{align}
where the function $w(v) = (1+ c_1|v|^2)^{c_2}$ satisfies $w^{-2}\in L^1_v$ with constants $c_1,\,c_2>0$.
We define a new operator $\mathbf{P}_w$ as follows:
\begin{align}\label{DEF:Pw}
    \mathbf{P}_w h(t,x,v):=w(v)\mathbf{P}\LC{h\over w}\RC(t,x,v)=\int K_w(v,v')h(v')\, dv',\quad\hbox{with }K_w(v,v'):=k(v,v'){w(v)\over w(v')},
\end{align}
by recalling the definition of the macroscopic projection $\mathbf{P}$:
$$
    \mathbf{P}f = \int k(v,v')f(v')\, dv', 
\quad\hbox{ with }
    k(v,v') :=  \sqrt{\mu (v)}\sqrt{\mu (v')} \bigg(1+v\cdot v'+\frac{(|v|^2-3)(|v'|^2-3)}{6}\bigg).
$$
We have $K_w(v,\cdot),\,K_w(\cdot,v')\in L^1_v$, see Lemma~\ref{lemma:K}. 
  
    \begin{proposition}\label{prop:L infty estimate}
    Let $\Omega$ be a convex, open and bounded domain with smooth boundary.  Suppose that 
      $$\mathbf{P}g(t,x,v)=0\quad (t,x,v)\in (0,t^*)\times \Omega\times\R^3.$$ 
    Suppose that the functions $f_{in}$, $f_-$, and $g$ satisfy
    \begin{align}\label{prop:Linfty data}
    \|wf_{in}\|_{L^\infty_{x,v}}  + \sup_{0\leq t\leq t^*} \|wf_-(t)\|_{L^\infty_{x,v}}+\sup_{0\leq t\leq t^*} \| wg(t)\|_{L^\infty_{x,v}}<\infty.
    \end{align}
    Then there exists a unique $L^\infty$ solution $h$ to the problem~\eqref{EQN:h}.

    Moreover, we have the estimate 
    \begin{align}
        \sup_{0\leq t\leq t^*}  \|h(t)\|_{L^\infty_{x,v}} \leq C \LC\|wf_{in}\|_{L^\infty_{x,v}} + \sup_{0\leq t\leq t^*}\|wf_-(t)\|_{L^\infty_{x,v}}+\sup_{0\leq t\leq t^*} \| wg(t)\|_{L^\infty_{x,v}}\RC,
    \end{align}
    where the constant $C>0$ depends on $t^*,\,w$, and $\gamma$.
\end{proposition}
 
\begin{remark}\label{remark:L2-Linfty} The weight function $w$ is used to control the $L^2$-norm by the weighted $L^\infty$-norm.  
Specifically, by applying $w^{-2}\in L^1_v$, one gets
\begin{align}\label{L2toLinfty}
\|g(s)\|_{L^2_{x,v}}=\left\| {w \over w}g(s)\right\|_{L^2_{x,v}}\leq C\|wg\|_{L^\infty_{x,v}} \LC\int w^{-2}\,dv \RC^{1/2}\leq C\|wg\|_{L^\infty_{x,v}}.
\end{align}  
Similar argument also leads to $\|f_{in}\|_{L^2_{x,v}} \leq C\|w f_{in}\|_{L^\infty_{x,v}}$ and $\|f_-(s)\|_{L^2_{x,v}}\leq C \|wf_-(s)\|_{L^\infty_{x,v}}$.
Hence, under the hypothesis \eqref{prop:Linfty data}, the condition \eqref{prop:L2 data}  holds in Proposition~\ref{prop: L2 estimate of boltzmann}, which suggests that   there exists a unique solution $f\in L^2$ to
\begin{align*} 
        \{\p_t + v\cdot\nabla _x-L_\gamma\} f=g,\quad  f|_{t=0}=f_{in},\quad f|_{S_-}=f_-.
\end{align*}
In fact, the existence of $h\in L^\infty$ implies $h/w\in L^2$ satisfies the same equation. Applying the uniqueness of the $L^2$ solution yields $f=h/w$. 
\end{remark} 
\begin{proof}
{\bf Step 1. Existence of weighted $L^\infty$ solution.}
The existence of such $L^\infty$ solution will be established via the approximating sequence $\{h^m\}_{m\geq 0}$ with $h^{-1}\equiv 0$:
\begin{align*}
      \{\p_t + v\cdot\nabla_x +\gamma(x) \}h^{m } = \gamma(x)\mathbf{P}_w h^{m-1} + wg , \quad \hbox{with }  h^{m }|_{t=0}=w f_{in},\quad h^{m }|_{S_-} = w f_-. 
\end{align*}
We show that $h^{m}$ is a Cauchy sequence and thus, converges to $h$.

We apply the method of characteristic to express the solution as
\begin{align}\label{characteristic h}
     h^m (t,x,v) &=\mathbf{1}_{t-\tau_-\leq 0}\, w f_{in}(x-vt,v)\exp\LC-\int^t_0 \gamma(x-sv)\,ds\RC &\Big\} I_1\notag\\
     &\quad + \mathbf{1}_{t-\tau_->0}\, \exp\LC-\int^{\tau_-(x,v)}_0 \gamma(x-sv)\,ds \RC wf_-(t-\tau_-(x,v),x-v\tau_-(x,v),v) &\Big\} I_2\notag\\
     &\quad + \int^t_{\max\{0,t-\tau_-\}} \exp\LC-\int^{t-s}_0 \gamma(x-\tau v)\,d\tau\RC wg(s,x-v(t-s),v)\,ds   &\Big\} I_3  \notag\\
     &\quad + \int^t_{\max\{0,t-\tau_-\}} \exp\LC-\int^{t-s}_0 \gamma(x-\tau v)\,d\tau\RC [\gamma \mathbf{P}_w(h^{m-1} )](s,x-v(t-s),v)\,ds. &\Big\} I_4 
\end{align}
Here $\mathbf{1}_{x\leq 0}$ is the Heaviside function, which is $1$ if $x<0$, and is $0$ if $x>0$.
Suppose that $\gamma\geq c_0>0$.  
The first two terms can be estimated as follows:
    \begin{align}\label{EST:I1-I2}
        |I_1+I_2 | \leq \|w f_{in} \|_{L^\infty_{x,v}}  + \sup_{0\leq t\leq t^*} \|wf_-(t)\|_{L^\infty_{x,v}}.
    \end{align}
    Also, for the third term, we obtain 
    \begin{align}\label{EST:I3}
        |I_3|\leq \sup_{0\leq t\leq t^*} \|wg(t)\|_{L^\infty_{x,v}} \int^t_{\max\{0,t-t_-\}} e^{-(t-s)c_0} \,  ds\leq{1\over c_0}\sup_{0\leq t\leq t^*} \|wg(t)\|_{L^\infty_{x,v}}.
    \end{align}
Using the fact that $\|\mathbf{P}_w(h^{m-1} )(t)\|_{L^\infty_{x,v}}\leq C\|h^{m-1}(t)\|_{L^\infty_{x,v}}$, we obtain
\begin{align}\label{EST:I4 0}
        |I_4|\leq C\int^t_0\|h^{m-1}(s)\|_{L^\infty_{x,v}}\,ds.
\end{align}
With these estimates \eqref{EST:I1-I2}-\eqref{EST:I4 0}, from \eqref{characteristic h}, it yields that 
\begin{align*}
         \|h^{m} (t)\|_{L^\infty_{x,v}}
      &\leq \|wf_{in}\|_{L^\infty_{x,v}} + \sup_{0\leq t\leq t^*}\|wf_-(t)\|_{L^\infty_{x,v}}+C\sup_{0\leq t\leq t^*} \| wg(t)\|_{L^\infty_{x,v}} 
      +C\int^t_0\|h^{m-1}(s)\|_{L^\infty_{x,v}}\,ds.
\end{align*}

Write $\tilde{h}^{m}=h^{m}-h^{m-1}$ for $m\geq 1$. Then $\tilde{h}^{m}$ satisfies
\begin{align*}
      \{\p_t + v\cdot\nabla_x +\gamma(x) \}\tilde{h}^{m} = \gamma(x)\mathbf{P}_w\tilde{h}^{m-1}, \quad \hbox{with } \tilde{h}^{m}|_{t=0}= 0,\quad \tilde{h}^{m}|_{S_-} = 0 
\end{align*}
and, moreover,
\begin{align*}
         \|\tilde{h}^{m}(t)\|_{L^\infty_{x,v}}
      &\leq C \int^t_0 \|\tilde{h}^{m-1}(s)\|_{L^\infty_{x,v}} \,ds.
\end{align*}
A recurrence gives for $m\geq 1$ that
\begin{align*}
        \|\tilde{h}^{m}(t)\|_{L^\infty_{x,v}} 
      &\leq  {(Ct)^{m}\over m!} \sup_{0\leq t\leq t^*} \|h^0 (t)\|_{L^\infty_{x,v}} .
\end{align*}
Let $K=Ct^*$. We have
\begin{align*}
      \|h^{m}-h^{m-1}\|_{L^\infty} 
      &\leq  {K^{m}\over m!} \|h^0 \|_{L^\infty} .
\end{align*}
Since $\sum (K^m/m!) <\infty$, we deduce that $h^m$ is a Cauchy sequence in the space $L^\infty$, and it implies $h^m$ converges to a function $h$ in $L^\infty$. Applying \eqref{L2toLinfty} to deduce $h/w\in L^2$, and thus, the unique existence of $L^2$ solution stated in Proposition~\ref{prop: L2 estimate of boltzmann} yields that such solution $h/w$ satisfies the estimate \eqref{EST:L2 estimate for f}.

\noindent{\bf Step 2. Derivation of stability estimate.} Similar to \eqref{characteristic h}, we apply the method of characteristic to express the solution as a form of the integral identity:
\begin{align}\label{characteristic h 2}
     h (t,x,v) &=\mathbf{1}_{t-\tau_-\leq 0}\, w f_{in}(x-vt,v)\exp\LC-\int^t_0 \gamma(x-sv)\,ds\RC &\Big\} I_1\notag\\
     &\quad + \mathbf{1}_{t-\tau_->0}\, \exp\LC-\int^{\tau_-(x,v)}_0 \gamma(x-sv)\,ds \RC wf_-(t-\tau_-(x,v),x-v\tau_-(x,v),v) &\Big\} I_2\notag\\
     &\quad + \int^t_{\max\{0,t-\tau_-\}} \exp\LC-\int^{t-s}_0 \gamma(x-\tau v)\,d\tau\RC wg(s,x-v(t-s),v)\,ds   &\Big\} I_3  \notag\\
     &\quad + \int^t_{\max\{0,t-\tau_-\}} \exp\LC-\int^{t-s}_0 \gamma(x-\tau v)\,d\tau\RC [\gamma \mathbf{P}_w(h )](s,x-v(t-s),v)\,ds. &\Big\} I_4^h
\end{align}
As the estimates of $I_1$-$I_3$ are the same as in the step 1, we only need to control $I_4^h$.

For this purpose, we first define $x_1=x-v(t-s)$, $\tau_-' = \tau_-(x_1,v')$. By iterating the formula \eqref{characteristic h 2} to evaluate $ \mathbf{P}_w(h )$, we get
\begin{align}\label{Iterate h}
      &[\gamma \mathbf{P}_w(h )](s,x-v(t-s),v) = \int K_w(v,v') \gamma (x_1)h(s ,x_1,v')dv' \notag\\
      &=\int K_w(v,v') \gamma(x_1) \LC \mathbf{1}_{s-\tau'_-\leq 0}\, w f_{in}(x_1-vs,v)e^{ -\int^s_0 \gamma(x_1-\tau v')\,d\tau } \RC\,dv'  &\Big\}{\bf A_1}\notag\\
      &\quad +  \int K_w(v,v')\gamma(x_1) \mathbf{1}_{s-\tau'_->0}\, e^{ -\int^{\tau_-(x_1,v')}_0 \gamma(x_1-\tau v')\,d\tau}  wf_-(s-\tau_-' ,x_1-v'\tau_-',v')\,dv'&\Big\}{\bf A_2} \notag\\
      & \quad +\int K_w(v,v')\gamma(x_1) \int^s_{\max\{0,s-\tau'_-\}}e^{-\int_0^{s-s'}\gamma(x_1-\tau v')\,d\tau} w g(s',x_1-v'(s-s'),v')\,ds'dv' &\Big\} {\bf A_3}\notag\\
      &\quad +  \int \int^s_{\max\{0,s-\tau'_-\}}e^{-\int^{s-s'}_0 \gamma(x_1-\tau v')\,d\tau} \int K_w(v,v') K_w(v',v{''}) \gamma (x_1)\gamma  h  (s',x_1-v'(s-s'),v{''}) dv{''} ds'dv'. &\Big\} {\bf A_4} 
\end{align}
By substituting \eqref{Iterate h} into $I_4^h$ of $h(t,x,v)$, we get 
\begin{align}\label{EST:I4}
    I_4^h= \int^t_{\max\{0,t-\tau_-\}} e^{-\int^{t-s}_0 \gamma(x-\tau v)\,d\tau}\LC {\bf A_1} +\ldots + {\bf A_4}\RC\,ds.
\end{align}
All the terms above on the right, except the last term, are bounded by 
\begin{align}\label{EST:I4:A1-A3}
    &\int^t_{\max\{0,t-\tau_-\}} e^{-\int^{t-s}_0 \gamma(x-\tau v)\,d\tau}\LC {\bf A_1} +{\bf A_2}+ {\bf A_3}\RC\,ds \notag\\
    &\lesssim t^* \LC \|w f_{in} \|_{L^\infty_{x,v}}  + \sup_{0\leq s\leq t^*} \|wf_-(s)\|_{L^\infty_{x,v}} +\sup_{0\leq s\leq t^*} \|wg(s)\|_{L^\infty_{x,v}}\RC,
\end{align}
where we used the fact that $\int K_w(v,v')dv'<\infty$ in Lemma~\ref{lemma:K}.
Therefore, it remains to estimate the last term in $I_4^h$: for a fixed small $\xi>0$,  and $(t,x)\in [0,t^*]\times \overline\Omega$,
 
\begin{itemize}
    \item {\bf Case 1.} If $\xi >t-\tau_-(x,v)$, then we split the integral into 
\begin{align*}
 \hbox{Last term in $I_4^h$}
 &=\int^t_{\max\{0,t-\tau_-\}} e^{-\int^{t-s}_0 \gamma(x-\tau v)\,d\tau} {\bf (A_4)} \,ds\\
  &=\int^t_{\max\{0,t-\tau_-\}} \int \int^s_{\max\{0,s-\tau'_-\}} e^{-\int^{t-s}_0 \gamma(x-\tau v)\,d\tau}e^{ -\int^{s-s'}_0 \gamma(x_1-\tau v')\,d\tau} \\
       &\quad \cdot\int K_w(v,v') K_w(v',v{''}) \gamma \gamma  h  (s',x_1-v' (s-s') ,v{''})\,dv{''} ds'dv'ds\\
 &= \int^{ \xi}_{\max\{0,t-\tau_-\}}\int\int^s_{\max\{0,s-\tau'_-\}}  \cdots \,ds'dv'ds+\int^t_{\xi}\int \LC\int^s_{s-\xi} + \int^{s-\xi}_{\max\{0,s-\tau'_-\}}\RC \cdots \,ds'dv'ds;
\end{align*}
 
    \item {\bf Case 2.} If $0<\xi <t-\tau_-(x,v)$, then $s-\xi>0$ automatically holds as $t-\tau_-<s<t$. We obtain
\begin{align*}
 \hbox{Last term in $I_4^h$}
 &=\int^t_{\max\{0,t-\tau_-\}} e^{-\int^{t-s}_0 \gamma(x-\tau v)\,d\tau} {\bf (A_4)} \,ds\\
 &=\int^{t}_{\max\{0,t-\tau_-\}} \int \LC \int^s_{s-\xi}+ \int^{s-\xi}_{\max\{0,s-\tau'_-\}}\RC\cdots\,ds'dv'ds=:{\bf B_{1}}+{\bf B_{2}}.
\end{align*}    
\end{itemize}

We will only discuss Case 2, as similar arguments also work for Case 1.  \\
\noindent{\bf Step 2-1. Analyze ${\bf B_{1}}$ in last term in $I_4^h$.} Since the kernel $K_w$ is in $L^1$, we derive  
\begin{align}\label{EST:B1}
    |{\bf B_{1}}|&\lesssim   \sup_{0\leq t\leq t^*}\| h(t)\|_{L^\infty_{x,v}}\int^t_{\max\{0,t-\tau_-\}} \int^s_{s-\xi}\int_{\R^3}\int_{\R^3} |K_w(v,v')|\cdot |K_w(v',v'')| \,dv' dv''ds' ds \notag \\
    &\lesssim \xi t^*\sup_{0\leq t\leq t^*}\| h(t)\|_{L^\infty_{x,v}}.
\end{align}

\noindent{\bf Step 2-2. Analyze ${\bf B_{2}}$ in last term in $I_4^h$.} Let $D=\{|v'|\geq N,\,\hbox{or }|v''|\geq N\}$. We split ${\bf B_2}$ into
$$
    |{\bf B_{2}}|= \int^t_{\max\{0,t-\tau_-\}} \LC\int_D +\int_{(\R^3\times\R^3)\setminus D
    } \RC\int^{s-\xi}_{\max\{0,s-\tau'_-\}} \cdots\,  ds'dv'ds=:{\bf B_{2,D}}+{\bf B_{2,D^c}}.
$$
To estimate ${\bf B_{2,D}}$, in the region $|v'|\geq  N$ or $|v''|\geq  N$, we have by Lemma~\ref{lemma:K}, when $\theta,\,\delta>0$ satisfy $\theta+\delta <1/4$, 
\begin{align*}
   &  \LV\int_{D}K_w(v,v') K_w(v',v'') \,dv' dv'' \RV \\
   &\leq \int_{D} Ce^{\LC-{1\over 4} +\theta+\delta\RC |v|^2}e^{\LC-{1\over 4} -\theta+\delta\RC |v'|^2}Ce^{\LC-{1\over 4} +\theta+\delta\RC |v'|^2}e^{\LC-{1\over 4} -\theta+\delta\RC |v''|^2}\,dv'dv'' \lesssim e^{-CN^2}.
\end{align*}
With this estimate, we can deduce
\begin{align*}
    {\bf B_{2,D}}&\leq\int^t_{\max\{0,t-\tau_-\}} \int_D \int^{s-\xi}_{\max\{0,s-\tau'_-\}}     | K_w(v,v')K_w(v',v'')|   |h  (s',x_1-v'(s-s'),v{''})|
        \,ds'dv'dv''ds\\
     &\lesssim e^{-CN^2} ((t^*)^2+t^*)\sup_{0\leq t\leq t^*}\|  h(t)\|_{L^\infty_{x,v}}.
\end{align*}
To estimate ${\bf B_{2,D^c}}$, in the region $|v'|< N$ and $|v''|<N$, we use the change of variable $v'\mapsto y=x_1-(s-s')v'\in \Omega$ with Jacobian 
$$
    \LV \det\LC {\partial y\over \partial v'}\RC\RV=(s-s')^3\geq \xi^3.
$$
\begin{align*}
    {\bf B_{2,D^c}} &\leq\int^t_{\max\{0,t-\tau_-\}} \int_{|v''|<N}   \int_{|v'|<N}\int^{s-\xi}_{\max\{0,s-\tau'_-\}}| K_w(v,v')K_w(v',v'')| |h  (s',x_1-v' (s-s') ,v{''})|  \,ds'dv'dv''ds\\
    &\lesssim C_N{1\over \xi^{3}}\int^t_{\max\{0,t-\tau_-\}}\int^{s-\xi}_0 \LC  \int_{|v''|<N} \int_{\Omega}  | f  (s',y,v{''})|^2 dydv'' \RC^{1/2} \,ds'ds\\
    &\quad (\hbox{here we write $h=w (h/w)=wf $ and use H\"older estimate})\\
     &\lesssim  C_N   ((t^*)^2+t^*){1\over \xi^{3}}  \sup_{0\leq t\leq t^*}\| f(t)\|_{L^2_{x,v}}\quad (\hbox{where $f=h/w$ is a $L^2$ solution})\\
     &\lesssim  C_N  ((t^*)^2+t^*) {1\over \xi^{3}} \Bigg (  \|f_{in}\|_{L^2_{x,v}} + \LC\int^t_0  \|f_-(s)\|^2_{L^2(S_-)}\, ds \RC^{1/2}+ \LC\int^t_0 \| g(s)\|^2_{x,v}\,ds\RC^{1/2} \Bigg) \quad \hbox{(used \eqref{EST:L2 estimate for f})}\\
     &\leq  C_{N,t^*}  {1\over \xi^{3}} \LC\|w f_{in}\|_{L^\infty_{x,v}} + \sup_{0\leq t\leq t^*} \|wf_-(s)\|_{L^\infty_{x,v}}+\sup_{0\leq t\leq t^*} \| wg(s)\|_{L^\infty_{x,v}}\RC \quad \hbox{(used \eqref{L2toLinfty})}. 
\end{align*}

From the estimates for ${\bf B_{2,D}},\, {\bf B_{2,D^c}}$ above, we obtain
\begin{align}\label{EST:B2} 
    |{\bf B_{2}}|& \leq |{\bf B_{2,D}}|+ |{\bf B_{2,D^c}}| \notag\\
    &\lesssim e^{-CN^2} \sup_{0\leq t\leq t^*}\| h(t)\|_{L^\infty_{x,v}} + \LC\|w f_{in}\|_{L^\infty_{x,v}} + \sup_{0\leq t\leq t^*} \|wf_-(t)\|_{L^\infty_{x,v}}+\sup_{0\leq t\leq t^*}\| wg(t)\|_{L^\infty_{x,v}}\RC.
\end{align}
Finally, by combining estimates \eqref{EST:I1-I2}, \eqref{EST:I3}, \eqref{EST:I4}, \eqref{EST:I4:A1-A3}, \eqref{EST:B1}, and \eqref{EST:B2}, it follows that for $t\in [0,t^*]$,
\begin{align*}
      h(t,x,v)\lesssim (e^{-CN^2}+\xi) \sup_{0\leq t\leq t^*}\| h(t)\|_{L^\infty_{x,v}}+\LC\|w f_{in}\|_{L^\infty_{x,v}}+ \sup_{0\leq t\leq t^*} \|wf_-(t)\|_{L^\infty_{x,v}}+\sup_{0\leq t\leq t^*}\| wg(t)\|_{L^\infty_{x,v}}\RC.
\end{align*}
Note that only ${\bf B_{1}}$ and ${\bf B_{2}}$ have the term $\sup_{0\leq t\leq t^*}\| h(t)\|_{L^\infty_{x,v}}$ in their estimates.
By choosing small $\xi$ and large $N>0$, the term $\sup_{0\leq t\leq t^*}\| h(t)\|_{L^\infty_{x,v}}$ on the right-hand side can be absorbed to the left-hand side, and thus we complete the proof.

{\bf Step 3. Uniqueness.}
To show uniqueness of solution, it follows by applying a similar argument as in the proof of Proposition~\ref{prop: L2 estimate of boltzmann} and thus we omit the proof here. 
\end{proof}

\subsection{Local existence of the nonlinear BGK - Proof of Theorem~\ref{theorem:main forward}}
After establishing several lemmas and propositions that ensure the unique existence of weighted bounded solutions $f=h/w$ to the linear equation:
\begin{align*} 
    \{\p_t + v\cdot\nabla _x -L_\gamma\} f=  g,    
\end{align*} 
where $h$ is the solution to \eqref{EQN:h}.
Now we are ready to show the existence of the solution to the nonlinear BGK in the weighted $L^\infty$ space.

\begin{proof}[Proof of Theorem~\ref{theorem:main forward}]
To this end, it is sufficient to prove the well-posedness of the problem:
\begin{align*}
    \{\p_t + v\cdot\nabla _x +\gamma(x)(\mathbf{I}-\mathbf{P}_w)\} h = w\Gamma \LC{h\over w}\RC ,\quad \hbox{with } h|_{t=0}=w f_{in},\quad h|_{S_-} =w f_-. 
\end{align*}
 
\noindent {\bf Step 1. Iteration.}   
Let $h^{-1}\equiv 0$. We consider the following iteration sequence $h^m$ for $m\geq 0$:
\begin{align*}
     \{\p_t + v\cdot\nabla_x +\gamma(x)(\mathbf{I}-\mathbf{P}_w) \} h^{m} =  w\Gamma \LC{h^{m-1}\over w}\RC, \quad\hbox{with }  h^{m}|_{t=0}=w f_{in},\quad h^{m}|_{S_-} =w f_-.
\end{align*}

Since $h^{-1}=0$ implies $\Gamma\LC{h^{-1}\over w}\RC=0$, based on Proposition~\ref{prop:L infty estimate}, there exists a unique solution $h^0$ so that 
\begin{align}\label{EST: theorem h0}
    \sup_{0\leq t\leq t^*} \|h^0(t)\|_{L^\infty_{x,v}}\leq  C_1\LC\|w f_{in}\|_{L^\infty_{x,v}}+\sup_{0\leq t\leq t^*} \|wf_-(t)\|_{L^\infty_{x,v}}\RC<C_1\varepsilon,
\end{align}
where $C_1>0$ is a constant, and the parameter $\varepsilon>0$ comes from the hypothesis of Theorem~\ref{theorem:main forward}.
By the induction argument, we suppose that when $m=k$, there holds
\begin{align}\label{EST: theorem hk}
    \sup_{0\leq t\leq t^*} \|h^{k}(t)\|_{L^\infty_{x,v}}\leq 2C_1\varepsilon.
\end{align}  
Note that Lemma~\ref{lemma:coercivity} implies $\mathbf{P}\LC\Gamma \LC{h^{k}\over w}\RC\RC=0$, and also from Lemma~\ref{lemma:estimate of Gamma}, we have  
$$
      \left\| w\Gamma \LC{h^{k}\over w}\RC (t)\right\|_{L^\infty_{x,v}}\lesssim \|h^{k}(t)\|^2_{L^\infty_{x,v}}.
$$
With these, we can apply Proposition~\ref{prop:L infty estimate} to derive that there exist constants $C_1,\,C_2>0$ so that the unique solution $h^{k+1}$ satisfies
\begin{align}\label{EST: theorem hk+1}
     \sup_{0\leq t\leq t^*} \|h^{k+1}(t)\|_{L^\infty_{x,v}}\leq C_1\LC\|w f_{in}\|_{L^\infty_{x,v}}+\sup_{0\leq t\leq t^*} \|wf_-(t)\|_{L^\infty_{x,v}}\RC+C_2  \sup_{0\leq t\leq t^*}  \|h^{k}(t) \|^2_{L^\infty_{x,v}}.
\end{align}
In particular, for $h^1$, and sufficiently small $\varepsilon>0$ that satisfies
$$
     C_2 C_1 \varepsilon < {1\over 4},
$$
this then yields, by applying \eqref{EST: theorem h0} and \eqref{EST: theorem hk+1}, that
$$
    \sup_{0\leq t\leq t^*} \|h^1(t)\|_{L^\infty_{x,v}}\leq C_1\varepsilon + 
    C_2(C_1\varepsilon)^2  
    < 2C_1\varepsilon.
$$
Moreover, for $m=k+1$, using the assumption \eqref{EST: theorem hk}, it implies that
\begin{align*}
    \sup_{0\leq t\leq t^*} \|h^{k+1}(t)\|_{L^\infty_{x,v}}
    &\leq C_1\LC\|w f_{in}\|_{L^\infty_{x,v}}+\sup_{0\leq t\leq t^*} \|wf_-(t)\|_{L^\infty_{x,v}}\RC + C_2 \sup_{0\leq t\leq t^*} \|h^k(t)\|^2_{L^\infty_{x,v}}\\
    &\leq C_1\LC\|w f_{in}\|_{L^\infty_{x,v}}+\sup_{0\leq t\leq t^*} \|wf_-(t)\|_{L^\infty_{x,v}}\RC + C_2 (2C_1\varepsilon)^2 <C_1\varepsilon+ C_1\varepsilon=2C_1\varepsilon.
\end{align*}
Hence, according to the induction argument, we can conclude that $\sup_{0\leq t\leq t^*} \|h^{m}\|_{L^\infty_{x,v}}\leq 2C_1\varepsilon$ for every integer $m\geq 0$.  

\noindent {\bf Step 2. Cauchy sequence.}
The function $\tilde{h}^m:=h^{m}-h^{m-1}$, $m\geq 1$, satisfies the following problem:
\begin{align*}
    \left\{\begin{array}{ll}
          \{\p_t + v\cdot\nabla_x +\gamma(x)(\mathbf{I}-\mathbf{P}_w)\}\tilde{h}^m=  w\LC \Gamma \LC{h^{m-1}\over w}\RC-\Gamma \LC{h^{m-2}\over w}\RC\RC ,&   \\
         \tilde{h}^m|_{t=0}=0,& \\
        \tilde{h}^m|_{S_-}=0.& \\
    \end{array}\right.
\end{align*}
By Lemma~\ref{lemma:diff Gamma} replacing $f$ by $h^m/w$, we can bound the inhomogeneous term
$$
    \left\| w \LC \Gamma \LC{h^{m-1}\over w}\RC-\Gamma \LC{h^{m-2}\over w}\RC \RC(t)\right\|_{L^\infty_{x,v}} 
    \lesssim \varepsilon  \left\|(h^{m-1}-h^{m-2})(t)\right\|_{L^\infty_{x,v}}=\varepsilon \|\tilde{h}^{m-1}(t) \|_{L^\infty_{x,v}}.
$$ 
We then apply Proposition~\ref{prop:L infty estimate} to deduce that the solution $\tilde{h}^m$ satisfies the estimate
$$
    \sup_{0\leq t\leq t^*} \|\tilde{h}^m(t)\|_{L^\infty_{x,v}}
    \leq C_1 \varepsilon\sup_{0\leq t\leq t^*}\|\tilde{h}^{m-1}(t)\|_{L^\infty_{x,v}}, \quad m\geq 1,
$$
which implies
$$
    \sup_{0\leq t\leq t^*} \|\tilde{h}^m(t)\|_{L^\infty_{x,v}}
    \leq (C_1 \varepsilon)^m\sup_{0\leq t\leq t^*}\|\tilde{h}^{0}(t)\|_{L^\infty_{x,v}}, \quad m\geq 1,
$$
For small $0<\varepsilon \ll 1$, since we can express 
$$
    h^m = \sum^m_{k=1} (h^k-h^{k-1})+h^0,\quad m\geq 1,
$$
this immediately shows that $h^m$ is a Cauchy sequence in $L^\infty$, and as $m\rightarrow\infty$, its limit $h$ is the desired unique solution. This gives the solution $f$ to the linearized BGK by letting $f=h/w$:
\begin{align*}
    \{\p_t + v\cdot\nabla _x +\gamma(x)(\mathbf{I}-\mathbf{P} )\} f=  \Gamma \LC{f}\RC,\quad\hbox{with } f|_{t=0}=  f_{in},\quad f|_{S_-} = f_-. 
\end{align*}

\noindent {\bf Step 3. Uniqueness:}
Suppose that there are two solutions $F=\mu+\sqrt{\mu}f$ and $\tilde{F}=\mu+\sqrt{\mu}\tilde{f}$ to \eqref{EQN:BGK}. Then 
$$
    \{\p_t + v\cdot\nabla_x-L_\gamma \} (f-\tilde{f}) = \Gamma (f)-\Gamma(\tilde{f}),\quad \hbox{with }  (f-\tilde{f})|_{t=0}=0,\quad  (f-\tilde{f})|_{S_-}=0.
$$
By Lemma~\ref{lemma:diff Gamma} and Proposition~\ref{prop:L infty estimate}, it follows for constant $C_3>0$,
$$
    \sup_{0\leq t\leq t^*} \|w(f-\tilde{f})(t)\|_{L^\infty_{x,v}}
    \leq C_3 \varepsilon\sup_{0\leq t\leq t^*} \|w(f-\tilde{f})(t)\|_{L^\infty_{x,v}}.
$$
Since $\varepsilon$ can be chosen sufficiently small so that $C_3\varepsilon <1$, we thus obtain $f=\tilde{f}$. This completes the proof.
\end{proof}

\subsection{Perturbation of the BGK solution}
From Theorem~\ref{theorem:main forward}, we have known that there exists a constant $\varepsilon_0>0$ so that for $0< \varepsilon <\varepsilon_0$, there exists a unique solution $F_\varepsilon(t,x,v)=\mu+ \sqrt{\mu}f_\varepsilon$ to the in-flow boundary value problem \eqref{EQN:BGK} with data
\begin{align}
	F_\varepsilon|_{t=0}=\mu + \varepsilon \sqrt{\mu}f_{in},\quad 
	F_\varepsilon |_{S_-}=\mu + \varepsilon \sqrt{\mu}f_-,
\end{align} 
where  
    $$
    \|wf_{in}\|_{L^\infty_{x,v}}  + \sup_{0\leq t\leq t^*} \|wf_-(t)\|_{L^\infty_{x,v}}< 1.
    $$
Moreover, the solution $f_\varepsilon$ to linearized BGK problem   
\begin{align}\label{EQN: decomposition sec 4 linear BGK new}
    \left\{\begin{array}{ll}
          \{\p_t + v\cdot\nabla _x +\gamma(x)(\mathbf{I}-\mathbf{P} )\} f_\varepsilon=  \Gamma \LC{f_\varepsilon}\RC ,&   \\
         f_\varepsilon|_{t=0}= \varepsilon  f_{in},& \\
          f_\varepsilon |_{S_-}= \varepsilon f_-,& \\
    \end{array}\right.
\end{align}
exists and there holds
$$
     \sup_{0\leq t\leq t^*} \|wf_\varepsilon\|_{L^\infty_{x,v}}\leq C\varepsilon \LC\|w f_{in} \|_{L^\infty_{x,v}}+\sup_{0\leq t\leq t^*} \|wf_-(t)\|_{L^\infty_{x,v}}\RC.  
$$

The following result shows that we have the expansion of the solution $f_\varepsilon$.   
\begin{theorem}\label{theorem:linearization}
    Let $\Omega$ be an open bounded and convex domain in $\R^3$ with smooth boundary $\p\Omega$.
Let $0< \varepsilon <\varepsilon_0$ with $\varepsilon_0$ in Theorem~\ref{theorem:main forward}. Then we have the $\varepsilon$-expansion of $f_\varepsilon$ as follows:
$$
    f_\varepsilon= \varepsilon f^{(1)}+ {\varepsilon^2\over 2}  f^{(2)} + R^\varepsilon.
$$
 
Moreover, there exists a constant $C>0$, independent of $f_{in}$ and $f_-$, so that 
$$
     \sup_{0\leq t\leq t^*}\|w R^\varepsilon(t) \|_{L^\infty_{x,v}}\leq C \varepsilon^3 \LC\|w f_{in} \|_{L^\infty_{x,v}}  + \sup_{0\leq t\leq t^*} \|wf_-(t)\|_{L^\infty_{x,v}} \RC^3 . 
$$
\end{theorem}
\begin{proof}
As we are looking for an approximation expansion of $f_\varepsilon$, we substitute the ansatz $f_\varepsilon = \varepsilon f^{(1)}+{\varepsilon^2\over 2} f^{(2)}+R^\varepsilon$ into \eqref{EQN: decomposition sec 4 linear BGK new}. Based on the order of $\varepsilon$, we have
\begin{align}\label{EQN:Linearized BGK f_1}
    \{\p_t + v\cdot\nabla _x -L_\gamma\} f^{(1)} =0,\quad\hbox{with } f^{(1)}|_{t=0}= f_{in},\quad f^{(1)}|_{S_-} = f_-,
\end{align}
and  
\begin{align}\label{EQN:Linearized BGK f_2}
     \{\p_t + v\cdot\nabla _x - L_\gamma\} f^{(2)} = \mathcal{G}, \quad\hbox{with } 
      f^{(2)}|_{t=0}=0,\quad f^{(2)}|_{S_-} = 0 . 
\end{align}
Here we apply the finite difference approximations to the second-order derivative to express the inhomogeneous term $\mathcal{G}$. Specifically, when $\varepsilon=0$, the well-posedness of \eqref{EQN: decomposition sec 4 linear BGK new} with trivial data implies $f_0\equiv 0$ and thus $\Gamma (f_0)=0$. We then denote 
$$
    \mathcal{G} 
    :=  \lim_{\varepsilon\rightarrow 0}{1\over \varepsilon^2}\LC \Gamma(f_{2\varepsilon})-2 \Gamma(f_{\varepsilon})\RC,
$$
where $\mathcal{G}$ is a function of $t,\,x,\,v$ variables. 
It is clear that $\mathbf{P}(\mathcal{G})=0$ and $w\mathcal{G}(t)\in L^\infty_{x,v}$ by direct computations. In addition, we get 
\begin{align}\label{DEF:inhomogeneous G}
    \mathcal{G} =\p^2_\varepsilon|_{\varepsilon=0}\Gamma (f_\varepsilon)
    &:=2 (\mathbf{P}f^{(1)}-f^{(1)}) \LC  \sum^{5}_{\ell=1} \LC\int^1_0 \mathcal{P}_\ell (0)d\theta\RC\langle f^{(1)},e_\ell\rangle \RC \notag\\
    &\quad + 2\gamma \sqrt{\mu}^{-1} \sum_{i,j=1}^5\LC \int^1_0  \mathcal{Q}_{ij}(0)(1-\theta)\,d\theta\RC \langle f^{(1)},e_i\rangle\langle f^{(1)},e_j\rangle ,
\end{align}
which is equivalent to the RHS of \eqref{EQN:2 linearization appendix} in Appendix. To see this, applying \eqref{first derivative of q} gives
\begin{align*}
    \left(\begin{array}{c}
           \mathcal{P}_1 (0) \\
            \mathcal{P}_2(0)\\
            \mathcal{P}_3 (0)\\
           \mathcal{P}_4(0)\\
            \mathcal{P}_5 (0) \\
        \end{array}\right)=
    \left(\begin{array}{ccccc}
            1 & 0 & 0& 0&0\\
            0 & 1&0 & 0 &0\\
            0 & 0& 1&0  &0\\
            0& 0&0 & 1&0\\
            0 &0&0&0& \sqrt{2\over 3} \\
        \end{array}\right) 
     \left(\begin{array}{c}
            \gamma(x)\alpha \\
            0\\
            0\\
            0\\
            \gamma(x)\beta \\
        \end{array}\right) 
\end{align*}
and thus 
\begin{align*}
   \sum^{5}_{\ell=1} \LC\int^1_0 \mathcal{P}_\ell (0)d\theta\RC\langle f^{(1)},e_\ell\rangle 
        = \gamma\LC \alpha \langle f^{(1)},e_1\rangle+\sqrt{2\over 3}\beta\langle f^{(1)},e_5\rangle\RC = \gamma(\alpha \rho^{(1)}+\beta T^{(1)}),
\end{align*}
where we used $\rho^{(1)}=\p_\varepsilon \rho_\varepsilon |_{\varepsilon=0}$ and $T^{(1)}=\p_\varepsilon T_\varepsilon |_{\varepsilon=0}$ in Lemma~\ref{first order 2}.

The existence of bounded solutions $f^{(1)}$ and $f^{(2)}$ are guaranteed by using the well-posedness results in Theorem~\ref{prop: L2 estimate of boltzmann} and Proposition~\ref{prop:L infty estimate}, and moreover, there hold
\begin{align}\label{EST:f^1}
    \sup_{0\leq t\leq t^*} \|wf^{(1)}(t)\|_{L^\infty_{x,v}} \leq  C \LC\|wf_{in}\|_{L^\infty_{x,v}}  + \sup_{0\leq t\leq t^*} \|wf_-(t)\|_{L^\infty_{x,v}} \RC,
\end{align}
and
\begin{align}\label{EST:f^2}
    \sup_{0\leq t\leq t^*}  \|wf^{(2)}(t)\|_{L^\infty_{x,v}}
    &\leq   \sup_{0\leq t\leq t^*} \|w\mathcal{G}(t)\|_{L^\infty_{x,v}} 
    \leq \sup_{0\leq t\leq t^*} \|wf^{(1)}(t)\|_{L^\infty_{x,v}}^2 \notag\\
    &\leq  C\LC\|wf_{in}\|_{L^\infty_{x,v}}  + \sup_{0\leq t\leq t^*}\|wf_-(t)\|_{L^\infty_{x,v}} \RC^2.
\end{align}
We denote $W:= f_\varepsilon- \varepsilon f^{(1)}$. It satisfies
$$
    \{\p_t + v\cdot\nabla _x -L_\gamma\}W= \Gamma (f_\varepsilon),  \quad\hbox{with } 
      W|_{t=0}=0,\quad W|_{S_-} = 0 . 
$$
Thus, by Lemma~\ref{lemma:estimate of Gamma},
\begin{align}\label{EST:wW}
    \sup_{0\leq t\leq t^*} \|wW(t)\|_{L^\infty_{x,v}} \leq  \sup_{0\leq t\leq t^*} \|w\Gamma (f_\varepsilon)  (t)\|_{L^\infty_{x,v}}^2\leq  \sup_{0\leq t\leq t^*} \|w f_\varepsilon  (t)\|_{L^\infty_{x,v}}^2.
\end{align}

Finally, we denote $R^\varepsilon:= f_\varepsilon- \varepsilon f^{(1)}- {\varepsilon^2\over 2} f^{(2)}$. It satisfies
$$
\{\p_t + v\cdot\nabla _x -L_\gamma\}R^\varepsilon = \Gamma (f_\varepsilon) -  {\varepsilon^2\over 2}\mathcal{G} , \quad\hbox{with } 
      R^\varepsilon|_{t=0}=0,\quad R^\varepsilon|_{S_-} = 0 . 
$$
By \eqref{EST:f^1} and \eqref{EST:wW}, a direct computation shows that the RHS is controlled by 
\begin{align*}
    \|w(\Gamma (f_\varepsilon) -  {\varepsilon^2\over 2} \mathcal{G} )\|_{L^\infty_{x,v}}   
    &\leq C \|wW(t)\|_{L^\infty_{x,v}} \|\varepsilon (wf^{(1)})(t)\|_{L^\infty_{x,v}}+ C\sup_{0\leq t\leq t^*} \|w f_\varepsilon  (t)\|_{L^\infty_{x,v}}^3\\
    &\leq C \varepsilon^3\LC\|wf_{in}\|_{L^\infty_{x,v}}  + \sup_{0\leq t\leq t^*} \|wf_-(t)\|_{L^\infty_{x,v}} \RC^3.
\end{align*}
Applying Proposition~\ref{prop:L infty estimate} yields
\begin{align*}
     \sup_{0\leq t\leq t^*}  \|wR^\varepsilon(t)\|_{L^\infty_{x,v}} \leq  C\varepsilon^3\LC\|wf_{in}\|_{L^\infty_{x,v}}  + \sup_{0\leq t\leq t^*} \|wf_-(t)\|_{L^\infty_{x,v}} \RC^3,
\end{align*}
which completes the proof.
\end{proof}

\begin{corollary}\label{COR:derivative of F} 
Suppose all the hypothesis of Theorem~\ref{theorem:linearization} hold. Then 
we have 
$$
    \lim_{\varepsilon\rightarrow 0}\left\|w \LC {F_{\varepsilon}-\mu\over \varepsilon}    -f^{(1)}\RC\right\|_{L^\infty_{x,v}}=0,
\quad\hbox{ and }\quad
    \lim_{\varepsilon\rightarrow 0}\left\|w \LC {F_{2\varepsilon}-2F_\varepsilon+\mu \over \varepsilon^2}    -f^{(2)}\RC\right\|_{L^\infty_{x,v}}=0.
$$
\end{corollary}
\begin{proof}
When $\varepsilon=0$, the well-posedness theorem of the BGK implies $F_0\equiv \mu$. Using the estimates for $f^{(2)}$ and $R^\varepsilon$, we have
\begin{align*}
    \left\|w\LC {F_{\varepsilon} -\mu\over \varepsilon}-f^{(1)}\RC\right\|_{L^\infty_{x,v}}
    &= \left\|   {\varepsilon \over 2}  wf^{(2)} + \varepsilon^{-1} w R^\varepsilon \right\|_{L^\infty_{x,v}}\\
    &\lesssim \varepsilon\sup_{0\leq t\leq t^*}  \|wf^{(2)}(t)\|_{L^\infty_{x,v}}+\varepsilon^2\LC\|wf_{in}\|_{L^\infty_{x,v}}  + \sup_{0\leq t\leq t^*} \|wf_-(t)\|_{L^\infty_{x,v}} \RC^3\\
    &\leq \varepsilon+ \varepsilon^2,
\end{align*}
which goes to zero as $\varepsilon\rightarrow 0$. 

Similarly, for sufficiently small $\varepsilon$, using the estimate of $R^\varepsilon$ in Theorem~\ref{theorem:linearization} yields that
\begin{align*}
     \lim_{\varepsilon\rightarrow 0}\left\|w \LC {F_{2\varepsilon}-2F_\varepsilon+\mu \over \varepsilon^2}    -f^{(2)}\RC\right\|_{L^\infty_{x,v}} 
     = \lim_{\varepsilon\rightarrow 0}\left\| \varepsilon^{-2} w ( R^{2\varepsilon} - 2R^{ \varepsilon} )  \right\|_{L^\infty_{x,v}}=0.
\end{align*}
\end{proof}
 
\begin{remark}
Based on Corollary~\ref{COR:derivative of F}, we have 
$$
     \p_\varepsilon|_{\varepsilon=0} f_\varepsilon(t,x,v) =f^{(1)}(t,x,v),\quad \p^2_\varepsilon|_{\varepsilon=0} f_\varepsilon(t,x,v) =f^{(2)}(t,x,v) 
$$ 
in the weighted $L^\infty$ norm.
\end{remark}

\section{Inverse Problems - Proof of Theorem~\ref{theorem:main inverse} and Theorem~\ref{theorem:main inverse power}}\label{Sec:inverse problem}

In this section, we study inverse problems of recovering both $\gamma$ and $(\alpha,\beta)$ in the BGK collision frequency. 
Let $F_\varepsilon=\mu+\sqrt{\mu} f_\varepsilon$  be the solution to 
\begin{align*}
       \p_t F+v\cdot\nabla_x F = q(t,x,F)(M(F) -F),\quad\hbox{with } F|_{t=0}=\mu+\varepsilon\sqrt{\mu}f_{in},\quad F |_{S_-} =\mu+\varepsilon\sqrt{\mu}f_-,
\end{align*}
where $q(t,x,F)=\gamma (x)\rho^{\alpha}T^{\beta}.$
Moreover, the function $f_\varepsilon$ is the unique solution to the linearized BGK problem
\begin{align}\label{EQN: decomposition sec 4 linear BGK}
    \left\{\begin{array}{ll}
          \{\p_t + v\cdot\nabla _x +\gamma(x)(\mathbf{I}-\mathbf{P} )\} f_\varepsilon=  \Gamma \LC{f_\varepsilon}\RC ,& \hbox{ in }(0,t^*)\times\Omega\times\R^3,   \\
         f_\varepsilon|_{t=0}= \varepsilon f_{in}\,,& \hbox{ on }\{t=0\}\times\Omega\times\R^3,\\
          f_\varepsilon |_{S_-}= \varepsilon f_-\,,  &  \hbox{ on }(0,t^*)\times S_-.\\
    \end{array}\right.
\end{align}

We will show that $\gamma$ is uniquely determined from the albedo operator (Theorem~\ref{theorem:main inverse}) in Section~\ref{sec:proof of theorem gamma}. Moreover, Section~\ref{sec:proof of theorem power} is devoted to recover $(\alpha,\,\beta)$ under an extra condition, see Theorem~\ref{theorem:main inverse power}.

\subsection{Proof of Theorem~\ref{theorem:main inverse} - recover $\gamma$}\label{sec:proof of theorem gamma}

\subsubsection{\bf  First-order linearization: bounded solutions}
\label{recover gamma}
Based on Theorem~\ref{theorem:linearization} and the remark, when taking $\p_ \varepsilon|_{\varepsilon=0}$ on the BGK equation \eqref{EQN: decomposition sec 4 linear BGK}, the nonlinear operator $\Gamma(f_\varepsilon)$ will vanish as it contains at least the order of $\varepsilon^2$, while the linear part $\p_ \varepsilon|_{\varepsilon=0}L_\gamma (f_\varepsilon)$ will stay.
Hence, $\p_\varepsilon|_{\varepsilon=0} f_\varepsilon(t,x,v) =f^{(1)}(t,x,v)$ solves the first-order linearization of the BGK equation:  
\begin{align}\label{EQN:Linearized BGK IP in section 5}
    \{\p_t + v\cdot\nabla _x +\gamma(x) (\mathbf{I}-\mathbf{P}) \}f^{(1)}=0, \quad\hbox{with } f^{(1)}|_{t=0}=  f_{in},\quad f^{(1)} |_{S_-} = f_-.
\end{align}

We denote the norm $\|\cdot\|_*$ by  
\begin{align*}
    \|g\|_*:= \left\|\int_{\R^3} \mathbf{1}_{t-\tau_->0} |g(t-\tau_-(x,v),x-v\tau_-(x,v),v)|\,dv\right\|_{L^\infty_{t,x}}.
\end{align*}
Since $\gamma$ only depends on $x$, it suffices to take the trivial initial data, i.e., $f_{in}\equiv 0$.
\begin{lemma}\label{lemma:EST of f1} 
Suppose $\sup_{0\leq t\leq t^*} \|wf_-(t)\|_{L^\infty_{x,v}}+\|f_-\|_*<\infty$. Let $f^{(1)}\in L^2$ satisfy $wf^{(1)}\in L^\infty$ and solve the following problem:
    \begin{align}\label{Def:delta f1 function}
        \left\{\begin{array}{ll}
          \{\p_t + v\cdot\nabla _x +\gamma(x) (\mathbf{I}-\mathbf{P}) \}f^{(1)}=0 ,& \hbox{ in }(0,t^*)\times\Omega\times\R^3,  \\
         f^{(1)}|_{t=0} =0,& \hbox{ on }\{t=0\}\times\Omega\times\R^3,\\
        f^{(1)}|_{S_-}  = f _- , & \hbox{ on }(0,t^*)\times S_-.\\
    \end{array}\right.
    \end{align}
    Then there exists a positive constant $C$ depending on $\Omega$,  $\gamma$, $\mu$, and $t^*$ so that
    \begin{align*}
        \|\langle |f^{(1)}|\rangle\|_{L^\infty_{t,x}} \leq  C \LC    
      \| f_-\|_*   +  \LC\int_0^{ t^*} \|f_-(s)\|^2_{L^2(S_-)}\, ds \RC^{1/2}  \RC. 
    \end{align*}
\end{lemma}

\begin{proof}
We will follow a similar argument to the one used in the proof of Proposition~\ref{prop:L infty estimate}, with slight adjustments for this case by factoring out $\mu^b(v)$ from $f^{(1)}$, $0<b<{1\over 2}$, so that $\langle |f^{(1)}|\rangle$ is still bounded. To simplify the notation, in this proof, we denote $f^{(1)}$ by $u$. We first write the solution along the characteristic and get
\begin{align}\label{characteristic u IP}
     u (t,x,v)&= \mathbf{1}_{t-\tau_->0}\, \exp\LC-\int^{\tau_-(x,v)}_0 \gamma(x-sv)\,ds \RC  f_-(t-\tau_-(x,v),x-v\tau_-(x,v),v) \notag\\
     &\quad + \underbrace{\int^t_{\max\{0,t-\tau_-\}} \exp\LC-\int^{t-s}_0 \gamma(x-\tau v)\,d\tau\RC [\gamma \mathbf{P}( u )](s,x-v(t-s),v)\,ds}_{H(u)} \notag \\
     &\leq \mathbf{1}_{t-\tau_->0} |f_-(t-\tau_-(x,v),x-v\tau_-(x,v),v)| + |H(u)(t,x,v)|.
\end{align}

\noindent{\bf Step 1. Estimate of $H(u)$.} 
It remains to control $H(u)$. 
For this purpose, we first define $x_1=x-v(t-s)$, $\tau_-' = \tau_-(x_1,v')$. Since the kernel $k(v,v')$ of $\mathbf{P}$ is exponentially decay in both $v$ and $v'$, we get for $0<b<{1\over 2}$ that
\begin{align*}
     k (v,v')& = \mu^{ b}(v) (\mu^{-b}(v) k(v,v') ) \\
    &=\mu^{b}(v) \LC  \mu^{{1\over 2}-b}(v) \sqrt{\mu (v')} \bigg(1+v\cdot v'+\frac{(|v|^2-3)(|v'|^2-3)}{6}\bigg)\RC\leq c_b \mu^{b}(v), 
\end{align*}
where the constant $c_b$ is defined by
\begin{align}\label{DEF:c_b}
    c_b:=\max_{v,v'} \mu^{-b}(v) |k(v,v')|<\infty.
\end{align}

By iterating the formula \eqref{characteristic u IP} to evaluate $\gamma\mathbf{P}(u)$, we get
\begin{align*} 
      &[\gamma \mathbf{P}(u)](s,x-v(t-s),v) = \mu^{b}(v) \int  \mu^{-b}(v)k(v,v') \gamma (x_1)u(s ,x_1,v')dv' \notag\\
      &= \mu^{b}(v) \underbrace{\int \mu^{-b}(v)k(v,v') \gamma(x_1) \mathbf{1}_{s-\tau'_->0}\, e^{ -\int^{\tau_-(x_1,v')}_0 \gamma(x_1-\tau v')\,d\tau} \, f_-(s-\tau_-' ,x_1-v'\tau_-',v')\,dv'}_{\bf A_1} \notag\\
      &\quad + \mu^{b}(v) \underbrace{ \int \int^s_{\max\{0,s-\tau'_-\}} e^{-\int^{s-s'}_0 \gamma(x_1-\tau v')\,d\tau} \int \mu^{-b}(v)k(v,v') k(v',v{''}) \gamma (x_1)\gamma u(s',x_1-v'(s-s'),v{''}) dv{''}\,ds'dv'}_{\bf A_2}.
\end{align*}
By substituting this $\gamma \mathbf{P}(u)$ into $H(u)$, it follows that
\begin{align}\label{EST:I2 IP}
    H(u)&= \mu^{b}(v) \int^t_{\max\{0,t-\tau_-\}} e^{-\int^{t-s}_0 \gamma(x-\tau v)\,d\tau}\LC {\bf A_1} +{\bf A_2}\RC\,ds \notag\\
    &\leq C_{b,\gamma} \mu^{b}(v) t^*   \| f_-\|_*+  \mu^{b}(v) \int^t_{\max\{0,t-\tau_-\}} e^{-\int^{t-s}_0 \gamma(x-\tau v)\,d\tau}{\bf (A_2)}\,ds .
\end{align}
Therefore, it is sufficient to estimate the last term of the RHS of \eqref{EST:I2 IP}. To this end, similar to the proof of Proposition~\ref{prop:L infty estimate},  for a fixed small $\xi>0$, and $(t,x)\in [0,t^*]\times \overline\Omega$, we consider 
\begin{itemize}
    \item {\bf Case 1.} If $\xi >t-\tau_-(x,v)$, then we split the last term of \eqref{EST:I2 IP} as follows:
\begin{align*}
 &\hbox{Last term in \eqref{EST:I2 IP}}\\
 &=\mu^{b}(v) \int^t_{\max\{0,t-\tau_-\}} e^{-\int^{t-s}_0 \gamma(x-\tau v)\,d\tau}{\bf (A_2)}\,ds \\
  &=\mu^{b}(v) \int^t_{\max\{0,t-\tau_-\}} \int \int^s_{\max\{0,s-\tau'_-\}}  e^{-\int^{t-s}_0 \gamma(x-\tau v)\,d\tau}e^{ -\int^{s-s'}_0 \gamma(x_1-\tau v')\,d\tau} \\
       &\quad \cdot\int \mu^{-b}(v)k (v,v') k(v',v{''}) \gamma \gamma  u(s',x_1-v' (s-s') ,v{''}) dv{''} ds' dv' ds\\
 &= \mu^{b}(v)\int^{ \xi}_{\max\{0,t-\tau_-\}} \int\int^s_{\max\{0,s-\tau'_-\}}  \cdots \, ds' dv'ds + \mu^{b}(v)\int^t_{\xi} \int\LC\int^s_{s-\xi} + \int^{s-\xi}_{\max\{0,s-\tau'_-\}}\RC \cdots \, ds' dv' ds;
\end{align*}
 
    \item {\bf Case 2.} If $0<\xi <t-\tau_-(x,v)$, then $s-\xi>0$ clearly holds as $t-\tau_-\leq s<t$. We obtain
\begin{align*}
 \hbox{Last term in \eqref{EST:I2 IP}}
 &=\mu^{b}(v)\int^{t}_{\max\{0,t-\tau_-\}} \int \LC \int^s_{s-\xi}+ \int^{s-\xi}_{\max\{0,s-\tau'_-\}}\RC\cdots\,ds' dv'ds=:{\bf B_{1}}+{\bf B_{2}}.
\end{align*}    
\end{itemize}

Again, it is sufficient to estimate Case 2 solely as similar analysis also works for Case 1.\\
\noindent{\bf Step 1-1. Analyze ${\bf B_{1}}$.}  
We obtain
\begin{align}\label{EST:B1 IP}
    |{\bf B_{1}}|&\lesssim  \mu^{b}(v)\| \langle |u |\rangle\|_{L^\infty_{t,x}}\int^t_{\max\{0,t-\tau_-\}} \int_{\R^3}\int^s_{s-\xi}   |\mu^{-b}(v)k(v,v')|\cdot \max_{v''}|k(v',v'')| \,ds' dv' ds \notag \\
    &\lesssim \xi t^* \mu^{b}(v)\|  \langle |u|\rangle\|_{L^\infty_{t,x}}.
\end{align}

\noindent{\bf Step 1-2. Analyze ${\bf B_{2}}$.} Let $D=\{|v'|\geq N,\,\hbox{or }|v''|\geq N\}$. We split ${\bf B_2}$ into
$$
    |{\bf B_{2}}|= \mu^{b}(v) \int^t_{\max\{0,t-\tau_-\}}  \LC\int_D +\int_{(\R^3\times\R^3)\setminus D
    } \RC
    \int^{s-\xi}_{\max\{0,s-\tau'_-\}}   \,  ds'dv'ds=:{\bf B_{2,D}}+{\bf B_{2,D^c}}.
$$
To estimate ${\bf B_{2,D}}$, in the region $|v'|\geq  N$ or $|v''|\geq  N$, we have for $0<\delta' <1/4$ and $ b/2+\delta' <1/4$, there holds in the region $D$: as $k(v,v')k(v',v'')$ exponentially decays in $v,\,v',\,v''$, it implies
\begin{align*}
\int|\mu^{-b}(v)k(v,v') k(v',v'')| \,dv'
&\leq C \int e^{\LC-{1\over 4} +  {b\over 2}+\delta'\RC |v|^2}e^{\LC-{1\over 4}  +\delta'\RC |v'|^2} e^{\LC-{1\over 4}  +\delta'\RC |v'|^2}e^{\LC-{1\over 4} +\delta'\RC |v''|^2} \,dv'
\lesssim e^{-CN^2}.
\end{align*}
With this estimate, we can deduce
\begin{align*}
    {\bf B_{2,D}}&\lesssim  \mu^{b}(v) \int^t_{\max\{0,t-\tau_-\}} \int_D\int^{s-\xi}_{\max\{0,s-\tau'_-\}} |\mu^{-b}(v)k(v,v') k(v',v'')|   |u  (s',x_1-v'(s-s'),v{''})|\,
      ds' dv'dv''ds\\
     &\lesssim e^{-CN^2} \mu^{b}(v)  ((t^*)^2+t^*) \|  \langle |u |\rangle\|_{L^\infty_{t,x}}. 
\end{align*}
To estimate ${\bf B_{2,D^c}}$, in the region $|v'|< N$ and $|v''|<N$, we use the change of variable $v'\mapsto y=x_1-(s-s')v'$ with Jacobian 
$$
    \LV \det\LC {\partial y\over \partial v'}\RC\RV=(s-s')^3\geq \xi ^3  
$$
and derive  
\begin{align*}
    {\bf B_{2,D^c}} &\leq C_\gamma \mu^{b}(v) \int^t_{\max\{0,t-\tau_-\}} \int_{|v'|<N} \int^{s-\xi}_{\max\{0,s-\tau'_-\}} \int_{|v''|<N} 
    |u(s',x_1-v' (s-s') ,v{''})|\, dv'' ds'dv'ds\\
    &\leq C_{\gamma N }  \mu^{b}(v)  {1\over \xi^{3}} \int^t_{\max\{0,t-\tau_-\}}\int^{s-\xi}_0  \LC  \int_{|v''|<N} \int_{\Omega}  | u (s',y,v{''})|^2 dydv'' \RC ^{1 \over 2}  \,ds'ds \quad (\hbox{by H\"older estimate})\\
     &\lesssim  C_{\gamma N } \mu^{b}(v)   ((t^*)^2+t^*){1\over \xi^{3}}  \sup_{0\leq t\leq t^*}\|u(t)\|_{L^2_{x,v}}\\
     &\lesssim  C_{\gamma N } \mu^{b}(v)  ((t^*)^2+t^*) {1\over \xi^{3}}  \LC\int_0^{ t^*}  \|f_-(s)\|^2_{L^2(S_-)}\, ds \RC^{1 \over 2}    \quad \hbox{(used $L^2$ estimate in \eqref{EST:L2 estimate for f})}.
\end{align*}
From the estimates for ${\bf B_{2,D}},\, {\bf B_{2,D^c}}$ above, we obtain
\begin{align}\label{EST:B2 IP} 
    |{\bf B_{2}}|& \leq |{\bf B_{2,D}}|+ |{\bf B_{2,D^c}}| \notag\\
    &\leq \mu^{b}(v)  \LC C_{t^*} e^{-CN^2} \| \langle |u|\rangle \|_{L^\infty_{t,x}} +C_{\gamma N 
    t^* \xi} \LC\int_0^{ t^*}  \|f_-(s)\|^2_{L^2(S_-)}\, ds \RC^{1 \over 2}  \RC.
\end{align}

Hence, by combining estimates \eqref{EST:I2 IP}, \eqref{EST:B1 IP}, and \eqref{EST:B2 IP}, we get 
\begin{align}\label{EST:H(u) in the proof}
     |H(u)(t,x,v)|\lesssim  \mu^{b}(v)  \LC t^*  \|f_-\|_*  + (e^{-CN^2}+t^*\xi)\|\langle |u |\rangle\|_{L^\infty_{t,x}}+ C_{\gamma N t^* \xi} \LC\int_0^{ t^*}  \|f_-(s)\|^2_{L^2(S_-)}\, ds \RC^{1 \over 2}  \RC .
\end{align}
\noindent{\bf Step 2. Estimate of $\langle |u |\rangle$.} 
Moreover, we apply \eqref{characteristic u IP} and take average in $v$ to deduce for all $(t,x)\in (0,t^*)\times\Omega$ that 
\begin{align*}
      \langle |u |\rangle(t,x) 
      &\leq \int_{\R^3} \mathbf{1}_{t-\tau_->0} |f_-(t-\tau_-(x,v),x-v\tau_-(x,v),v)| \,dv + \int_{\R^3} |H(u)(t,x,v)|\,dv\\
      &\lesssim \|f_-\|_* +  \LC \int_{\R^3} \mu^{b}(v)\,dv \RC \LC t^*  \| f_-\|_*  + (e^{-CN^2}+t^*\xi) \|\langle |u |\rangle\|_{L^\infty_{t,x}}+ C_{\gamma N  t^* \xi} \LC\int_0^{t^*}  \|f_-(s)\|^2_{L^2(S_-)}\, ds \RC^{1 \over 2}  
      \RC .
\end{align*}
Note that $\LC \int_{\R^3} \mu^{b}(v)\,dv \RC<\infty$ for $b>0$. Lastly, by choosing sufficiently large $N$ and small $\xi$ so that the term with $\|\langle |u |\rangle\|_{L^\infty_{t,x}}$ on the RHS can be absorbed by the left, we then achieve the following estimate:  
\begin{align*}
       \|\langle |u |\rangle\|_{L^\infty_{t,x}} \lesssim 
       (1+t^*) \| f_-\|_* + \LC\int_0^{ t^*} \|f_-(s)\|^2_{L^2(S_-)}\, ds \RC^{1 \over 2}  .
\end{align*}
\end{proof}

Let's return to the term $H(u)$ and derive its upper bound
\begin{corollary}\label{corollary:H(u)}
   Recall $$ H(f^{(1)})(t,x,v) = \int^t_{\max\{0,t-\tau_-\}} \exp\LC-\int^{t-s}_0 \gamma(x-\tau v)\,d\tau\RC [\gamma \mathbf{P}(f^{(1)})](s,x-v(t-s),v)\,ds.$$
    Then there exists a positive constant $C$ depending on $\Omega$,  $\gamma$, $\mu$, and $t^*$ so that
    \begin{align*}
    |H(f^{(1)})(t,x,v)|+
    \langle|H(f^{(1)})(t,x,v)|\rangle
      \leq  C \LC    
      \| f_-\|_*   +  \LC\int_0^{t^*} \|f_-(s)\|^2_{L^2(S_-)}\, ds \RC^{1 \over 2}  \RC.   
    \end{align*}
\end{corollary}

\begin{proof}
   From \eqref{EST:H(u) in the proof} and the estimate of $f^{(1)}$ in Lemma~\ref{lemma:EST of f1}, we have 
\begin{align*} 
    |H(f^{(1)})(t,x,v)|
    &\lesssim \mu^{b}(v) \Bigg( t^*  \| f_-\|_*  + (e^{-CN^2}+t^*\xi) \LC      
       \| f_-\|_* + \LC\int_0^{ t^*} \|f_-(s)\|^2_{L^2(S_-)}\, ds \RC^{1 \over 2}  \RC \\
       &\quad + C_{\gamma N t^* \xi} \LC\int_0^{ t^*} \|f_-(s)\|^2_{L^2(S_-)}\, ds \RC^{1 \over 2}  \Bigg)\\
    &\lesssim   \| f_-\|_* + \LC\int_0^{ t^*} \|f_-(s)\|^2_{L^2(S_-)}\, ds \RC^{1 \over 2}.
\end{align*}
\end{proof}

\subsubsection{\bf Reconstruction of $\gamma$}
Let $\varphi\in C^\infty_c(\R^3)$ be a smooth function and has compact support within a unit ball $B_1(0)$ with center at the origin and radius $1$ so that $0\leq \varphi\leq 1$, $\int_{\R^3} \varphi=1$ and $\varphi(0)=1$. Fix a nonzero vector $v_0\in \R^3$ with $|v_0|>1$. 
In order to ensure $f_-^\delta\in L^2(S_-)$ has a uniform bound in $\delta$, we take the boundary data to be
$$
     f^\delta_-(t,x,v) = \LC {1\over \delta^{3 }}  \varphi \LC {v-v_0\over \delta}\RC \RC^{1\over 2}, \quad 
     0<\delta<1.
$$
We then deduce that $\|f^\delta_-\|_* $ is finite for this particular choice of boundary data:
\begin{align*}
    \|f^\delta_-\|_* \leq \int_{\R^3} \LC {1\over \delta^{3 }}  \varphi \LC {v-v_0\over \delta}\RC \RC^{1\over 2}\,dv \leq \LC\int_{\R^3}  {1\over \delta^{3}}  \varphi \LC {v-v_0\over \delta}\RC \,dv\RC^{1\over 2} \LC\int_{|v-v_0|\leq \delta} 1\,dv\RC^{1 \over 2}\leq  \LC{4\pi\over 3}\RC^{1 \over 2}.
\end{align*}
Let $f^{(1),\delta}$ solve \eqref{Def:delta f1 function} with incoming data $f^\delta_-$. Note that the existence of the solution $f^{(1),\delta}\in L^2$ satisfying $wf^{(1),\delta}\in L^\infty$ is due to Proposition~\ref{prop:L infty estimate}. The following result shows that both $\langle |f^{(1),\delta}|\rangle$ and $H(f^{(1),\delta})$ are uniformly bounded by a constant, independent of $t,\,x,\,v$ and $\delta$. 
\begin{lemma}\label{lemma:control boundary}
    For a fixed $v_0\in \R^3$, we have 
    \begin{align*}
       \|f^\delta_-(t)\|^2_{L^2(S_-)} \leq |\p\Omega| \sup_{\substack{ x\in\p\Omega \\ |v-v_0|\leq 1 }}|\nu(x)\cdot v |\leq |\p\Omega|(|v_0|+1).
    \end{align*}
Moreover, there holds
\begin{align*}
     \|\langle |f^{(1),\delta}|\rangle\|^2_{L^\infty_{t,x}}+\|H(f^{(1),\delta})\|^2_{L^\infty}
      \leq  C  
     \LC{4\pi\over 3}\RC+ C  t^* |\p\Omega|(|v_0|+1).  
    \end{align*}
\end{lemma}
\begin{proof}
    By definition, 
\begin{align*}
    \|f^\delta_-(t)\|_{L^2(S_-)}^2
    & =\int_{\p\Omega}\int_{\nu(x)\cdot v<0} |f^\delta_-(t,x,v)|^2  |\nu(x)\cdot v |\, dvd\mathcal{S}_x \\
    &\leq  \int_{\p\Omega}\int_{\R^3}{1\over  \delta^{3}} \varphi\LC {v-v_0\over \delta}\RC |\nu(x)\cdot v |\, dvd\mathcal{S}_x.
\end{align*}

Combining it with Lemma~\ref{lemma:EST of f1} and Corollary~\ref{corollary:H(u)}, the last estimate follows immediately. 
\end{proof}
 
\begin{theorem}\label{theorem:recover gamma}
    Let $0<c_0\leq \gamma\in L^\infty(\Omega)$ for some constant $c_0$. For any $\mathfrak{t}\in (0,\,t^*)$, there exists $N_\mathfrak{t}>0$ such that for any $v_0\in \mathbb R^3$, $|v_0|>N_\mathfrak{t}$, and for almost every $(t,x)\in (\mathfrak{t},\,t^*)\times \overline{\Omega}$,   
    \begin{align*}
        \lim_{\eta\rightarrow 0}\lim_{\delta\rightarrow 0} \int_{\R^3} (f^{(1),\delta})^2 (t,x,v) \varphi\LC {v-v_0\over \eta}\RC \,dv = \exp\LC-2 \int^{\tau_-(x,v_0)}_0\gamma(x-sv_0)\,ds\RC.
    \end{align*}
\end{theorem}
\begin{proof}
We use the same notation as those in lemma~\ref{lemma:EST of f1} with data $f_-=f^\delta_-$, and express the solution $u\equiv f^{(1),\delta}$ along the characteristic 
$$
u(t,x,v)= \mathbf{1}_{t-\tau_->0}\, \exp\LC-\int^{\tau_-(x,v)}_0 \gamma(x-sv)\,ds \RC f^\delta_-(t-\tau_-(x,v),x-v\tau_-(x,v),v) 
+ H(u).
$$ 
Note that for each $(t,x)\in (\mathfrak{t},\,t^*)\times \overline{\Omega}
$, since $\tau_-$ is continuous on $\overline\Omega\times \mathbb R^3$, one can find $N_\mathfrak{t}>0$ such that for any $v_0\in \mathbb R^3$ with $|v_0|>N_\mathfrak{t}$, there is $\delta_0>0$ with
$$\tau_-(x,v)\leq \mathfrak{t}\quad \hbox{ for all }|v-v_0|\leq \delta_0.$$
For such $(t,x,v)$, we have $t-\tau_-(x,v)>0$, so $\mathbf{1}_{t-\tau_->0}=1$.
Now we can evaluate
\begin{align*}
      &\lim_{\eta\rightarrow 0}\lim_{\delta\rightarrow 0} \int_{\mathbb R^3} u^2(t,x,v) \varphi\LC {v-v_0\over \eta}\RC \,dv \\
      &=\lim_{\eta\rightarrow 0}\lim_{\delta\rightarrow 0} \int_{\mathbb R^3} \exp\LC-2\int^{\tau_-(x,v)}_0 \gamma(x-sv)\,ds \RC {1\over \delta^{3}} \varphi\LC {v-v_0\over \delta}\RC \varphi\LC {v-v_0\over \eta}\RC \,dv   \\
      &\quad +\lim_{\eta\rightarrow 0}\lim_{\delta\rightarrow 0} \int_{\mathbb R^3} \left[ H(u)^2 +  2 H(u) \exp\LC- \int^{\tau_-(x,v)}_0 \gamma(x-sv)\,ds \RC \LC {1\over \delta^{3 }}  \varphi \LC {v-v_0\over \delta}\RC \RC^{1\over 2} \right] \varphi\LC {v-v_0\over \eta}\RC \,dv \\
      & =  \exp\LC -2\int^{\tau_-(x,v_0)}_0 \gamma(x-sv_0)\,ds\RC,
\end{align*}
by noting that if $g\in L^p(1\leq p\leq \infty)$, then $\int_{\R^3}{1\over  \delta^{3}} \varphi\LC {v-v_0\over \delta} \RC g(v) \,dv\rightarrow g(v_0)$ as $\delta\rightarrow 0$ for almost every $v_0$ \cite{Folland1984}. 
Moreover, in the last identity, we applied the following facts to control the last two terms involving $H(u)$. First, by using H\"older inequality and Corollary~\ref{corollary:H(u)}, we have
\begin{align*}
    &\int_{\mathbb R^3} H(u) \LC {1\over \delta^{3 }}  \varphi \LC {v-v_0\over \delta}\RC \RC^{1\over 2}  \varphi\LC {v-v_0\over \eta}\RC \,dv\\
     &\leq  \LC\int_{\mathbb R^3} |H(u)|^2   {1\over \delta^{3}} \varphi\LC {v-v_0\over \delta}\RC   \,dv\RC^{1\over 2}\LC \int_{\mathbb R^3}  \varphi^2
     \LC {v-v_0\over \eta}\RC \,dv \RC^{1 \over 2} \\
     &\leq  \|H(u)\|^2_{L^\infty} \LC \int_{\mathbb R^3}  \varphi^2
     \LC {v-v_0\over \eta}\RC \,dv \RC^{1\over 2} \rightarrow 0, \quad \hbox{as $\eta \rightarrow 0$}.
\end{align*}
Second fact we applied is as follows:
\begin{align*}
    \LV\int_{\mathbb R^3} H(u)^2   \varphi\LC {v-v_0\over \eta}\RC \,dv\RV
    &\leq \|H(u)\|^2_{L^\infty}\LV\int_{\mathbb R^3} \varphi\LC {v-v_0\over \eta}\RC \,dv \RV  \rightarrow 0,\quad \eta\rightarrow 0.
\end{align*}
\end{proof}

\begin{remark}
    With the choice of boundary data $f^\delta_-$, we observe from the proof above that $\langle |f^{(1)}|\rangle$ is smoother and is in $L^\infty_{t,x}$, compared to $f^{(1)}$, whose value increases when $\delta$ goes to zero.
\end{remark}

\begin{proof}[Proof of Theorem~\ref{theorem:main inverse}] 
Since $\gamma\in L^\infty$, we extend $\gamma$ by zero for $x\in\R^3\setminus\Omega$. From Theorem~\ref{theorem:recover gamma}, once we fix $t\in (0,t^*)$, there exists $N_0>0$, such that for any $(x,v)\in S_+$ with $|v|>N_0$, we recover the X-ray transform of $\gamma$, namely,
$$\int_0^{\tau_-(x,v)}\gamma (x-sv)\,ds.$$
Hence, applying the inversion formula for the X-ray transform, see e.g. \cite{Helgason1999, Natterer1986, SUbook}, we recover the coefficient $\gamma$.  
\end{proof}

\subsection{Proof of Theorem~\ref{theorem:main inverse power} - recover $\alpha,\,\beta$}\label{sec:proof of theorem power}
\subsubsection{\bf Second-order linearization: bounded solutions with an extra condition}  
After recovering $\gamma$ in the previous section, we will focus on the determination of the powers $\alpha$ and $\beta$ in $q=\gamma \rho^\alpha T^\beta$. 
For this purpose, 
similar to the first-order linearization in Section~\ref{recover gamma}, 
we apply Theorem~\ref{theorem:linearization} (or Lemma~\ref{prop:appendix second linearization of BGK} in the appendix) to derive the second-order linearization (taking $\p^2_ \varepsilon|_{\varepsilon=0}$) of the BGK equation \eqref{EQN: decomposition sec 4 linear BGK} and then obtain
\begin{align}\label{EQN:f2 linearization}
     \{\p_t + v\cdot\nabla_x +\gamma(x)(\mathbf{I}-\mathbf{P}) \} f^{(2)}  = 2\gamma (\alpha  \rho^{(1)}+\beta T^{(1)})(\mathbf{P}-\mathbf{I})f^{(1)}+\gamma B(v,f^{(1)}),
\end{align}
with $f^{(2)} |_{t=0}=f^{(2)} |_{S_-}=0$, where $B(v,f^{(1)})$, defined in Lemma~\ref{lemma:second linear MF}, depends only on $f^{(1)}$ and thus the term $\gamma B(v,f^{(1)})$ is known. Here $f^{(1)}$ is the solution to \eqref{EQN:Linearized BGK IP in section 5},  and from Lemma~\ref{first order 2}, we have
$$
    \rho^{(1)}:= \p_{\varepsilon}|_{\varepsilon=0} \rho_\varepsilon = \int_{\R^3} \sqrt{\mu} f^{(1)}\,dv=\langle f^{(1)},e_1\rangle,
$$
and 
$$
     T^{(1)}:= \p_{\varepsilon}|_{\varepsilon=0} T_\varepsilon =\frac{1}{3}\int_{\R^3} \sqrt{\mu} (|v|^2-3) f^{(1)}\,dv= \sqrt{2\over 3}\langle f^{(1)},e_5\rangle.
$$

We will show that under additional assumption \eqref{Hypothesis small} (i.e., {\bf Assumption 1} holds), the estimate of $\|\langle |f^{(1)}|\rangle\|_{L^\infty_{t,x}}$ in Lemma~\ref{lemma:EST of f1} still holds and the proof will be significantly simplified. Hence we have the following results.
 \begin{lemma}\label{lemma:EST of f1 part 2}
Suppose $\sup_{0\leq t\leq t^*} \|wf_-(t)\|_{L^\infty_{x,v}}+\|f_-\|_*<\infty$. 
For some $b\in (0,{1\over 2})$, suppose that
\begin{align}\label{Hypothesis small}
 t^* c_b  \|\gamma\|_{L^\infty(\Omega)}  \LC\int_{\R^3} \mu^{b}(v) \,dv\RC < \tilde\varepsilon\in (0,1),
\end{align}
where 
$ c_b=\max_{v,v'\in\R^3} \mu^{-b}(v) |k(v,v')|<\infty$ is defined in \eqref{DEF:c_b}.
Then the solution $f^{(1)}$ to the problem
    \begin{align}\label{EQN:f1 recover parameters}
        \left\{\begin{array}{ll}
          \{\p_t + v\cdot\nabla _x +\gamma(x) (\mathbf{I}-\mathbf{P}) \}f^{(1)}=0 ,& \hbox{ in }(0,t^*)\times\Omega\times\R^3,  \\
         f^{(1)}|_{t=0} =0,& \hbox{ on }\{t=0\}\times\Omega\times\R^3,\\
        f^{(1)}|_{S_-}   = f _-\, , & \hbox{ on }(0,t^*)\times S_-,\\
    \end{array}\right.
    \end{align}
satisfies
    \begin{align}\label{lemma:estimate of u in reconstruction of parameter}
        \|\langle |f^{(1)}|\rangle\|_{L^\infty_{t,x}} \leq    
        {1\over 1- \tilde\varepsilon} \| f _-\|_*,\quad\hbox{and}\quad |H(f^{(1)})(t,x,v)|
    \leq { \tilde\varepsilon\over 1- \tilde\varepsilon}{\mu^{b}(v)\over \LC\int \mu^{b}(v) \,dv\RC} \| f _-\|_*.
    \end{align}
\end{lemma}
 
\begin{proof}
Denote $u\equiv f^{(1)}$.
Hence, from \eqref{characteristic u IP}, we again obtain
\begin{align}\label{characteristic u IP new}
     |u (t,x,v)|
     &\leq \mathbf{1}_{t-\tau_->0} |f _-(t-\tau_-(x,v),x-v\tau_-(x,v),v)| + |H(u)(t,x,v)|.
\end{align}
 
Recall the notation $x_1:=x-v(t-s)$.  
From the definition of $H(u)$, using the estimate
\begin{align}\label{estimate of Pu}
    |\gamma\mathbf{P} u(s,x_1,v)|=\mu^{b}(v)\LV \int_{\R^3} \gamma(x_1)\mu^{-b}(v)k(v,v')\, u(s,x_1,v')  \,dv' \RV \leq c_b\|\gamma\|_{L^\infty} \|\langle | u|\rangle \|_{L^\infty_{t,x}} \mu^{b}(v)
\end{align}
to deduce the following uniform bound of $H(u)$ for $t,\,x$ variables:
\begin{align}\label{estimate of H new}
    |H( u)(t,x,v)|
    &=\LV\int^t_{\max\{0,t-\tau_-\}} \exp\LC-\int^{t-s}_0 \gamma(x-\tau v)\,d\tau\RC  \underbrace{\LC  \int_{\R^3} \gamma(x_1)k(v,v')\, u(s,x_1,v')  \,dv'\RC }_{\gamma\mathbf{P} u(s,x_1,v)} \,ds\RV \notag\\
    &\leq  t^* c_b \|\gamma\|_{L^\infty} \|\langle | u|\rangle \|_{L^\infty_{t,x}} \mu^{b}(v) .
\end{align}
Taking average of \eqref{characteristic u IP new} in $v$, from the hypothesis, we get
\begin{align*}
    \langle |u|\rangle (t,x)
    &\leq  \int_{\R^3} \mathbf{1}_{t-\tau_->0} |f_-(t-\tau_-(x,v),x-v\tau_-(x,v),v)| \,dv +  \int_{\R^3}  |H( u)|\,dv\\
    &\leq \|f _-\|_*+\underbrace{ t^* c_b  \|\gamma\|_{L^\infty}  \LC\int_{\R^3} \mu^{b}(v) \,dv\RC }_ {<\,  \tilde\varepsilon\hbox{ by hypothesis}}\|\langle |u|\rangle\|_{L^\infty_{t,x}},
\end{align*}
which yields \eqref{lemma:estimate of u in reconstruction of parameter}.
Finally, substituting it into \eqref{estimate of H new} implies
$$
|H( u)(t,x,v)|
 \leq  \tilde\varepsilon{\mu^{b}(v)\over \LC\int \mu^{b}(v) \,dv\RC}   \|\langle | u|\rangle \|_{L^\infty_{t,x}} \leq { \tilde\varepsilon\over 1- \tilde\varepsilon}{\mu^{b}(v)\over \LC\int \mu^{b}(v) \,dv\RC} \| f _-\|_*  .
$$
\end{proof}

\begin{lemma}\label{theorem:recover alpha beta}
Suppose that $f_-\equiv f_-(v)$. Suppose that
\eqref{Hypothesis small} holds. 
Then the solution of \eqref{EQN:f2 linearization} satisfies 
$$ 
    \|\langle |f^{(2)}|\rangle\|_{L^\infty_{t,x}} <\infty.
$$ 
\end{lemma}
\begin{proof}  
By the method of characteristic, with zero initial and incoming data, the solution of \eqref{EQN:f2 linearization} is expressed as
\begin{align*}
    & f^{(2)}(t,x,v)= \underbrace{\int_{\max\{0,t-\tau_-\}}^t\exp\bigg(-\int_0^{t-s}\gamma(x-\tau v)\,d\tau\bigg)\bigg[\gamma \mathbf{P} f^{(2)}\bigg](s,x-v(t-s),v)\,ds}_{\hbox{denoted by }H(f^{(2)})}\\
   & + \underbrace{\int_{\max\{0,t-\tau_-\}}^t  \exp\bigg(-\int_0^{t-s}\gamma(x-\tau v)\,d\tau\bigg)\bigg[2 \gamma \big(\alpha \rho^{(1)}+\beta T^{(1)}\big) (\mathbf{P}-\mathbf{I}) f^{(1)} \bigg](s,x-v(t-s),v)\,ds}_{\hbox{denoted by }G(f^{(1)})}\\ 
    & \quad +\underbrace{\int_{\max\{0,t-\tau_-\}}^t \exp\bigg(-\int_0^{t-s}\gamma(x-\tau v)\,d\tau\bigg)\bigg[\gamma B(v,f^{(1)})\bigg](s,x-v(t-s),v)\,ds}_{\hbox{denoted by }J(f^{(1)})}.
\end{align*}
From \eqref{estimate of H new}, we have
\begin{align}\label{EST:H(u) in the proof second-order}
     |H(f^{(2)})(t,x,v)| \leq t^* c_b \|\gamma\|_{L^\infty} \|\langle | f^{(2)}|\rangle \|_{L^\infty_{t,x}} \mu^{b}(v). 
\end{align}
Also, from \eqref{estimate of Pu}, the following estimate holds:
\begin{align*} 
    |\gamma{\bf P}(f^{(1)})(t,x,v)|  \leq c_b\|\gamma\|_{L^\infty} \|\langle | f^{(1)}|\rangle \|_{L^\infty_{t,x}} \mu^{b}(v).
\end{align*} 
In the meantime, by the definition of $\rho^{(1)}$ and $T^{(1)}$,
we obtain
$$ |\rho^{(1)}(t,x)|\lesssim  \|\langle |f^{(1)} |\rangle\|_{L^\infty_{t,x}},\quad  |T^{(1)}(t,x)|\lesssim \|\langle |f^{(1)} |\rangle\|_{L^\infty_{t,x}},$$
which lead to
\begin{align*}
     &|\gamma\big(\alpha \rho^{(1)}+\beta T^{(1)}\big) (\mathbf{P}-\mathbf{I}) f^{(1)} (t,x,v)|\\
    &\leq  |\gamma\big(\alpha \rho^{(1)}+\beta T^{(1)}\big)  \mathbf{P}  f^{(1)}  (t,x,v)|+|\gamma\big(\alpha \rho^{(1)}+\beta T^{(1)}\big) f^{(1)}  (t,x,v)|\\
    &\lesssim  \|\gamma\|_{L^\infty} \|\langle |f^{(1)} |\rangle\|^2_{L^\infty_{t,x}} \mu^b (v) +  \|\langle |f^{(1)} |\rangle\|_{L^\infty_{t,x}}|f^{(1)}(t,x,v)|.  
\end{align*}
Hence, we have
\begin{align}\label{EST:Gf1}
    |G(f^{(1)})(t,x,v)|\lesssim  \|\langle |f^{(1)} |\rangle\|^2_{L^\infty_{t,x}} \mu^b (v) +  \|\langle |f^{(1)} |\rangle\|_{L^\infty_{t,x}} \LC \int^t_{\max\{0,t-\tau_-\}} |f^{(1)}(s,x-v(t-s),v)|\,ds \RC.
\end{align}
For $J(f^{(1)} )$, one can check that $|B(v,f^{(1)})|\lesssim \mu^b(v)\langle |f^{(1)} |\rangle^\ell$ for some power $\ell\in\mathbb{N}$, and therefore
\begin{align}\label{EST: Kf1}
    |J(f^{(1)})(t,x,v)|\lesssim \|\langle |f^{(1)} |\rangle\|_{L^\infty_{t,x}}^\ell \mu^b(v).
\end{align}

Combining the above estimates \eqref{EST:H(u) in the proof second-order}-\eqref{EST: Kf1}, we have
\begin{align*}
   \|\langle|f^{(2)}|\rangle\|_{L^\infty_{t,x}} 
   &\lesssim \underbrace{ \LC\int_{\R^3}\mu^{b}(v)\,dv\RC  t^* c_b \|\gamma\|_{L^\infty} }_{<\tilde{\varepsilon}\in(0,1)}\|\langle | f^{(2)}|\rangle \|_{L^\infty_{t,x}} +  \LC\int_{\R^3}\mu^{b}(v)\,dv\RC  \|\langle |f^{(1)} |\rangle\|^2_{L^\infty_{t,x}}   \\
   &\quad  + \|\langle |f^{(1)} |\rangle\|_{L^\infty_{t,x}}  \sup_{t,x}\underbrace{\LC  \int_{\R^3} \int^t_{\max\{0,t-\tau_-\}} |f^{(1)}(s,x-v(t-s),v)|\, dsdv \RC }_{G'(f^{(1)})(t,x)} \\ 
   &\quad + \LC\int_{\R^3}\mu^{b}(v)\,dv\RC  \|\langle |f^{(1)} |\rangle\|_{L^\infty_{t,x}}^\ell.
\end{align*}
Note that based on the hypothesis~\eqref{Hypothesis small}, we can absorb the term with $\|\langle |f^{(2)} |\rangle\|_{L^\infty_{t,x}}$ on the RHS to the left. Next, due to \eqref{lemma:estimate of u in reconstruction of parameter}, we get $\|\langle |f^{(1)} |\rangle\|_{L^\infty_{t,x}}<\infty$. 
It remains to bound $G'(f^{(1)})$. Recall $x_1=x-v(t-s)$. Applying \eqref{lemma:estimate of u in reconstruction of parameter} and noting that $f_-\equiv f_-(v)$ only depends on $v$ based on the hypothesis, we obtain
\begin{align*}
    &\int_{\R^3} \int^t_{\max\{0,t-\tau_-\}} |f^{(1)}(s,x_1,v)|\, dsdv\\
    &\leq \int_{\R^3} \int^t_{\max\{0,t-\tau_-\}} \mathbf{1}_{s-\tau_-(x_1,v)>0} |f_-(v)| \, dsdv+ \int_{\R^3} \int^t_{\max\{0,t-\tau_-\}}|H(f^{(1)})(s,x_1,v)|\, dsdv\\
 & \leq t\int_{\mathbb R^3} \mathbf{1}_{t-\tau_->0} |f_-(v)|\,dv+t\int_{\mathbb R^3} \frac{\tilde\varepsilon}{1-\tilde\varepsilon}\frac{\mu^b(v)}{\int \mu^b(v')\,dv'}\|f_-\|_*\,dv\\
    &\leq t^* \|f_- \|_*+t^* { \tilde\varepsilon\over 1- \tilde\varepsilon} \| f _-\|_*.
\end{align*} 
Therefore, we derive the uniform bound
$$ \|\langle|f^{(2)}|\rangle\|_{L^\infty_{t,x}}  <\infty.$$
\end{proof}

\subsubsection{\bf Reconstruction of $\alpha,\,\beta$}
For $0< \delta<1$, we take the boundary data of the solution $f^{(1),\delta}$ of \eqref{EQN:f1 recover parameters} to be
$$
    f^{\delta}_-(t,x,v) = {1\over \delta^{3 }} \varphi \LC {v-v_0\over \delta}\RC . 
$$
For a fixed $v_0\in \R^3$, since $\int_{\R^3} \varphi=1$ by the definition of $\varphi$, it is clear that
    \begin{align*}
        \|f^{ \delta}_-\|_* \leq \int_{\R^3} |f^{ \delta}_-(t-\tau_-(x,v),x-v\tau_-(x,v),v)| \,dv =1.
    \end{align*}
From Lemma~\ref{lemma:EST of f1 part 2}, it yields that
\begin{align}\label{characteristic u IP q}
     f^{(1),\delta} (t,x,v)&= \mathbf{1}_{t-\tau_->0}\, \exp\LC-\int^{\tau_-(x,v)}_0 \gamma(x-sv)\,ds \RC f^{\delta}_-(t-\tau_-(x,v),x-v\tau_-(x,v),v) \notag\\
     &\quad + \underbrace{\int^t_{\max\{0,t-\tau_-\}} \exp\LC-\int^{t-s}_0 \gamma(x-\tau v)\,d\tau\RC [\gamma \mathbf{P}(f^{(1),\delta})](s,x-v(t-s),v)\,ds}_{H(f^{(1),\delta})},
\end{align}
satisfies 
\begin{align}\label{estimate f1 and Hf1}
    \|\langle |f^{(1),\delta}|\rangle\|_{L^\infty_{t,x}} \leq {1\over 1-\tilde\varepsilon} ,\quad\hbox{ and }\quad  |H(f^{(1),\delta})(t,x,v)|
    \leq {\tilde\varepsilon\over 1-\tilde\varepsilon}{\mu^{b}(v)\over \LC\int \mu^{b}(v) \,dv\RC} .
\end{align}
Let $f^{{(2)},\delta}$ solve the problem \eqref{EQN:f2 linearization} by replacing $f^{(1)}$ with $f^{(1),\delta}$.

We are ready to derive the following result.
\begin{theorem}\label{Theorem:ID for f2}
 Let $0<c_0\leq \gamma\in L^\infty(\Omega)$ for some constant $c_0$. For any $\mathfrak{t}\in (0,\,t^*)$, there exists $N_\mathfrak{t}>0$ such that for any $v_0\in \mathbb R^3$, $|v_0|>N_\mathfrak{t}$, and for almost every $(t,x)\in (\mathfrak{t},\,t^*)\times \overline{\Omega}$, 
\begin{align}
    & \lim_{\eta\to 0}\lim_{ \delta\to 0}  \int_{\R^3} f^{{(2)},\delta}(t,x,v) \varphi\LC \frac{v-v_0}{\eta}\RC\,dv \notag \\
    &= \Lambda_\gamma(t,x,v_0) \Bigg(
      \alpha \LC \sqrt{\mu (v_0)}   + \mathcal{O}(\tilde\varepsilon)\RC 
      + \beta \LC \sqrt{\mu (v_0)} \LC {|v_0|^2-3\over 3}\RC   + \mathcal{O}(\tilde\varepsilon)\RC \Bigg), 
\end{align}
where $\Lambda_\gamma(t,x,v_0)$ is known data and is defined by
$$     \Lambda_\gamma(t,x,v_0):= -2 e^{-\int_0^{\tau_-(x,v_0)}\gamma(x-\tau v_0)\,d\tau} \LC 
      \int_{\max\{0,t-\tau_-\}}^t \gamma(x -v_0(t-s)) e^{ -\int^{\tau_-(x_1, v_0)}_0\gamma(x_1-\tau v_0)\,d\tau}\,ds\RC.
$$
\end{theorem}
Here the notation $f(x)=\mathcal{O}(g(x))$ is read as $f(x)$ is a big $O$ of $g(x)$ if there exists a constant $M>0$ so that $|f(x)|\leq Mg(x)$ for all $x$ in the domain of $f$.
\begin{proof}
In this proof, we will simply write $f^{(k)}$, instead of $f^{(k),\delta}$ for simplicity of notation. Now we proceed similarly to the proof of Theorem~\ref{theorem:recover gamma}, in which we recovered $\gamma$.
Let's analyze the following limits:
\begin{align}\label{Identity f2}
    & \lim_{\eta\to 0}\lim_{ \delta\to 0}  \int_{\R^3} f^{(2)}(t,x,v) \varphi\LC \frac{v-v_0}{\eta}\RC\,dv  \notag\\
    &=\underbrace{ \lim_{\eta\to 0}\lim_{ \delta\to 0}   \int_{\R^3} H(f^{(2)}) \varphi\LC \frac{v-v_0}{\eta}\RC\,dv}_{=0}+ \lim_{\eta\to 0}\lim_{ \delta\to 0}  \int_{\R^3} G(f^{(1)}) \varphi\LC \frac{v-v_0}{\eta}\RC\,dv  \notag\\
    &\quad +\underbrace{ \lim_{\eta\to 0}\lim_{ \delta\to 0}   \int_{\R^3} J(f^{(1)}) \varphi\LC \frac{v-v_0}{\eta}\RC\,dv}_{=0},
\end{align}
where we used the fact that since both $\|\langle|f^{(1)}|\rangle\|_{L^\infty_{x,v}}$ and $\|\langle|f^{(2)}|\rangle\|_{L^\infty_{x,v}}$ are finite and their upper bounds are independent of $\delta$, the first and third terms on the RHS converge to zero when $\eta\rightarrow 0$.

It remains to analyze the second term on the RHS of \eqref{Identity f2}. 
We denote
\begin{align*}
    G(f^{(1)}) &= \alpha \underbrace{\int_{\max\{0,t-\tau_-\}}^t  \exp\bigg(-\int_0^{t-s}\gamma(x-\tau v)\,d\tau\bigg)\bigg[2 \gamma \rho^{(1)} (\mathbf{P}-\mathbf{I}) f^{(1)} \bigg](s,x-v(t-s),v)\,ds}_{G_1(\rho^{(1)},f^{(1)})} \\
    &\quad + \beta \underbrace{\int_{\max\{0,t-\tau_-\}}^t  \exp\bigg(-\int_0^{t-s}\gamma(x-\tau v)\,d\tau\bigg)\bigg[2 \gamma T^{(1)}  (\mathbf{P}-\mathbf{I}) f^{(1)} \bigg](s,x-v(t-s),v)\,ds}_{G_1(T^{(1)},f^{(1)})}.
\end{align*}

\noindent{\bf Step 1. Estimate $\rho^{(1)}$ and $T^{(1)}$.} We start by estimating the term with $\rho^{(1)}$ in $G(f^{(1)})$ first as one can deal with the one with $T^{(1)}$ similarly. 
By substituting $f^{(1)}$ in \eqref{characteristic u IP q} into $\rho^{(1)}$, one can derive
\begin{align}\label{limit of rho 1}
    \rho^{(1)}(s,x_1)&:=\int_{\R^3} \sqrt{\mu(v')} f^{(1)}(s,x_1,v') \,dv' \notag\\
    &= \underbrace{\int_{\R^3} \sqrt{\mu (v')} \mathbf{1}_{s-\tau_-(x_1,v')>0} e^{ -\int^{\tau_-(x_1,v' )}_0\gamma(x_1-\tau v' )\,d\tau} \LC  {1\over \delta^{3}} \varphi
    \LC {v'-v_0\over \delta}\RC\RC \,dv'}_{\hbox{converges to }\sqrt{\mu (v_0)}  \mathbf{1}_{s-\tau_-(x_1,v_0)>0} e^{ -\int^{\tau_-(x_1, v_0)}_0\gamma(x_1-\tau v_0)\,d\tau},\,\,\delta\rightarrow 0} \notag\\
    &\quad +\underbrace{\int_{\R^3} \sqrt{\mu (v')}  H(f^{(1)})(s,x -(t-s)v ,v') \,dv'}_{\hbox{denoted by $\mathcal{H}_1$}},
\end{align}
where the estimate of $H(f^{(1)})$ in \eqref{estimate f1 and Hf1} leads to, as $\tilde\varepsilon \rightarrow 0$,
\begin{align*}
    |\mathcal{H}_1|\leq   \int_{\R^3} \sqrt{\mu (v')} {\tilde\varepsilon\over 1-\tilde\varepsilon}{\mu^{b}(v')\over \int_{\R^3} \mu^{b}(v) \,dv }   \,dv'= {\tilde\varepsilon\over 1-\tilde\varepsilon} {\int_{\R^3}  \mu^{{1\over 2}+b} (v') \,dv'\over \int_{\R^3} \mu^{b}(v) \,dv }= \mathcal{O}(\tilde\varepsilon).
\end{align*}
{\bf Step 1-1. Existence of the limit of $\mathcal{H}_1$.}
To further analyze the convergence of $\mathcal H_1$ as $\delta\to 0$, we denote 
$$f^\delta_0(t,x,v)= \mathbf{1}_{t-\tau_->0}\, \exp\LC-\int^{\tau_-(x,v)}_0 \gamma(x-sv)\,ds \RC {1\over \delta^{3 }} \varphi \LC {v-v_0\over \delta}\RC,$$ 
then using the linearity of the operator $H$, one can derive by iteration that
$$f^{(1)}=f^{(1),\delta}=\sum_{k=0}^\infty H^k(f^\delta_0).$$
Using \eqref{estimate of H new} and $\langle|f^\delta_0|\rangle\leq 1$, it is easy to check by an induction argument that
$$|H^k(f^\delta_0)(t,x,v)|\leq \tilde\varepsilon^k {\mu^{b}(v)\over \int \mu^{b}(v) \,dv},\quad k\geq 1.$$
As $\tilde\varepsilon\in(0,1)$, this implies that $f^{(1)} =\sum^\infty_{k=0}H^k (f^\delta_0)$
converges uniformly and absolutely in $(0,t^*)\times\Omega\times \R^3$.
In the meantime, note that
\begin{align*}
    \lim_{\delta\to 0}{\bf P} (f^\delta_0)(t,x,v) & =\lim_{\delta\to 0}\int_{\mathbb R^3} k(v,v') \mathbf{1}_{t-\tau_->0}\, \exp\LC-\int^{\tau_-(x,v')}_0 \gamma(x-sv')\,ds \RC {1\over \delta^{3 }} \varphi \LC {v'-v_0\over \delta}\RC\, dv'\\
    & = k(v,v_0)\mathbf{1}_{t-\tau_-(x,v_0)>0}\exp\LC \int^{\tau_-(x,v_0)}_0 \gamma(x-sv_0)\,ds\RC,
\end{align*}
we can derive that $\lim_{\delta\to 0} H^k(f^\delta_0)$ exists for any $k\geq 1$.
Therefore the limit
\begin{align*}
    \lim_{\delta\to 0}\mathcal H_1 & =\lim_{\delta\to 0} \int_{\mathbb R^3}\sqrt{\mu(v')}\LC\sum_{k=1}^\infty H^k(f^\delta_0)\RC (s,x-(t-s)v,v')\,dv'\\
    & =\sum_{k=1}^\infty \int_{\mathbb R^3}\sqrt{\mu(v')}\LC \lim_{\delta\to 0} H^k(f^\delta_0)\RC (s,x-(t-s)v,v')\,dv',
\end{align*}
exists, and satisfies
$$|\lim_{\delta\to 0} \mathcal H_1 |\leq \mathcal O (\tilde\varepsilon).$$

\noindent{\bf Step 1-2. Analyze $T^{(1)}$.}
Similarly, we obtain
\begin{align}\label{limit of T 1}    
T^{(1)}(s,x_1)&:=\underbrace{\int_{\R^3} \sqrt{\mu(v')}\LC {|v'|^2-3\over 3}\RC \Big[  \mathbf{1}_{s-\tau_-(x_1,v')>0} e^{ -\int^{\tau_-(x_1,v' )}_0\gamma(x_1-\tau v' )\,d\tau} \LC  {1\over \delta^{3}} \varphi\LC {v'-v_0\over \delta}\RC\RC\Big]\,dv' }_{\hbox{converges to }\sqrt{\mu (v_0)}\LC {|v_0|^2-3\over 3}\RC \mathbf{1}_{s-\tau_-(x_1,v_0)>0}  e^{ -\int^{\tau_-(x_1, v_0)}_0\gamma(x_1-\tau v_0)\,d\tau},\,\,\delta\rightarrow 0} \notag\\
    &\quad +\underbrace{\int_{\R^3} \sqrt{\mu(v')}\LC {|v'|^2-3\over 3}\RC H(f^{(1)})(s,x_1,v')\,dv' }_{\mathcal{H}_2}, 
\end{align}
where as $\tilde\varepsilon \rightarrow 0$,
\begin{align*}
    |\mathcal{H}_2| \leq  {\tilde\varepsilon\over 1-\tilde\varepsilon} {\int_{\R^3} \mu^{{1\over 2}+b} (v') \LC {|v'|^2-3\over 3}\RC \,dv'\over \int_{\R^3} \mu^{b}(v) \,dv } = \mathcal{O}(\tilde\varepsilon).
\end{align*}
Moreover, $\lim_{\delta\to 0} \mathcal H_2$ exists and 
$$|\lim_{\delta\to 0}\mathcal H_2|\leq \mathcal O (\tilde\varepsilon).$$

\noindent{\bf Step 2. Estimate $G(f^{(1)})$ term in \eqref{Identity f2}.}\\
{\bf Step 2-1. Analyze $G_1(\rho^{(1)},f^{(1)})$.} We split it into two terms:
\begin{align*}
     &\alpha \lim_{\eta\to 0}\lim_{ \delta\to 0} \int_{\R^3} G_1(\rho^{(1)},f^{(1)})\varphi\LC \frac{v-v_0}{\eta}\RC\,dv  \\
     &= \alpha \lim_{\eta\to 0}\lim_{ \delta\to 0}  \underbrace{\int_{\R^3} \LC\int_{\max\{0,t-\tau_-\}}^t  e^{-\int_0^{t-s}\gamma(x-\tau v)\,d\tau} 2 \gamma  \rho^{(1)} \mathbf{P} f^{(1)} (s,x-v(t-s),v)\,ds\RC \varphi\LC \frac{v-v_0}{\eta}\RC\,dv}_{\bf A^\rho_1} \\
     &\quad - \alpha\lim_{\eta\to 0}\lim_{ \delta\to 0}  \underbrace{ \int_{\R^3} \LC \int_{\max\{0,t-\tau_-\}}^t  e^{-\int_0^{t-s}\gamma(x-\tau v)\,d\tau}2 \gamma  \rho^{(1)} f^{(1)} (s,x-v(t-s),v)\,ds \RC\varphi\LC \frac{v-v_0}{\eta}\RC\,dv}_{\bf A^\rho_2}. 
\end{align*}
Since $\mathbf{P}(f^{(1)})$ is uniformly bounded in the $L^\infty$ norm, we derive that the first term 
$$
{\bf A^\rho_1} =\mathcal{O}(\eta^3) \rightarrow 0 \quad \hbox{ as $\eta\rightarrow 0$.}
$$

Also, the ${\bf A^\rho_2}$ term can be estimated as follows:
\begin{align*}
    &\lim_{ \delta\to 0}\int_{\R^3} \LC \int_{\max\{0,t-\tau_-\}}^t  e^{-\int_0^{t-s}\gamma(x-\tau v)\,d\tau}2 \gamma  \rho^{(1)} f^{(1)} (s,x-v(t-s),v)\,ds \RC\varphi\LC \frac{v-v_0}{\eta}\RC\,dv \\
    &= \lim_{ \delta\to 0}\int_{\R^3} \Bigg( \int_{\max\{0,t-\tau_-\}}^t  e^{-\int_0^{t-s}\gamma(x-\tau v)\,d\tau}2 \gamma (x_1) \rho^{(1)} \\
    &\quad \hskip1cm\cdot \LC \mathbf{1}_{s-\tau_-(x_1,v)>0} e^{ -\int^{\tau_-(x_1,v)}_0\gamma(x_1-\tau v)\,d\tau} {1\over \delta^{3}} \varphi\LC {v-v_0\over \delta}\RC \RC\,ds \Bigg)\varphi\LC \frac{v-v_0}{\eta}\RC\,dv \\
    &\hskip9cm(\hbox{use \eqref{characteristic u IP q} and the boundedness of $H(f^{(1)})$})\\
    & = 2 e^{-\int_0^{\tau_-(x,v_0)}\gamma(x-\tau v_0)\,d\tau}  \int_{\max\{0,t-\tau_-\}}^t  \gamma(x-v_0(t-s)) \\
    &\quad \hskip3cm\cdot\LC \sqrt{\mu (v_0)} \mathbf{1}_{s-\tau_-(x-v_0(t-s),v_0)>0} e^{ -\int^{\tau_-(x_1, v_0)}_0\gamma(x_1-\tau v_0)\,d\tau}+ \mathcal{O} (\tilde\varepsilon)\RC\,ds.
\end{align*}
Note that $\mathbf{1}_{s-\tau_-(x-v_0(t-s),v_0)>0}=1$ since
$$
    s-\tau_-(x-v_0(t-s),v_0) =t-\tau_-(x ,v_0)>0.
$$
Thus, we obtain 
\begin{align}\label{EST:G1}
     \alpha \lim_{\eta\to 0}\lim_{ \delta\to 0} \int_{\R^3} G_1(\rho^{(1)},f^{(1)})\varphi\LC \frac{v-v_0}{\eta}\RC\,dv  = \alpha \Lambda_\gamma(t,x,v_0)  
      \LC \sqrt{\mu (v_0)} + \mathcal{O}(\tilde\varepsilon)\RC,
\end{align}
where $\Lambda_\gamma(t,x,v_0)$ is known data and is defined by
$$
     \Lambda_\gamma(t,x,v_0):=-2 e^{-\int_0^{\tau_-(x,v_0)}\gamma(x-\tau v_0)\,d\tau} \LC 
      \int_{\max\{0,t-\tau_-\}}^t \gamma(x -v_0(t-s)) e^{ -\int^{\tau_-(x_1, v_0)}_0\gamma(x_1-\tau v_0)\,d\tau}\,ds\RC.
$$

\noindent{\bf Step 2-2. Analyze $G_1(T^{(1)},f^{(1)})$.}
For the terms with $T^{(1)}$, following a similar discussion, we can also deduce that 
\begin{align}\label{EST:G2}
     &\beta \lim_{\eta\to 0}\lim_{ \delta\to 0} \int_{\R^3} G_2(T^{(1)},f^{(1)})\varphi\LC \frac{v-v_0}{\eta}\RC\,dv  \notag\\
     &= \beta \lim_{\eta\to 0}\lim_{ \delta\to 0}  \underbrace{\int_{\R^3} \LC\int_{\max\{0,t-\tau_-\}}^t  e^{-\int_0^{t-s}\gamma(x-\tau v)\,d\tau} 2 \gamma  T^{(1)} \mathbf{P} f^{(1)} (s,x-v(t-s),v)\,ds\RC \varphi\LC \frac{v-v_0}{\eta}\RC\,dv}_{\bf A^T_1} \notag\\
     &\quad - \beta \lim_{\eta\to 0}\lim_{ \delta\to 0}  \underbrace{ \int_{\R^3} \LC \int_{\max\{0,t-\tau_-\}}^t  e^{-\int_0^{t-s}\gamma(x-\tau v)\,d\tau}2 \gamma T^{(1)} f^{(1)} (s,x-v(t-s),v)\,ds \RC\varphi\LC \frac{v-v_0}{\eta}\RC\,dv}_{\bf A^T_2} \notag\\
     &= 2\beta e^{-\int_0^{\tau_-(x,v_0)}\gamma(x-\tau v_0)\,d\tau}  \int_{\max\{0,t-\tau_-\}}^t  \gamma(x-v_0(t-s)) \LC \sqrt{\mu (v_0)} \LC {|v_0|^2-3\over 3}\RC  e^{ -\int^{\tau_-(x_1, v_0)}_0\gamma(x_1-\tau v_0)\,d\tau}+ \mathcal{O} (\tilde\varepsilon) \RC\,ds  \notag\\
     &=\beta \Lambda_\gamma(t,x,v_0) 
      \LC \sqrt{\mu (v_0)} \LC {|v_0|^2-3\over 3}\RC + \mathcal{O}(\tilde\varepsilon)\RC,
\end{align}
where ${\bf A^T_j}$ was analyzed similarly as ${\bf A^\rho_j}$ for $j=1,\,2$.
By \eqref{Identity f2}, \eqref{EST:G1} and \eqref{EST:G2}, we arrive at
\begin{align*}
     I(t,x,v_0)&:=\lim_{\eta\to 0}\lim_{ \delta\to 0}  \int_{\R^3} f^{(2)}(t,x,v) \varphi\LC \frac{v-v_0}{\eta}\RC\,dv\\
     &=\lim_{\eta\to 0}\lim_{ \delta\to 0}  \int_{\R^3}G(f^{(1)}) \varphi\LC \frac{v-v_0}{\eta}\RC\,dv \\
     &=\Lambda_\gamma(t,x,v_0) \Bigg(
      \alpha \LC \sqrt{\mu (v_0)}   + \mathcal{O}(\tilde\varepsilon)\RC 
      + \beta \LC \sqrt{\mu (v_0)} \LC {|v_0|^2-3\over 3}\RC   + \mathcal{O}(\tilde\varepsilon)\RC \Bigg),
\end{align*}
where the LHS and $\Lambda_\gamma(t,x,v_0)$ are known.
\end{proof}

\begin{proof}[Proof of Theorem~\ref{theorem:main inverse power}]
By applying Theorem~\ref{Theorem:ID for f2}, we choose two vectors $v_0,\,v_0'$ with distinct magnitudes such that $|v_0|,\,|v_0'|$ sufficiently large. We obtain
\begin{align*}
  \begin{pmatrix}
  I(t,x,v_0)\\[1em]
  I(t,x,v_0')
   \end{pmatrix}
      &= \begin{pmatrix}
      \Lambda_\gamma(t,x,v_0)\LC ce^{- |v_0|^2\over 4} +\mathcal{O}(\tilde\varepsilon)\RC & \Lambda_\gamma(t,x,v_0)\LC ce^{- |v_0|^2\over 4} \LC {|v_0|^2-3\over 3}\RC+\mathcal{O}(\tilde\varepsilon)\RC \\[1em]
       \Lambda_\gamma(t,x,v_0')\LC ce^{-  |v_0'|^2\over 4}   +\mathcal{O}(\tilde\varepsilon)\RC & \Lambda_\gamma(t,x,v_0')\LC ce^{- |v_0'|^2\over 4}\LC { |v_0'|^2-3\over 3}\RC+\mathcal{O}(\tilde\varepsilon)\RC 
    \end{pmatrix} \cdot \begin{pmatrix}
     \alpha\\[1em]
     \beta
    \end{pmatrix},
\end{align*} 
where $c=(2\pi)^{-3/4}$.
Here the $2\times 2$ matrix is invertible provided that $\tilde{\varepsilon}$ is sufficiently small. This suggests the unique determination of $(\alpha,\beta)$.
\end{proof}


\newpage
\appendix
\section{Estimates}\label{sec:appendix} 
Let 
$
    w(v) = (1+ c_1|v|^{2})^{c_2},
$
for constants $c_1,\,c_2>0$, satisfy $w^{-2} \in L^1(\R^3)$. The following result states the boundedness of the kernel $K_w$ in the operator $\mathbf{P}_w$.
\begin{lemma}\label{lemma:K}
    There exist constants $\theta,\, \delta>0$ satisfying $0\leq \theta<{1\over 4}$ and $-{1\over 4} +\theta+\delta < 0$ so that
    $$
    \LV k(v,v'){w(v)\over w(v')}\RV\leq Ce^{\LC-{1\over 4} +\theta+\delta\RC |v|^2}e^{\LC-{1\over 4} -\theta+\delta\RC |v'|^2},
    $$
    where constant $C>0$ depends on $c_1,\,c_2$, $\theta$, and $\delta$.

    Moreover, 
    $$
    \int_{\R^3} \LV k(v,v'){w(v)\over w(v')}\RV\,dv'\leq C e^{\LC-{1\over 4} +\theta+\delta\RC |v|^2}. 
    $$
\end{lemma}
\begin{proof}
Following a similar argument in \cite[page 727]{Guo2010}, there exists a constant $C>0$ depending on $c_1,\,c_2$ such that
$$
    \LV{w(v)\over w(v')}\RV\leq  C\LC 1+|v-v'|^2\RC^{c_2} e^{-\theta\LC|v'|^2-|v|^2\RC}.
$$
Combining it with the definition of $k(v,v')$,  
we obtain
    \begin{align*}
        \LV k(v,v'){w(v)\over w(v')} \RV 
        &\leq C\LC 1+|v-v'|^2\RC^{c_2}\bigg(1+ |v||v'|+\frac{(|v|^2+3)(|v'|^2+3)}{6}\bigg) e^{\LC-{1\over 4} +\theta \RC |v|^2}e^{\LC-{1\over 4} -\theta \RC |v'|^2}.
    \end{align*}
Since 
$\LC 1+|v-v'|^2\RC^{c_2}\bigg(1+ |v||v'|+\frac{(|v|^2+3)(|v'|^2+3)}{6}\bigg)$ are polynomial in terms of $v$ and $v'$, they can be bounded by $C e^{ \delta |v|^2 +\delta|v'|^2}$ for some $\delta >0$. Thus we complete the proof.
\end{proof}

\subsection{Estimates of $\Gamma$}\label{sec:estimates of gamma}
We will present several useful estimates for functions defined in Section~\ref{sec:preliminaries}, and also for the nonlinear terms $\Gamma$ in the linearized BGK equation.
Similar arguments can be found in \cite{CKP2026}.
\begin{lemma}\label{prop:estimate integral Pi Qij}
Suppose that $ \sup_{0\leq t\leq t^*}\|wf(t)\|_{L^\infty_{x,v}}\lesssim \varepsilon$ for some $0<\varepsilon<1$. Then 
$$
    \|\rho -1\|_{L^\infty_{t,x}} + \|U \|_{L^\infty_{t,x}}+\|T -1\|_{L^\infty_{t,x}}\lesssim\varepsilon.
$$
Moreover, for almost every $(t,x,v)\in (0,t^*)\times\Omega\times\R^3$, we have
    \begin{align}\label{EST:int P_alpha_beta}
     \LV\int^1_0 \mathcal{P}_i (f)d\theta\RV\lesssim 1,
\end{align}    
and 
\begin{align}\label{EST:int Q_ij}
     w(v)\sqrt{\mu}^{-1}\LV\int^1_0  \mathcal{Q}_{ij}(f)(1-\theta)\,d\theta \RV\lesssim 1.
\end{align} 
\end{lemma}
 
\begin{proof}
Following the computations in Lemma~\ref{lemma:decomposition}, we can derive 
    \begin{align*}
        |\rho (t,x) - 1|=  \LV \int_{\R^3} \sqrt{\mu}f\,dv \RV\leq \|wf(t)\|_{L^\infty_{x,v}} \int_{\R^3}  \sqrt{\mu} w^{-1}\,dv\leq C\varepsilon,
    \end{align*}
    and thus, for sufficiently small $\varepsilon$, we obtain
    $$
       0<1-C\varepsilon \leq  |\rho (t,x)|\leq 1+C\varepsilon,
    $$
    for some constant $C>0$ depending on $w$.
Also, 
\begin{align*}
    |\rho U(t,x)| =  \LV\int_{\R^3} v\sqrt{\mu}f\,dv\RV \leq  \|wf(t)\|_{L^\infty_{x,v}}   \int_{\R^3} |v|\sqrt{\mu}w^{-1}\,dv\leq C\varepsilon,
\end{align*}
which yields
$$
    |U (t,x)| \leq {C\varepsilon\over 1-C\varepsilon}\lesssim \varepsilon.
$$
Next, from the proof of Lemma~\ref{lemma:decomposition},
\begin{align*}
		|3\rho T (t,x) - 3| 
        &\leq  \LV\int_{\R^3}  |v|^2 \sqrt{\mu}f\,dv\RV +|\rho ||U |^2\\
        &\lesssim  \|wf(t)\|_{L^\infty_{x,v}}  \LV\int_{\R^3} |v|^2 \sqrt{\mu}w^{-1}\,dv\RV 
        +\varepsilon^2\lesssim \varepsilon,
	\end{align*}
from which, we deduce that 
\begin{align*}
    |T (t,x) -1|&= \LV{\rho T -1 \over \rho } - {\rho  -1\over \rho }\RV  \lesssim {\varepsilon\over 1-C\varepsilon} + {\varepsilon \over 1-C\varepsilon}\lesssim \varepsilon 
\end{align*}
and 
$$
    0< 1-C\varepsilon \leq  |T (t,x)|\leq 1+C\varepsilon.
$$
Now we will apply the above estimates to show \eqref{EST:int P_alpha_beta}. Note that they also hold for $\rho_\theta,\,U_\theta,\,T_\theta$ for $\theta\in[0,1]$ defined in \eqref{DEF:macro epsilon}. We first estimate $\mathcal{P}_i$'s denominator. From the definition of $R_{i}(\rho_{ \theta },T_{ \theta })$, we get
$$
      1\lesssim a_\sigma (1-C\varepsilon)^{\sigma_1}(1-C\varepsilon)^{\sigma_2} \lesssim R_{i}(\rho_{ \theta },T_{ \theta }).
$$
This implies
$$
    \LV\int^1_0 \mathcal{P}_i d\theta\RV\lesssim  \sum_{ \kappa \in S_{i} }b_\kappa    (1+C\varepsilon)^{\kappa_1}\varepsilon^{\kappa_2+\kappa_3+\kappa_4}(1+C\varepsilon)^{\kappa_5} \lesssim 1.
$$

Finally, to estimate \eqref{EST:int Q_ij}, we first note that $\mathcal{Q}_{ij}$'s denominator is also bounded from below by a positive constant. Hence, for $\kappa=(\kappa_1,\ldots,\kappa_8)$, it is sufficient to give an upper bound for large $|v|$:
\begin{align*}
        |\mathcal{Q}_{ij}|&\lesssim |P_{ij}(\rho_{ \theta },v-U_{ \theta},U_{ \theta }, T_{ \theta })|M(F_{ \theta })= \sum_{\kappa\in S_{ij} }b_\kappa [ \rho_{ \theta } (v-U_{ \theta }) U_{ \theta }  T_{ \theta})]^\kappa M(F_{ \theta })\\
        &\lesssim \sum_{ \kappa \in S_{ij} }b_\kappa [(1+C\varepsilon)^{\kappa_1} |v-U_{\theta  }|^{\kappa_2+\kappa_3+\kappa_4}  \varepsilon^{\kappa_5+\kappa_6+\kappa_7} (1+C\varepsilon)^{\kappa_8}]  {1+C\varepsilon \over (2\pi(1-C\varepsilon))^{3/2}}e^{-{|v-U_{\theta }|^2\over 2(1+C\varepsilon)}}\\
        &\lesssim \sum_{\kappa \in S_{ij} } |v-U_{\theta }|^{\kappa_2+\kappa_3+\kappa_4} e^{-{|v-U_{\theta }|^2\over 2(1+C\varepsilon)}}\\
        &\lesssim  e^{-|v-U_{\theta }|^2\LC {1\over 2(1+C\varepsilon)}-c\RC}\leq e^{-(|v|^2- \varepsilon |v|)\LC {1\over 2(1+C\varepsilon)}-c\RC}\leq e^{- |v|^2 c'} \quad \hbox{ for }|v|\gg 1,
    \end{align*}
    where we used $|v-U_{\theta }|^{\kappa_2+\kappa_3+\kappa_4} \lesssim e^{c|v-U_{\theta }|^2}$ for some small constant $c>0$. Here we denote $c'= {1\over 2(1+C\varepsilon)}-c$ and it satisfies $c'\in (1/4,1/2)$ when $\varepsilon$ is small enough. For large $|v|$, we apply the estimate of $\mathcal{Q}_{ij}$ to bound
$$
 w(v)\sqrt{\mu}^{-1}(v)\LV\int^1_0  \mathcal{Q}_{ij}(1-\theta)\,d\theta\RV \lesssim   w(v)e^{|v|^2\over 4}e^{- |v|^2 c'} \lesssim 1.
$$
The completes the proof.
\end{proof}

\begin{lemma}\label{lemma:second derivative estimates}
We denote the macroscopic quantities $\rho_{j }$, $U_{j }$, and $T_{j }$, which are corresponding to the solution $F_{j}=\mu+ \sqrt{\mu}f_j$ to the BGK equation with collision frequency $q_j$ for $j=1,\,2$.
Suppose that $ \sup_{0\leq t\leq t^*}\|wf_j(t)\|_{L^\infty_{x,v}}\lesssim \varepsilon$ for some $0<\varepsilon<1$. Then
    \begin{align}\label{EST:int q diff}
        \int^1_0 |\mathcal{P}_i (f_1)-\mathcal{P}_i (f_2)| d\theta \lesssim  \|w(f_1-f_2)\|_{L^\infty_{x,v}},
    \end{align}
    and
    \begin{align}\label{EST:int Q_ij diff}
     w(v)\sqrt{\mu}^{-1}\int^1_0  |\mathcal{Q}_{ij}(f_1)-\mathcal{Q}_{ij}(f_2)|(1-\theta)\,d\theta \lesssim  \|w(f_1-f_2)\|_{L^\infty_{x,v}}.
    \end{align}
\end{lemma}    

\begin{proof} 
From the definitions of $\mathcal{P}_{j}$ and $\mathcal{Q}_{ij}$, we know their denominators are bounded from below by a positive constant due to 
$$
    \|\rho_{j,\theta } -1\|_{L^\infty_{t,x}} + \|U_{j,\theta } \|_{L^\infty_{t,x}}+\|T_{j,\theta } -1\|_{L^\infty_{t,x}}\lesssim\varepsilon,
$$
and thus it suffices to control their numerators. To this end, we will prove the following estimates:
\begin{align*}
    |\rho_{1,\theta}-\rho_{2,\theta }| +|U_{1,\theta}-U_{2,\theta}| +|T_{1,\theta}-T_{2,\theta}|\lesssim  \|w(f_1-f_2)\|_{L^\infty_{x,v}},
\end{align*}
which will yield then the estimates \eqref{EST:int q diff} and \eqref{EST:int Q_ij diff} after direct computations.

Now recalling the definition \eqref{DEF:macro epsilon}, we have
\begin{align*}
    |\rho_{1,\theta}-\rho_{2,\theta }| = \LV \int_{\R^3} F_{1,\theta} -F_{2,\theta}\,dv\RV = \LV \theta\int_{\R^3}   \sqrt{\mu}(f_1-f_2)\,dv \RV\lesssim  \|w(f_1-f_2)\|_{L^\infty_{x,v}},
\end{align*}
\begin{align*}
    |U_{1,\theta}-U_{2,\theta }| 
    &= \LV  \int_{\R^3} v(\rho_{1,\theta}^{-1} F_{1,\theta} -\rho_{2,\theta}^{-1} F_{2,\theta})\,dv\RV =  \LV  \rho_{1,\theta }^{-1}\int_{\R^3} v (F_{1,\theta} -  F_{2,\theta})\,dv\RV  + \LV \int_{\R^3} v F_{2,\theta} \,dv\RV\LV\rho_{1,\theta }^{-1}-\rho_{2,\theta }^{-1}\RV\\
    &\lesssim  \|w(f_1-f_2)\|_{L^\infty_{x,v}}+ \|w f_2 \|_{L^\infty_{x,v}} |\rho_{1,\theta}-\rho_{2,\theta}| (\rho_{1,\theta}\rho_{2,\theta})^{-1}\lesssim  \|w(f_1-f_2)\|_{L^\infty_{x,v}},
\end{align*}
and similar arguments yield that
\begin{align*}
    |T_{1,\theta}-T_{2,\theta }| 
    &= {1\over 3}\LV  \int_{\R^3} \rho_{1,\theta}^{-1}F_{1,\theta}|v-U_{1,\theta}|^2 -\rho_{2,\theta}^{-1}F_{2,\theta}|v-U_{2,\theta}|^2\,dv\RV \lesssim  \|w(f_1-f_2)\|_{L^\infty_{x,v}}.
\end{align*}
Finally, the results of this lemma follow by a straightforward and laborious computations. We refer the interested readers to \cite[proof of Lemma 6]{CKP2026} for more details.
\end{proof}

\begin{lemma}\label{lemma:estimate of Gamma}
        There hold
        \begin{align}\label{EST:f e_i}
        \LV\langle f,e_i\rangle \RV \lesssim \|wf\|_{L^\infty_{x,v}},\quad i= 1,\,\ldots,\,5.
        \end{align}
        Moreover, 
        $$       \|w\Gamma_j(f)\|_{L^\infty_{x,v}}\lesssim \|wf\|^2_{L^\infty_{x,v}}\quad\hbox{ for }j=1,\,2;
        \quad       \|w\Gamma_3(f)\|_{L^\infty_{x,v}}\lesssim \|wf\|^3_{L^\infty_{x,v}}.
        $$
\end{lemma}
\begin{proof}
Based on the definition of $e_i$, it follows
    \begin{align*}
        \LV\langle f,e_1\rangle \RV = \LV\int_{\R^3} fw(v)\sqrt{\mu}w^{-1}(v)dv\RV\leq \|wf\|_{L^\infty_{x,v}}\LV\int_{\R^3} \sqrt{\mu}w^{-1}(v)dv\RV\lesssim \|wf\|_{L^\infty_{x,v}},
    \end{align*}
    \begin{align*}
        \LV\langle f,e_\ell\rangle \RV = \LV\int_{\R^3} fw(v)v_\ell\sqrt{\mu}w^{-1}(v)dv\RV\leq \|wf\|_{L^\infty_{x,v}}\LV\int_{\R^3} v_\ell\sqrt{\mu}w^{-1}(v)dv\RV\lesssim \|wf\|_{L^\infty_{x,v}}, \quad \ell=2,\,3,\,4,
    \end{align*}
    and 
    \begin{align*}
        \LV\langle f,e_5\rangle \RV = {1\over \sqrt{6}} \LV\int_{\R^3} fw(v)(|v|^2-3)\sqrt{\mu}w^{-1}(v)dv\RV\leq \|wf\|_{L^\infty_{x,v}}\LV\int_{\R^3} (|v|^2-3)\sqrt{\mu}w^{-1}(v)dv\RV\lesssim \|wf\|_{L^\infty_{x,v}}.
    \end{align*}
    Also, by Lemma~\ref{lemma:K}, we have 
    \begin{align}\label{EST:Pf-f}
        |w(\mathbf{P}f-f)|&=\LV w(v) \int_{\R^3}  k(v,v')   f(v')dv' -wf\RV=\LV\int_{\R^3}  k(v,v'){w(v)\over w(v')} w(v')f(v')dv' -wf\RV \notag\\
        &\leq \|wf\|_{L^\infty_{x,v}} \LV\int_{\R^3}  k(v,v'){w(v)\over w(v')}  dv'+1\RV\lesssim \|wf\|_{L^\infty_{x,v}} .
    \end{align}
    To estimate $\Gamma_1$, with the above estimates, we apply Lemma~\ref{prop:estimate integral Pi Qij} to get
    \begin{align*}
        |w(v) \Gamma_1(f)| &\leq  \LV w(v) (\mathbf{P}f-f)\RV \LV  \sum^{5}_{\ell=1} \LC\int^1_0 \mathcal{P}_\ell d\theta\RC\langle f,e_\ell\rangle \RV\lesssim  \|wf\|^2_{L^\infty_{x,v}}.
    \end{align*}
    For $\Gamma_2$, using \eqref{EST:int Q_ij}, similar arguments give
    \begin{align*}
        |w(v) \Gamma_2(f)| &\leq |\gamma| \LV w \sqrt{\mu}^{-1} \sum_{i,j=1}^5\LC \int^1_0  \mathcal{Q}_{ij}(f)(1-\theta)\,d\theta\RC \langle f,e_i\rangle\langle f,e_j\rangle\RV \lesssim \|wf\|^2_{L^\infty_{x,v}}.
    \end{align*}
    The estimate for $w\Gamma_3(f)$ can be proved similarly and thus we omit it here. 
\end{proof}

\begin{lemma}\label{lemma:diff Gamma}
Suppose that $\sup_{0\leq t\leq t^*}(\|wf_1(t)\|_{L^\infty_{x,v}}+\|wf_2(t)\|_{L^\infty_{x,v}})\lesssim \varepsilon$ for some $0<\varepsilon<1$.
    Then 
    $$
    \|w(\Gamma (f_1)-\Gamma(f_2))\|_{L^\infty_{x,v}}\leq \varepsilon \|w(f_1-f_2)\|_{L^\infty_{x,v}}.
    $$
\end{lemma}
\begin{proof}
    Recall the definition of $\Gamma(f)= \Gamma_1(f)+  \Gamma_2(f)+ \Gamma_3(f)$ in Proposition~\ref{Prop:linearization of M(F)}.
    We start with the estimate of the operator $\Gamma_1$ by first analyzing
    \begin{align*}
       \LV \Upsilon_1(f_1)- \Upsilon_1(f_2)\RV 
       &=  \LV \sum^{5}_{\ell=1} \LC\int^1_0 \mathcal{P}_\ell (f_1)-\mathcal{P}_\ell (f_2) d\theta\RC\langle f_1,e_\ell\rangle  +  \sum^{5}_{\ell=1} \LC\int^1_0 \mathcal{P}_\ell (f_2) d\theta\RC\langle f_1-f_2,e_\ell\rangle \RV\\
       &\lesssim \|w(f_1-f_2)\|_{L^\infty_{x,v}}  \|wf_1\|_{L^\infty_{x,v}} +\|w(f_1-f_2)\|_{L^\infty_{x,v}} \lesssim (\varepsilon +1) \|w(f_1-f_2)\|_{L^\infty_{x,v}},
    \end{align*}
    where we used Lemma~\ref{prop:estimate integral Pi Qij}-Lemma~
    \ref{lemma:estimate of Gamma}.
    Next, applying \eqref{EST:f e_i} and \eqref{EST:Pf-f} gives
    \begin{align*}
         |w(\Gamma_1(f_1) - \Gamma_1(f_2))| 
         &\leq |w\LC(\mathbf{P}(f_1-f_2)-(f_1-f_2)\RC|| \Upsilon_1(f_1)| + |w(\mathbf{P}f_2-f_2)|\LV \Upsilon_1(f_1)- \Upsilon_1(f_2)\RV\\
         &\leq  \|w(f_1-f_2)\|_{L^\infty_{x,v}} \|wf_1\|_{L^\infty_{x,v}}+\|w f_2 \|_{L^\infty_{x,v}} \|w(f_1-f_2)\|_{L^\infty_{x,v}}\lesssim \varepsilon \|w(f_1-f_2)\|_{L^\infty_{x,v}}.
    \end{align*}

    For $\Gamma_2$, we proceed similarly as follows. By Lemma~\ref{prop:estimate integral Pi Qij} and Lemma~\ref{lemma:second derivative estimates}, we derive
    \begin{align*}
         |w(\Gamma_2(f_1) - \Gamma_2(f_2))| 
         &=  |\gamma w\sqrt{\mu}^{-1} (\Upsilon_2(f_1)-\Upsilon_2(f_2))|\\
         &\lesssim  \sum^{5}_{i,j=1}\LV w\sqrt{\mu}^{-1} \int^1_0 (\mathcal{Q}_{ij}(f_1)-\mathcal{Q}_{ij}(f_2))(1-\theta)\,d\theta \RV |\langle f_1,e_i\rangle\langle f_1,e_j\rangle| \\
         &\quad +\sum^{5}_{i,j=1}\LV w \sqrt{\mu}^{-1}\int^1_0 \mathcal{Q}_{ij}(f_2) (1-\theta)\,d\theta \RV\LV \langle f_1,e_i\rangle\langle f_1,e_j\rangle-\langle f_2,e_i\rangle\langle f_2,e_j\rangle\RV\\
         &\leq \|w(f_1-f_2)\|_{L^\infty_{x,v}} \|wf_1\|_{L^\infty_{x,v}}^2 + \|w(f_1-f_2)\|_{L^\infty_{x,v}} (\|wf_1\|_{L^\infty_{x,v}}+\|wf_2\|_{L^\infty_{x,v}})\\
         &\lesssim \varepsilon \|w(f_1-f_2)\|_{L^\infty_{x,v}}.
    \end{align*}
    The derivation of the estimate for $\Gamma_3$ is similar. Recall $\Upsilon_3(f)=\Upsilon_1(f)\Upsilon_2(f)$. Under the smallness hypothesis, we get
    \begin{align*}
        |w(\Gamma_3(f_1) - \Gamma_3(f_2))| 
        &= | w\sqrt{\mu}^{-1}(\Upsilon_3(f_1)-\Upsilon_3(f_2))|\\
        &\lesssim |w\sqrt{\mu}^{-1}(\Upsilon_1(f_1)-\Upsilon_1(f_2)) \Upsilon_2(f_1)|+|w\sqrt{\mu}^{-1}(\Upsilon_2(f_1)-\Upsilon_2(f_2)) \Upsilon_1(f_2)|\\
        &\lesssim  \|w(f_1-f_2)\|_{L^\infty_{x,v}}\|wf_1\|_{L^\infty_{x,v}}^2 + \varepsilon \|w(f_1-f_2)\|_{L^\infty_{x,v}}\|wf_2\|_{L^\infty_{x,v}},
    \end{align*}
    which yields the desired estimate.
\end{proof}

\section{The first-order and second-order linearizations}\label{sec:appendix B}
In this section, we will compute the derivatives of $M(F_\varepsilon)$ with respect  to $\varepsilon$ explicitly provided that such derivatives exist, which are proved in Theorem~\ref{theorem:linearization} utilizing the asymptotic expansion.

Let $F_\varepsilon=\mu+\sqrt{\mu} f_\varepsilon$  be the solution to the BGK equation with $q(t,x,F)=\gamma (x)\rho^{\alpha}T^{\beta}$:
\begin{align*}
       \p_t F+v\cdot\nabla_x F = q(t,x,F)(M(F) -F),\quad\hbox{with } F|_{t=0}=\mu+\varepsilon\sqrt{\mu}f_{in},\quad F |_{S_-} =\mu+\varepsilon\sqrt{\mu}f_-.
\end{align*}

\subsection{First-order linearization}\label{first-order linearization}
By direct computations, we have
\begin{lemma}\label{first order 2}  
    \begin{align*}
        \p_\varepsilon \rho_\epsilon& = \int_{\R^3}  \p_\varepsilon F_\varepsilon\,dv,\\
        \p_\varepsilon U_\epsilon & = -\rho_\varepsilon^{-2}\p_\varepsilon \rho_\varepsilon \int v F_\varepsilon\,dv+\rho_\varepsilon^{-1} \int_{\R^3} v (\p_\varepsilon F_\varepsilon)\,dv, \\
        \p_\varepsilon T_\epsilon & = \frac{1}{3}\LC -\rho_\varepsilon^{-2}\p_\varepsilon \rho_\varepsilon \int_{\R^3} |v-U_\varepsilon|^2 F_\varepsilon\,dv+\rho_\varepsilon^{-1} \int_{\R^3} |v-U_\varepsilon|^2 (\p_\varepsilon F_\varepsilon)\,dv-2\rho_\varepsilon^{-1}\int_{\R^3} (v-U_\varepsilon)\cdot (\p_\varepsilon U_\varepsilon) F_\varepsilon\,dv\RC.
    \end{align*}
In particular,
 \begin{align*}
       \p_\varepsilon \rho_\epsilon|_{\varepsilon=0} & = \int_{\R^3} \sqrt{\mu} f^{(1)}\,dv,\quad 
       \p_\varepsilon U_\epsilon|_{\varepsilon=0} & = \int_{\R^3} v\sqrt{\mu} f^{(1)}\,dv,\quad
        \p_\varepsilon T_\epsilon|_{\varepsilon=0} & = \frac{1}{3}\int_{\R^3} \sqrt{\mu} (|v|^2-3) f^{(1)}\,dv.
    \end{align*}
\end{lemma}
\begin{proof}
    Recall that definitions of $\rho$, $U$ and $T$, by direct calculation, we obtain the first three equalities.
    Note that $\p_\varepsilon F_\varepsilon=\sqrt{\mu}\p_\varepsilon f_\varepsilon$. At $\varepsilon=0$, we have $\rho_0=T_0=1$, $U_0=0$, and $\p_\varepsilon F_\varepsilon|_{\varepsilon=0}=\sqrt{\mu} f^{(1)}$, and thus
    \begin{align*}
        \p_\varepsilon \rho_\epsilon|_{\varepsilon=0} 
        & = \int_{\R^3} \sqrt{\mu} f^{(1)}\,dv,\\
        \p_\varepsilon U_\epsilon|_{\varepsilon=0} 
        & = -\LC \int_{\R^3} \sqrt{\mu}f^{(1)}\,dv\RC\LC \int_{\R^3} v\mu(v)\,dv\RC+\int_{\R^3} \sqrt{\mu}v f^{(1)}\,dv ,\\
        \p_\varepsilon T_\epsilon|_{\varepsilon=0} 
        & = \frac{1}{3}\LC -\LC\int_{\R^3} \sqrt{\mu}f^{(1)}\,dv\RC\LC   \int_{\R^3} |v|^2\mu(v)\,dv\RC+ \int_{\R^3} \sqrt{\mu}|v|^2 f^{(1)}\,dv-2 (\p_\varepsilon U_\varepsilon|_{\varepsilon=0})\int_{\R^3} v \mu(v)\,dv \RC.
    \end{align*}
    The proof is complete by noting that $\mu$ is a normal distribution, $\int_{\R^3} \mu v\,dv=0$,  and $\int_{\R^3} \mu |v|^2\,dv=3$. 
\end{proof}

\begin{lemma}\label{lemma:first derivative of M}
    $$
    \left.\frac{\p M(F_\varepsilon)}{\p \varepsilon} \right|_{\varepsilon=0} = \sqrt{\mu(v)}\int_{\R^3} k(v,v') f^{(1)}(t,x,
    v')\,dv'=\sqrt{\mu} \,\mathbf{P} f^{(1)},
    $$
    where the kernel is defined as
    $$
    k(v,v') :=\sqrt{\mu(v)}\sqrt{\mu(v')}\bigg(1+v\cdot v'+\frac{(|v|^2-3)(|v'|^2-3)}{6}\bigg).
    $$
\end{lemma}
\begin{proof}
By applying Lemma~\ref{lemma:first order 1} and Lemma~\ref{first order 2} above, and the chain rule, we derive
    \begin{align*}
    \left.\frac{\p M(F_\varepsilon)}{\p \varepsilon} \right|_{\varepsilon=0} & 
    =\frac{\p M(F_\varepsilon)}{\p \rho_\varepsilon}\p_\varepsilon \rho_\varepsilon |_{\varepsilon=0}+\frac{\p M(F_\varepsilon)}{\p U_\varepsilon}\p_\varepsilon U_\varepsilon |_{\varepsilon=0}+\frac{\p M(F_\varepsilon)}{\p T_\varepsilon} \p_\varepsilon T_\varepsilon |_{\varepsilon=0}\\
    & =\mu(v)\int_{\R^3} \sqrt{\mu(v')} f^{(1)}(v')\,dv'+\mu(v) v\cdot \int_{\R^3} \sqrt{\mu(v')}v' f^{(1)}(v')\,dv'\\
     & \quad +\mu(v)\LC\frac{|v|^2-3}{2}\RC\int_{\R^3} \sqrt{\mu(v')}\LC\frac{|v'|^2-3}{3}\RC f^{(1)}(v')\,dv'\\
     &= \sqrt{\mu(v)}\int_{\R^3} \sqrt{\mu(v)}\sqrt{\mu(v')}\bigg(1+v\cdot v'+\frac{(|v|^2-3)(|v'|^2-3)}{6}\bigg) f^{(1)}(v')\,dv'.
\end{align*}
\end{proof}

\subsection{Second-order linearization}
To derive the second-order linearization of the BGK equation, we need the following lemma.

\begin{lemma}\label{second order 1}
 \begin{align*}
       \p^2_\varepsilon \rho_\epsilon|_{\varepsilon=0} & = \int_{\R^3}  \sqrt{\mu} f^{(2)}\,dv,\\
       \p^2_\varepsilon U_\epsilon|_{\varepsilon=0} & = -2\LC\int_{\R^3}  \sqrt{\mu} f^{(1)}\,dv\RC\LC\int_{\R^3} \sqrt{\mu}v f^{(1)}\,dv\RC+\int_{\R^3}  \sqrt{\mu}v f^{(2)}\,dv,\\
        \p^2_\varepsilon T_\epsilon|_{\varepsilon=0} & =\int_{\R^3}  \sqrt{\mu}\LC\frac{|v|^2-3}{3}\RC f^{(2)}\,dv +\sqrt{\mu}A(v,f^{(1)}),
    \end{align*}
    where $A(v,f^{(1)})$ is a function of $f^{(1)}$.
\end{lemma}
\begin{proof}
By Lemma \ref{first order 2}, we obtain
$$\p_\varepsilon^2 \rho_\varepsilon|_{\varepsilon=0}=\int_{\R^3}  \p_\varepsilon^2 F_\varepsilon |_{\varepsilon=0}\,dv=\int_{\R^3}  \sqrt{\mu} f^{(2)}\,dv,$$ 
    \begin{align*}
        \p^2_\varepsilon U_\varepsilon  
        & =-\p_\varepsilon(\rho_\varepsilon^{-2}\p_\varepsilon \rho_\varepsilon)\LC\int_{\R^3}  vF_\varepsilon\,dv\RC-2\rho^{-2}_\varepsilon \p_\varepsilon\rho_\varepsilon \LC\int_{\R^3}  v (\p_\varepsilon F_\varepsilon)\,dv\RC+\rho^{-1}_\varepsilon\LC\int_{\R^3} v (\p^2_\varepsilon F_\varepsilon)\,dv\RC.
    \end{align*}
Since $\int_{\R^3}  v F_0\,dv=\int_{\R^3}  v \mu(v)\,dv=0$, it implies
\begin{align*}
    \p_\varepsilon^2 U_\varepsilon|_{\varepsilon=0}=-2 \LC\int_{\R^3}  \sqrt{\mu} f^{(1)}\,dv\RC\LC\int_{\R^3}  \sqrt{\mu}v f^{(1)}\,dv\RC+\int_{\R^3}  \sqrt{\mu}v f^{(2)}\,dv.
\end{align*}
Similarly, we can derive
    \begin{align*}
        \p^2_\varepsilon T_\varepsilon 
        & =\frac{1}{3} \Bigg( -\underbrace{(\p^2_\varepsilon\rho_\varepsilon )\rho_\varepsilon^{-2}\LC \int_{\R^3} |v-U_\varepsilon|^2 F_\varepsilon\,dv\RC}_{\varepsilon=0:\quad  3\int \sqrt{\mu} f^{(2)}\,dv}-(\p_\varepsilon\rho_\varepsilon )\p_\varepsilon\LC \rho_\varepsilon^{-2}\int_{\R^3} |v-U_\varepsilon|^2 F_\varepsilon\,dv\RC\\
        &\quad -\rho_\varepsilon^{-2}\p_\varepsilon\rho_\varepsilon\LC \int_{\R^3} |v-U_\varepsilon|^2 (\p_\varepsilon F_\varepsilon)\,dv\RC  +\rho_\varepsilon^{-1}\LC \int_{\R^3} (\p_\varepsilon |v-U_\varepsilon|^2) (\p_\varepsilon F_\varepsilon)\,dv\RC\\
        & \quad +\underbrace{\rho_\varepsilon^{-1}\LC \int_{\R^3}  |v-U_\varepsilon|^2 (\p_\varepsilon^2 F_\varepsilon)\,dv\RC}_{\varepsilon=0:\quad  \int \sqrt{\mu} |v|^2 f^{(2)}\,dv } +2\rho_\varepsilon^{-2}\p_\varepsilon\rho_\varepsilon \LC\int_{\R^3}  (v-U_\varepsilon)\cdot (\p_\varepsilon U_\varepsilon) F_\varepsilon\,dv\RC\\
        & \quad -2\rho_\varepsilon^{-1} \LC\int_{\R^3}  \p_\varepsilon \LC(v-U_\varepsilon)F_\varepsilon\RC\cdot (\p_\varepsilon U_\varepsilon)\,dv\RC-2\rho_\varepsilon^{-1} (\p^2_\varepsilon U_\varepsilon) \cdot\LC\int_{\R^3}  (v-U_\varepsilon)  F_\varepsilon\,dv \RC\Bigg).
    \end{align*}
Since $\int |v-U_0|^2 F_0\,dv=\int |v|^2\mu(v)\,dv=3$,
\begin{align*}
    \p_\varepsilon^2 T_\varepsilon|_{\varepsilon=0}=\frac{1}{3}\LC -3\int_{\R^3} \sqrt{\mu} f^{(2)}\,dv+\int_{\R^3} \sqrt{\mu} |v|^2 f^{(2)}\,dv\RC +\sqrt{\mu}A(v,f^{(1)}),
\end{align*}
where $A(v,f^{(1)})$ is a function of $f^{(1)}$, independent of $f^{(2)}$. 
\end{proof}

\begin{lemma}\label{lemma:second linear MF}
   \begin{align*}
     \left.\frac{\p^2 M(F_\varepsilon)}{\p \varepsilon^2}\right|_{\varepsilon=0}=\sqrt{\mu}\int_{\R^3} k(v,v')f^{(2)}(v')\,dv'+\sqrt{\mu}B(v,f^{(1)})=\sqrt{\mu}\, \mathbf{P} f^{(2)}+\sqrt{\mu}B(v,f^{(1)}),
\end{align*}
where $B(v,f^{(1)})$ is a function of $f^{(1)}$.
\end{lemma}
\begin{proof}
    \begin{align*}
    \frac{\p^2 M(F_\varepsilon)}{\p \varepsilon^2} &=\p_\varepsilon \LC \frac{\p M(F_\varepsilon)}{\p \rho_\varepsilon}\p_\varepsilon \rho_\varepsilon+\frac{\p M(F_\varepsilon)}{\p U_\varepsilon}\p_\varepsilon U_\varepsilon +\frac{\p M(F_\varepsilon)}{\p T_\varepsilon} \p_\varepsilon T_\varepsilon  \RC\\
    & = \LC \p_\varepsilon \frac{\p M(F_\varepsilon)}{\p \rho_\varepsilon}(\p_\varepsilon \rho_\varepsilon) +\p_\varepsilon\frac{\p M(F_\varepsilon)}{\p U_\varepsilon}(\p_\varepsilon U_\varepsilon )+\p_\varepsilon\frac{\p M(F_\varepsilon)}{\p T_\varepsilon} (\p_\varepsilon T_\varepsilon) \RC \\
    & \quad +\LC \frac{\p M(F_\varepsilon)}{\p \rho_\varepsilon}\p^2_\varepsilon \rho_\varepsilon+\frac{\p M(F_\varepsilon)}{\p U_\varepsilon}\p^2_\varepsilon U_\varepsilon +\frac{\p M(F_\varepsilon)}{\p T_\varepsilon} \p^2_\varepsilon T_\varepsilon \RC.
\end{align*}
By Lemma \ref{first order 2} and Lemma~\ref{second order 1}, $\p^2_\varepsilon|_{\varepsilon=0} \rho_\varepsilon$, $\p^2_\varepsilon|_{\varepsilon=0}  U_\varepsilon$, $\p^2_\varepsilon|_{\varepsilon=0} T_\varepsilon$ give the first three terms containing $f^{(2)}$ on the RHS of the following identity:
\begin{align*}
     \left.\frac{\p^2 M(F_\varepsilon)}{\p \varepsilon^2}\right|_{\varepsilon=0}& =\mu(v)\LC \int_{\R^3} \sqrt{\mu(v')} f^{(2)}(v')\,dv'\RC +v\mu(v)\LC\int_{\R^3} \sqrt{\mu(v')} v' f^{(2)}(v')\,dv'\RC\\
     & \quad +\LC \frac{|v|^2-3}{2}\RC\mu(v)\LC\int_{\R^3} \sqrt{\mu(v')}\LC\frac{|v'|^2-3}{3}\RC f^{(2)}(v')\,dv'\RC+\sqrt{\mu}B(v,f^{(1)})\\
     &=\sqrt{\mu}\LC\int_{\R^3} k(v,v') f^{(2)}(v')\,dv'\RC+\sqrt{\mu} B(v,f^{(1)}).
\end{align*}
Here again $B(v,f^{(1)})$ is a function of $f^{(1)}$, independent of $f^{(2)}$.
\end{proof}

We conclude this section by linearizing the BGK equation: 
Let $F_\varepsilon=\mu+\sqrt{\mu} f_\varepsilon$  be the solution to 
\begin{align*}
       \p_t F+v\cdot\nabla_x F = q(t,x,F)(M(F) -F),\quad\hbox{with } F|_{t=0}=\mu+\varepsilon\sqrt{\mu}f_{in},\quad F |_{S_-} =\mu+\varepsilon\sqrt{\mu}f_-,
\end{align*}
where $q(t,x,F)=\gamma (x)\rho^{\alpha}T^{\beta}.$ We first linearize the nonlinear BGK collision operator.
\begin{proposition}\label{prop:appendix second linearization of BGK} We denote $\rho^{(1)}=\p_\varepsilon \rho_\varepsilon |_{\varepsilon=0}$ and $T^{(1)}=\p_\varepsilon T_\varepsilon |_{\varepsilon=0}$. Then
\begin{align*}
\p_{\varepsilon}|_{\varepsilon=0} (q(t,x,F_\varepsilon)(M(F_\varepsilon) -F_\varepsilon)) 
=\sqrt{\mu}\gamma (\mathbf{P}-\mathbf{I}) f^{(1)},
\end{align*}
and
\begin{align*}
&\p^2_{\varepsilon}|_{\varepsilon=0} (q(t,x,F_\varepsilon)(M(F_\varepsilon) -F_\varepsilon)) 
\\
&=\sqrt{\mu}\gamma (\mathbf{P}-\mathbf{I}) f^{(2)} +2 \sqrt{\mu} \gamma(x)(  \alpha \rho^{(1)}+ \beta T^{(1)})  (\mathbf{P}-\mathbf{I}) f^{(1)}+ \gamma\sqrt{\mu}B(v,f^{(1)}) .
\end{align*}
\end{proposition}
\begin{proof}
Note that 
\begin{align*}
    \p_\varepsilon q(t,x,F_\varepsilon)=\p_\varepsilon (\gamma  \rho_\varepsilon^{\alpha }T_\varepsilon^{\beta } ) 
    = \gamma \bigg(\alpha \rho_\varepsilon^{\alpha-1}(\p_\varepsilon\rho_\varepsilon)T_\varepsilon^{\beta }+\beta \rho_\varepsilon^{\alpha }T_\varepsilon^{\beta -1}(\p_\varepsilon T_\varepsilon)\bigg),
\end{align*}
and
$$
    \p_{\varepsilon} (M(F_\varepsilon) -F_\varepsilon) = \p_{\varepsilon} M(F_\varepsilon)-\sqrt{\mu} \p_{\varepsilon}f_\varepsilon.  
$$
It follows that, by Lemma \ref{first order 2}, Lemma~\ref{lemma:first derivative of M}, and $\rho_0=T_0=1$, we obtain
$$
    q(t,x,F_\varepsilon)|_{\varepsilon=0} =\gamma,\quad \hbox{and}\quad\p_\varepsilon q(t,x,F_\varepsilon)|_{\varepsilon=0}= \gamma(x)(  \alpha  \p_\varepsilon \rho_\varepsilon |_{\varepsilon=0}+ \beta \p_\varepsilon T_\varepsilon|_{\varepsilon=0}).
$$
and 
$$
   (M(F_\varepsilon) -F_\varepsilon)|_{\varepsilon=0} =M(\mu)-\mu=0,\quad\hbox{and}\quad \p_{\varepsilon} (M(F_\varepsilon) -F_\varepsilon)|_{\varepsilon=0} =\sqrt{\mu} (\mathbf{P}-\mathbf{I}) f^{(1)}.
$$
Hence, we obtain
\begin{align*}
\p_{\varepsilon}|_{\varepsilon=0} (q(t,x,F_\varepsilon)(M(F_\varepsilon) -F_\varepsilon)) 
=\sqrt{\mu}\gamma (\mathbf{P}-\mathbf{I}) f^{(1)}.
\end{align*}

Moreover, by Lemma~\ref{lemma:second linear MF} and $(M(F_\varepsilon) -F_\varepsilon)|_{\varepsilon=0}=0$, we have
\begin{align*}
&\p^2_{\varepsilon}|_{\varepsilon=0} (q(t,x,F_\varepsilon)(M(F_\varepsilon) -F_\varepsilon)) \\
&= \LC \p^2_{\varepsilon}q(t,x,F_\varepsilon)(M(F_\varepsilon) -F_\varepsilon))+ 2\p_{\varepsilon}q(t,x,F_\varepsilon)  \p_{\varepsilon}(M(F_\varepsilon) -F_\varepsilon)) +q(t,x,F_\varepsilon)\p^2_{\varepsilon}(M(F_\varepsilon) -F_\varepsilon))\RC|_{\varepsilon=0} \\
&= 2 \sqrt{\mu} \gamma(x)(  \alpha \rho^{(1)}+ \beta T^{(1)})  (\mathbf{P}-\mathbf{I}) f^{(1)} +
\sqrt{\mu}\gamma (\mathbf{P}-\mathbf{I}) f^{(2)} + \gamma\sqrt{\mu}B(v,f^{(1)}) .
\end{align*}
\end{proof}

With this Proposition, now we differentiate both sides of the BGK equation with respect to $\varepsilon$ and let $\varepsilon=0$ to get the first-order linearization:
\begin{align*}
       \{\p_t  +v\cdot\nabla_x +\gamma(x) (\mathbf{I} -\mathbf{P})\}f^{(1)} =0,\quad\hbox{with } f^{(1)}|_{t=0}=  f_{in},\quad f^{(1)}|_{S_-} = f_-,
\end{align*}
and also the second-order linearization:
\begin{align}\label{EQN:2 linearization appendix}
    \{\p_t  +v\cdot\nabla_x +\gamma(x) (\mathbf{I} -\mathbf{P})\}f^{(2)} 
    =2   \gamma(x)(  \alpha \rho^{(1)}+ \beta T^{(1)})  (\mathbf{P}-\mathbf{I}) f^{(1)}+ \gamma B(v,f^{(1)})  ,
\end{align}
with $f^{(2)}|_{t=0}= 0$ and $f^{(2)}|_{S_-} = 0$.

\section*{Acknowledgements}
The first author would like to thank Prof. Chanwoo Kim for helpful discussions during the workshop held at the SLMath Institute in Summer 2025, and the author greatly appreciates the hospitality of the SLMath Institute.

Ru-Yu Lai is partially supported by the National Science Foundation (NSF) through grants DMS-2306221. Hanming Zhou is partly supported by the NSF grant DMS-2408369 and Simons Foundation Travel Support for Mathematicians MPS-TSM-00008046.

\bibliographystyle{abbrv}

\bibliography{1Ntransport}

\end{document}